\documentclass[10pt]{article}

\usepackage[english]{babel}

\usepackage[letterpaper,top=2cm,bottom=2cm,left=2cm,right=2cm,marginparwidth=1.75cm]{geometry}

\usepackage[T1]{fontenc}
\usepackage{amsmath}
\usepackage{amsthm}
\usepackage{amsfonts}
\usepackage{graphicx}
\usepackage[colorlinks=true, allcolors=blue]{hyperref}
\newtheorem{theorem}{Theorem}[section]
\newtheorem{lemma}[theorem]{Lemma}
\newtheorem{sublemma}[theorem]{Sublemma}
\newtheorem{subsublemma}[theorem]{Subsublemma}
\newtheorem{remark}[theorem]{Remark}
\newtheorem{df}[theorem]{Definition}

\newtheorem{as}[theorem]{Assumption}

\title{Bernoullicity of some skew products with hyperbolic base and Kochergin flow in the fiber}
\author{Mateusz Nowak}
\date{}
\begin{document}
\maketitle
\begin{abstract}
    We study the Bernoulli property for skew products with hyperbolic diffeomorphisms equipped with a Gibbs measure in the base and Kochergin flows in the fiber, when the cocycle is aperiodic and of zero mean. The flow in the fiber can be represented as a special flow over an irrational rotation and a roof function with power singularity. We show that if the growth near the singularity is given by an exponent smaller than $\frac{1}{2}$, then for almost every rotation the resulting skew product is Bernoulli.
\end{abstract}

\tableofcontents

\section{Introduction}
Among the many notions of chaotic behavior in the theory of dynamical systems, one of the strongest is the Bernoulli property. It asserts that a measure preserving system is isomorphic to a Bernoulli shift, which from the point of view of ergodic theory immediately implies all the other chaotic properties, such as ergodicity, mixing or K-property. Starting with the work of Ornstein and Weiss \cite{16} a lot of effort was put into understanding Bernoullicity in the category of smooth dynamical systems. They developed new geometric tools which allowed them to reprove Katznelson's result \cite{11} on Bernoullicity of ergodic toral automorphisms and applied them to show that the geodesic flow on the unit tangent bundle of a surface of constant negative curvature is also Bernoulli. The major ingredient in their proofs was the use of isometric center in the partially hyperbolic structure of these systems. It also turned out to be the primary obstacle in generalizing their argument to other examples. 

However, in \cite{8} and later in \cite{6} the authors managed to extend these methods to partially hyperbolic systems with respectively polynomial and subexponential growth on the center. Both of these results are based on the fact that exponential mixing implies exponential equidistribution of the pieces of unstable manifolds and so the growth on the center is still relatively slow.

The opposite situation takes place for certian skew products, where we can control the dynamics on the fiber, which plays the role of the center, but only polynomial rates of mixing have been established \cite{5} and we do not expect these systems to satisfy exponential mixing.
It turns out that such systems give rise to new dynamical phenomena resulting in either Bernoulli or non Bernoulli systems. Historically they have been the first examples of $K$ but not Bernoulli systems \cite{kal,kat,rud} and recently \cite{9} a new class of examples of smooth $K$ non Bernoulli systems has been found by considering a Kochergin flow with sufficiently strong singularity in the fiber. Generally most examples of skew product systems with hyperbolic base were shown to be non-Bernoulli. The only exception until recently was the result of Rudolph, \cite{rud2}, who showed that isometric extensions of Bernoulli systems are Bernoulli. However no examples of weakly mixing fibers were known for which the resulting skew product would remain Bernoulli. This changed with the work of Dong and Kanigowski \cite{7} where the authors showed in particular that typical translation flows in the fiber produce a Bernoulli system. Nevertheless their examples were rigid and hence not mixing. 

The main goal of this paper is to study skew products over hyperbolic systems with mixing fibers. Given the work \cite{9} it is natural to consider Kochergin flows in the fiber. It is known that all Kochergin flows are mixing, and \cite{9} shows the non-Bernoulli property only for highly degenerated saddle. In this paper we will be interested in Kochergin flows with lower degeneracies of the saddle. The main result is that if the degeneracy is lower than 1/2 and the cocycle has mean 0 then the resulting system is Bernoulli. This is in some contrast with \cite{9} where the authors assumed positive mean and strong degeneracy. We prove the following theorem (see Section 2.5 for the definition of $K^{\gamma,\alpha}_t$)

\begin{theorem}
    Let $(X,S,\lambda)$ be a mixing Anosov diffeomorphism with a Gibbs measure $\lambda$ corresponding to a H\"older continuous potential $\psi$, let $\varphi:X\rightarrow\mathbb{R}$ be an aperiodic H\"older continuous cocycle  satysfying $\int \varphi=0$ and let $K_t^{\gamma,\alpha}$ be a Kochergin flow on $(\mathbb{T}^2,\nu)$ with exponent $\gamma<1/2$. Then for a.e. $\alpha\in [0,1)$ the skew product 
    $$T(x,z)=S^{K_t^{\gamma,\alpha}}_\varphi(x,z)=(S(x),K^{\gamma,\alpha}_{\varphi(x)}(z)) $$ on $(X\times \mathbb{T}^2,\lambda\times\nu)$ is Bernoulli.
\end{theorem}

It is known \cite{1} that Anosov diffeomorphisms admit Markov partitions i.e. there exists a mixing subshift of finite type with an invariant measure $(\Sigma,\sigma,\mu)$ and a H\"older continuous $\pi:\Sigma\rightarrow X$, which is an isomorphism between the two measure preserving dynamical systems. Then if $\lambda$ is the Gibbs measure corresponding to $\psi$, it follows that $\mu$ is the Gibbs measure corresponding to $\psi\circ\pi$. Therefore in order to prove Theorem $1.1$ it is enough to prove the following

\begin{theorem}
    Let $(\Sigma,\sigma,\mu)$ be a mixing subshift of finite type with a Gibbs measure $\mu$ corresponding to some H\"older continuous potential, let $\varphi:X\rightarrow\mathbb{R}$ be an aperiodic H\"older continuous cocycle  satysfying $\int \varphi=0$ and let $K_t^{\gamma,\alpha}$ be a Kochergin flow with exponent $\gamma<1/2$. Then for a.e. $\alpha\in [0,1)$ the skew product
    $$T(x,z)=\sigma^{K_t^{\gamma,\alpha}}_\varphi(x,z)=(\sigma(x),K^{\gamma,\alpha}_{\varphi(x)}(z)) $$ on $(\Sigma\times \mathbb{T}^2,\mu\times\nu)$ is Bernoulli.
\end{theorem}

We focus on proving this theorem in the rest of the paper.
\paragraph{Outline of the proof}
Following \cite{16}, we reduce the proof of Bernoullicity to finding a generating VWB (very weakly Bernoulli) partition. We specify the generating partition $\mathcal{P}\times\mathcal{Q}$ in Section 2.5, in fact our proof works for any smooth (satisfying \eqref{smpart}) partition $\mathcal{Q}$ of the fiber and so the main part is the proof of the VWB property. By the definition, it means that we want to find a matching between most atoms of the past partition $\bigvee_{i=0}^\infty T^i(\mathcal{P}\times\mathcal{Q})$. By property \eqref{smpart} and the fact that the fiber entropy is 0, we see by Lemma 2.11, that all we need is to find a measurable isomorphism between most of the unstable sets $\Phi:\Sigma^+(x)\rightarrow\Sigma^+(\bar x)$ in the base, which would cause the fiber coordinates of $T^i(x,z)$ and $T^i(\Phi(x),\bar z)$ to stay close to each other for most of the times $i$ in some future interval $[1,N]$.

In most examples VWB has been established by constructing a matching which would bring two points close enough in time $o(N)$ to ensure that for the remaining $N-o(N)$ times in the future interval $[1,N]$ they would stay close together. In the case of skew products over SFTs this corresponds to finding the right matching between blocks of size $o(N)$ and then copying the remaining $N-o(N)$ symbols.

The situation gets more complicated in Theorem 1.2 and interestingly it requires a matching with $\mathcal{P}\times\mathcal{Q}$ names of $\{T^i(x,z)\}_{i=1}^N$ and $\{T^i(\Phi(x),\bar z)\}_{i=1}^N$ agreeing on a family of disjoint short intervals of $i$, each of length $o(N)$.

Indeed, even with the assumption of good equidistribution in the fiber given by the slow growth of denominators \eqref{lancuch} of the rotation angle in the Kochergin flow, the time it takes to bring two typical points in the fiber $\frac{1}{N}$-close together is roughly of size $N$. On the other hand, the Denjoy-Koksma inequality tells us that two points which are $\frac{1}{N}$ apart in the horizontal direction in the fiber can only travel together in the direction of the flow for a distance bounded by $N^{\frac{2}{3}+\epsilon}$ before getting separated, if the exponent $\gamma$ is close to $\frac{1}{2}$.

The idea is therefore to construct the matching one block at a time and to carefully calculate the maximal measure of the domain on which it is possible to shift the Birkhoff sums of the cocycle $S_n\varphi$ by some time $t$. Obtaining the optimal measure, equal to $\frac{t}{\sqrt{n}}$, involves studying various limit theorems for conditional measures on the unstable sets $\Sigma^+(x)$ and we dedicate the whole Section 3. to that matter.

Then we may hope to adjust the two trajectories with the use of the next block and $t$ chosen to be the distance by which our two points have separated by the time the previous block ends. Section 4 gives upper bounds on these distances.

In Section 5 we construct the final matching using Lemma \ref{32final} on blocks of varying length, that is with different $n$. The idea is to fix a finite set of lengths ($H_i(N),\ i=1,\dots,k$) and always use Lemma \ref{32final} with $n=H_i(N)$ for some $i$ and $t$ such that the loss of measure, which is roughly $\frac{t}{\sqrt{H_i(N)}}$, does not exceed the number of blocks of length $H_i(N)$ in our whole future interval $[1,N]$. In fact making sure the loss of measure for each such partial matching defined on a block at level $i$ (that is of length $H_i(N)$) is $o(\frac{L(N)}{H_i(N)})$, we obtain Lemma \ref{52good}, which guarantees that the matching shifts the Birkhoff sums in a controlled way (we fall into the set of good control from Lemma \ref{32final} at every step of the construction) for most words in $\Sigma^+(x)$. In Lemma \ref{52mapworks} we use the lemmas from Section 4 to bound the distances between our two points at the end of each block. Finally, in Lemma \ref{52precel} we show that for most words in $\Sigma^+(x)$ most of the interval $[1,L(N)]$ in our construction is covered by the shortest blocks (at level $k$) where the Birkhoff sums of the cocycle $\varphi$ can be bounded by $N^{\frac{2}{3}+\frac{2\eta}{10}}$ which suffices for $o(1)$ separation in the direction of the flow. 

\paragraph{Acknowledgements}
I would like to thank Adam Kanigowski for suggesting the topic as well as many helpful discussions. I am also grateful to Dmitry Dolgopyat for providing useful references and Kosma Kasprzak for detailed reading and feedback on a preliminary version of this paper. 

\section{Basic definitions and facts}

\subsection{Measure theory}
We recall the notion of an isomorphism of probability spaces.
A map $\phi: X\rightarrow Y$ between two probability spaces $(X,\mathcal{B},\mu)$ and $(Y,\mathcal{A},\nu)$ will be called an isomorphism if there exist $\bar X\in\mathcal{B}$ and $\bar Y\in\mathcal{A}$ with $\mu(\bar X)=\nu(\bar Y)=1$ such that $\phi_{|\bar X}:\bar X\rightarrow \bar Y$ is a bimeasurable bijection and $\nu(\phi(A))=\mu(A)$ for every $A\in\mathcal{B},\ A\subset\bar X$.
    
We also recall the definitions of a standard Borel probability space and a Lebesgue space. We say that a probability space $(X,\mathcal{B},\mu)$ is a standard Borel probability space if $\mu$ is nonatomic and there exists a metric on $X$ which turns it into a complete separable space, such that $\mathcal{B}$ is the Borel $\sigma$-algebra. $(X,\mathcal{B},\mu)$ is a Lebesgue space if it is isomorphic to $[0,1]$ with the $\sigma$-algebra of Lebesgue measurable subsets equipped with the Lebesgue measure.

Up to completion, every standard Borel probability space is a Lebesgue space.
    
\begin{theorem}
    Let $(X,\mathcal{B},\mu)$ be a standard Borel probability space, then  $(X,\mathcal{B},\mu)$ and $([0,1],\mathcal{A},m)$ are isomorphic, where $\mathcal{A}$ denotes the $\sigma$-algebra of Borel subsets and $m$ is the Lebesgue measure.
\end{theorem}

We will also need the following theorem from \cite{18}
    
\begin{theorem}
    If $(X,\mathcal{B},\mu)$ is a Lebesgue space and $A\in \mathcal{B}$ is such that $\mu(A)>0$ then $(A,\mathcal{B}_{|A},\mu_A)$ 
    where $\mathcal{B}_{|A}=\{B\cap A|\ B\in\mathcal{B}\}$, $\mu_A(B)=\mu(B)/\mu(A)$
    is also a Lebesgue space
\end{theorem}

When dealing with measurable partitions (in the sense of Rokhlin \cite{18}) into subsets of measure $0$, the conditional measures are described by the Rokhlin's disintegration theorem.

\begin{theorem}
    Let $(X,\mathcal{B},\mu)$ be a Lebesgue space and $P$ a measurable partition. Then there is a map $\Phi$ on $P$, such that $\Phi(C)$ is a probability measure on $(C,\mathcal{A}_{C})$ for any atom $C\in P$, for every $B\in \mathcal{B}$ we have $B\cap C\in \mathcal{A}_C$ for almost every $C$ and finally we have
    $$\mu(B)=\int_{P}\Phi(C)(B\cap C)\ d\mu_P(C)$$
    where $\mu_P(A)=\mu(\bigcup A)$ for $A\subset P$
\end{theorem}    

\subsection{Bernoulli property}

In order to prove Bernoullicity we will use the notion of VWB (very weakly Bernoulli) partitions. For a finite (and measurable) partition $\mathcal{P}=\{P_1,\dots,P_k\}$ of a Lebesgue space $(X,\mathcal{B},\mu)$, 
$$\bigvee_{i=0}^\infty T^i\mathcal{P}:=\{\bigcap_{i=0}^\infty T^i(P_{n_i})|\quad  1\le n_i\le k\quad\text{for every } i\}$$ is a measurable partition and Theorem 2.3. determines for almost every atom $r\in\bigvee_{i=0}^\infty T^iP$ a conditional measure $\mu_r$ on $r$. One of the many equivalent definitions (see \cite{20}) of VWB can be formulated with the use of conditional measures as follows 

\begin{df}\label{defvwb}
     Given a measure preserving dynamical system $(X,T,\mathcal{B},\mu)$, a finite partition $\mathcal{P}$ is called very weak Bernoulli if for every $\epsilon>0$ there exists $N\in\mathbb{N}$ and $G\subset\bigvee_{i=0}^\infty T^i\mathcal{P}$ such that $\mu(\bigcup G)>1-\epsilon$ and for any two atoms $r,\bar r\in G$
    there is a measurable isomorphism $\phi_{r,\bar r}^N:r\rightarrow \bar r$ and $U\subset r$ with $\mu_r(U)>1-\epsilon$ such that
    $$\frac{1}{N}\#\{1\le i\le N|\ T^i(\phi_{r,\bar r}^N(x))\ \text{and}\ T^i(x)\ \text{are in the same atom of}\ \mathcal{P}\}>1-\epsilon$$
    for $x\in U$    
\end{df}

These partitions are precisely the generators of Bernoulli systems as proven in \cite{21}.
We can therefore use the following theorem to establish Bernoullicity.
\begin{theorem}
    If $(X,T,\mathcal{B},\mu)$ is a measure preserving dynamical system and $\mathcal{P}$ is a VWB partition such that $\mathcal{P}$ generates $\mathcal{B}$ under $T$, then $(X,T,\mathcal{B},\mu)$ is Bernoulli.
\end{theorem}
\subsection{Subshifts of finite type}
By a mixing subshift of finite type we will understand a dynamical system $(\Sigma,\sigma)$,
where the space $\Sigma$ is defined by
$$\Sigma:=\{x=\dots,x_{-1},x_0,x_1,\ldots\in\mathcal{A}^{\mathbb{Z}}\ |\quad A_{x_i,x_{i+1}}=1\quad\text{for every}\ i\in\mathbb{Z} \},$$
where $\mathcal{A}=\{a_1,\dots,a_n\}$ is a finite alphabet and $A$ is an aperiodic\footnote{meaning that $\exists n$ such that $\forall k\ge n\ \forall i,j$ we have $(A^k)_{i,j}>0$} $n\times n$ matrix whose entries take values $0$ or $1$. The transformation $\sigma$ shifts the indices by one: $\sigma(x)_i=x_{i+1}$.

We define the standard metric $d$, which makes $(\Sigma,d)$ into a compact metric space, by
$d(x,y)=2^{-n}$, where $n$ is the largest integer such that $x_i=y_i$ for all $|i|< n$.
It is known \cite{22} that for each H\"older continuous potential $\psi:\Sigma\rightarrow\mathbb{R}$ there is a unique invariant Borel probability measure, called the Gibbs measure, $\mu_{\psi}$ on $\Sigma$ which maximizes the quantity
$$h_\mu(\sigma)+\int_\Sigma\psi\; d\mu,$$
called pressure.
Such measures exhibit good statistical properties which we will adjust to our problem in Section 3.

Additionally we would like to consider the projections onto the space of left infinite words $\pi_-:\Sigma\rightarrow\mathcal{A}^{\mathbb{Z}_{\le0}}$ and right infinite words $\pi_+:\Sigma\rightarrow\mathcal{A}^{\mathbb{Z}_{>0}}$. Each $\pi_i(\Sigma)$ carries a measure $\mu_i=\mu\circ\pi_i^{-1}$, for $i=+,-$, which makes it into a Lebesgue space.
It makes sense to consider the product measure $\mu_-\times\mu_+$ in the neighborhood of $x$ given by the ball $B:=B_d(x,1/2)$ and then it is known \cite{2} that Gibbs measures have a local product structure. This means that there exists $K>0$ such that for any such ball $B$, $\mu_{|B}$ is absolutely continuous with respect to $\mu_-\times\mu_+$ and 
$$\frac{d\mu_{|B}}{d\mu_-\times\mu_+}=h,$$
where $h:B\rightarrow\mathbb{R}$ is H\"older continuous and $K^{-1}<h<K$.

 In an attempt to prove the Bernoulli property of our system through the VWB property of partitions as described in Definition \ref{defvwb} it is essential to consider the 'past partition' $\Sigma^+$ of $\Sigma$, which is a measurable partition defined by specifying the atoms $\Sigma^+(x)$ of $\Sigma^+$ for every $x\in\Sigma$
$$\Sigma^+(x):=\pi_-^{-1}(\dots,x_{-1},x_0).$$
By Rokhlin's disintegration theorem almost every such atom can be assigned a conditional measure $\mu^+_x$. In the case of SFTs with Gibbs measures the situation is particularly nice as for every $x\in\Sigma$ there is a formula (\cite{22}) for the conditional measure of any Borel subset $E\subset\Sigma^+(x)$
$$\mu_x^+(E)=\lim_{n\to\infty}\frac{\mu([x_{-n},x_{-n+1}\dots,x_0]\cap\pi^{-1}_+(\pi_+(E)))}{\mu([x_{-n},x_{-n+1}\dots,x_0])},$$
where for $w\in\Sigma$ we use the notation 
$$[w_k,w_{k+1},\dots,w_{k+n}]:=\{y\in\Sigma:\quad y_i=w_i\quad\forall i= k,k+1,\dots,k+n\}.$$
Then it follows from local product structure that there exist $C_0, \beta>0$ depending only on $h$ such that if $E_1\subset\Sigma^+(x)$, $E_2\subset\Sigma^+(y)$ and $\pi_+(E_1)=\pi_+(E_2)$, then
$$\left|\frac{\mu_y^+(E_2)}{\mu_x^+(E_1)}-1\right|\le C_0d(x,y)^\beta$$
We can rewrite the above differently. If we let $H:\Sigma^+(x)\rightarrow\Sigma^+(y)$ be the stable holonomy map
$$H(\dots,x_{-1},x_0,x_1,x_2,\dots)=\dots y_{-1},y_0,x_1,x_2,\dots,$$
defined for $d(x,y)\le1/2$, then we obtain the following lemma.
\begin{lemma}\label{2hol}
    There exist $C_0, \beta>0$ such that for any $U\subset \Sigma^+(x)$
    $$\left|\frac{\mu_y^+(H(U))}{\mu_x^+(U)}-1\right|\le C_0d(x,y)^\beta$$
\end{lemma}
Since $\mu^+_y=\mu^+_{\bar y}$ for $y,\bar y$ sharing the same past, we actually have 
$$\left|\frac{\mu_y^+(H(U))}{\mu_x^+(U)}-1\right|\le C_02^{-k\beta}$$
for any $x,y$ with $x_{-k}x_{1-k}\dots x_0=y_{-k}y_{1-k}\dots y_0$.

Let $\Sigma^+_n(x)=\Sigma_x^n:=\sigma^{-n}(\Sigma^+(\sigma^{n}(x)))$ and let $\mu^{n}_x$ be the conditional measures corresponding to this partition. For SFTs we also get the transitive property for conditional measures, namely for $w\in\Sigma^+(x)$, $k>0$ and $U\subset\Sigma^+_k(w)$ we have
\begin{equation*}
    \begin{split}
        \mu_x^+(U)=\lim_{n\to\infty}\frac{\mu([x_{-n},x_{-n+1}\dots,x_0]\cap\pi^{-1}_+(\pi_+(U)))}{\mu([x_{-n},x_{-n+1}\dots,x_0])}=\lim_{n\to\infty}\frac{\mu([w_{-n},w_{-n+1}\dots,w_k]\cap\pi^{-1}_+(\pi_+(U)))}{\mu([x_{-n},x_{-n+1}\dots,x_0])}\\
        =\lim_{n\to\infty}\frac{\mu([w_{-n},w_{-n+1}\dots,w_k]\cap\pi^{-1}_+(\pi_+(U)))}{\mu([w_{-n},w_{-n+1}\dots,w_k])}\frac{\mu([w_{-n},w_{-n+1}\dots,w_k])}{\mu([x_{-n},x_{-n+1}\dots,x_0])}\\
        =\lim_{n\to\infty}\frac{\mu([w_{-n},w_{-n+1}\dots,w_k]\cap\pi^{-1}_+(\pi_+(U)))}{\mu([w_{-n},w_{-n+1}\dots,w_k])}\lim_{n\to\infty}\frac{\mu([x_{-n},x_{-n+1}\dots,x_0])\cap[w_1,\dots,w_k]}{\mu([x_{-n},x_{-n+1}\dots,x_0])}\\
        =\mu_x^+([w_1,\dots,w_k]\cap\Sigma^+(x))\mu_w^k(U)
    \end{split}
\end{equation*}
Since $\mu$ is shift invariant and we can condition on more than one word, we obtain the following generalization.
\begin{lemma}\label{2trans}
    For every $x$ the conditional measures $\mu_x^k$ split into linear combinations of $\mu_w^{k+n}$ for $w\in\Sigma^+_k(x)$, that is for any $k,n>0,U\subset\Sigma^+_k(x)$  
    $$\mu_x^k(U)=\sum_{\substack{[w_{k+1}w_{k+2}\dots w_{k+n}]\\ \text{such that}\\
    \exists w=\ldots w_{0}w_{1}\ldots\in\Sigma^+_k(x)}}\mu_x^{k}([w_{k+1}w_{k+2}\dots w_{k+n}]\cap\Sigma^+_k(x))\mu_{w}^{k+n}(U\cap\Sigma_{k+n}^+(w))$$
\end{lemma}
\begin{remark}\label{rem}
    In particular for $A\subset \Sigma_k^+(x)$, $\bar A\subset\Sigma_k^+(\bar x)$ and $B\subset A$, $\bar B\subset\bar A$\\
    1.
    $$\frac{\mu_x^+(B)}{\mu_x^+(A)}=\frac{\mu_x^k(B)}{\mu_x^k(A)}$$
    hence
    $$\frac{\mu_x^+(B)}{\mu_{\bar x}^+(\bar B)}=\frac{\frac{\mu_x^+(A)}{\mu_x^k(A)}\mu_x^k(B)}{\frac{\mu_{\bar x}^+(\bar A)}{\mu_{\bar x}^k(\bar A)}\mu_{\bar x}^k(\bar B)}.$$
    If $A,\bar A$ are such that $\mu_x^+(A)=\mu_{\bar x}^+(\bar A)$, then
    we have \\2.
    $$\frac{\mu_x^+(B)}{\mu_{\bar x}^+(\bar B)}=\frac{\mu_x^k(B)\mu_{\bar x}^k(\bar A)}{\mu_{ x}^k( A)\mu_{\bar x}^k(\bar B)}$$
\end{remark}

\subsection{Continued fraction expansion}
For $\alpha\in(0,1)$ we can consider its continued fraction expansion 
$$\alpha=\frac{1}{a_1+\frac{1}{a_2+\frac{1}{a_3+\dots}}},$$
with $a_n$ positive integers for $n\ge1$.
Then if we write 
$$\frac{p_n}{q_n}=\frac{1}{a_1+\frac{1}{a_2+\frac{1}{\dots+\frac{1}{a_n}}}}$$
where $p_n$ and $q_n$ are coprime for $n\ge 1$, we obtain a sequence of canonical denominators $(q_n)_{n=1}^\infty$ of $\alpha$ with the property that 
$$\frac{1}{q_n+q_{n+1}}<\rVert q_n\alpha\rVert<\frac{1}{q_{n+1}},$$
where $\rVert\cdot\rVert$ denotes the distance to the closest integer.
The set of $\alpha\in(0,1)$ we will work with is taken to satisfy $q_{n+1}\le q_n\log^2 q_n$ for sufficiently big $n$. It is easy to see that such irrationals form a subset of the interval $[0,1)$ of Lebesgue measure 1. Indeed, if we consider the sets 
$$C_n:=\left\{x\in[0,1):\quad \rVert nx\rVert\le\frac{1}{n\log^2 n}\right\}$$ 
and the Lebesgue measure $\mu$, then we get 
$\mu(C_n)\le\frac{2}{n\log^2 n}.$
It follows that
$$\sum_{n=1}^\infty\mu(C_n)<\infty$$
and as a consequence of the Borel-Cantelli lemma we obtain that only for a subset of $\alpha\in[0,1)$ of measure 0 the inequality
$$\rVert q_n\alpha\rVert<\frac{1}{q_{n+1}}\le\frac{1}{q_n\log^2 q_n}$$
could be satisfied for infinitely many $n\ge1$.
\subsection{Kochergin flows and VWB for zero entropy fiber}
In \cite{12} Kochergin studied certian locally Hamiltonian flows on the 2-torus and obtained their representations as special flows.
We start with the definition of a special flow. For a measure preserving system $(X,T,\mu)$ and $f\in L^1(X),\ f>0$ the special flow $T^f$ is defined on the space
$$X^f:=\{(x,y)\in X\times\mathbb{R}:\quad 0\le y<f(x)\}$$
and it preserves a measure which is the restriction of $\mu\times m$ to $X^f$, where $m$ is the one dimensional Lebesgue measure on $\mathbb{R}$.
It is defined by the formula $T^f_t(x,y)=(x_t,y_t),$
where $x_t=T^{N((x,y),t)}(x)$ and $y_t=y+t-S_{N((x,y),t)}f(x)$.\\
Here $N((x,y),t)$ is the unique integer such that
$$S_{N((x,y),t)}f(x)\le y+t<S_{N((x,y),t)+1}f(x)$$
and for $n\in\mathbb{Z},\ x\in X$
\begin{equation*}
   S_nf(x):=
    \begin{cases}
      \sum_{i=0}^{n-1}f(T^i(x)), & \text{if}\  0< n \\
      0  & \text{if}\ n=0\\
       \sum_{i=1}^{-n}f(T^{-i}(x))   & \text{if}\ n<0
    \end{cases}
  \end{equation*}
We define the Kochergin flow $K^{\gamma,\alpha}_t$ with exponent $\gamma$ as any special flow $R_\alpha^f$ on $M:=[0,1)^f$ over the irrational rotation by $\alpha$, with a roof function $f\in C^2((0,1),\mathbb{R})$, $f>0,\ f''>0,\ \int _{[0,1)}f=1$
such that 
$$\lim_{x\to 0^+}\frac{f(x)}{x^{-\gamma}}=M_1$$
$$\lim_{x\to 1^-}\frac{f(x)}{(1-x)^{-\gamma}}=M_2$$
$$\lim_{x\to 0^+}\frac{f'(x)}{x^{-\gamma-1}}=-M_3$$
$$\lim_{x\to 1^-}\frac{f'(x)}{(1-x)^{-\gamma-1}}=M_4$$
for some $M_i>0$. Since Theorem 1.2 is for arbitrary $1/2>\gamma>0$ we will WLOG assume that $M_i<1$. Let $\nu$ be the invariant measure for this special flow. 

Restricting further considerations to a set of $\alpha\in[0,1)$ of Lebesgue measure 1 (see previous section), we will assume that the denominators of $\alpha$ satisfy 
\begin{equation}\label{lancuch}
    q_{n+1}\le q_n\log^2 q_n,
\end{equation}
for sufficiently big $n$. For $a\in[0,1)\cong\mathbb{R/\mathbb{Z}}$ let $S_nf(a):=\sum_{i=0}^{n-1}f(a+i\alpha)$. Then as a consequence of the Denjoy-Koksma inequality (\cite{9}) we obtain 

\begin{lemma}[Denjoy-Koksma]
    For $q$, a denominator of $\alpha$ we have
    $$S_qf(z_1)\le q+f_{q,max}(z_1)$$
    $$|S_qf'(z_1)|\le qf_{q,max}(z_1)+2f'_{q,max}(z_1)$$
    and for sufficiently big $n$ we have 
    $$S_nf(z_1)\le n\log^2 n+f_{n,max}(z_1)$$
    $$|S_nf'(z_1)|\le n\log^2 nf_{n,max}(z_1)+2f'_{n,max}(z_1),$$ 
    where $f_{n,max}(a)=\max\{f(a+i\alpha)\}_{i=0}^n$ and $f'_{n,max}(a)=\max\{|f'(a+i\alpha)|\}_{i=0}^n$
\end{lemma}

Let $\mathcal{P}$ be the partition of $\Sigma$ by the $0^{th}$ symbol. Then it is generating for the base $(\Sigma,\sigma,\mu)$.
Let $\mathcal{Q}=\{Q_1,M\setminus Q_1\}\vee R$ be a partition of $M$, where $Q_1=\{(x,y)\in [0,1)^f|\ f(x)-y\le \inf f/2\}$ and $R=(R_1,\dots,R_k)$ is a partition of $[0,1)^f$ by vertical segments such that for any $(x,y)\in R_i$ and $(\bar x,\bar y)\in R_i$ we have $\rVert x-\bar x\rVert\le \rVert\alpha\rVert/2$.
Then analogously to \cite{7} one can prove that the partition $\mathcal{P}\times \mathcal{Q}$ is generating. 
The only other property of $\mathcal{Q}$ we will need is that 
\begin{equation}\label{smpart}
    \nu(V_\delta (\partial\mathcal{Q}))\xrightarrow{\delta\to 0}0,
\end{equation}
where $V_\delta(\partial \mathcal{Q})$ is a $\delta$ neighborhood of the union of boundaries of atoms of $\mathcal{Q}$ with respect to a metric $d$ on $M$ (see Section 4).
We can therefore use Theorem 2.5. and reduce our problem of showing Bernoullicity to showing that $\mathcal{P}\times \mathcal{Q}$ is VWB. This is achieved by finding a certain matching using the following lemmas from \cite{7},\cite{9}.\\
First, by the fact that $h_\nu(K_t^{\gamma,\alpha})=0$ it can be shown that the second coordinate in the past partition is almost surely determined.
\begin{lemma}\label{2par}
    We have $\bigvee_{i=0}^\infty T^i(\mathcal{P}\times \mathcal{Q})=\Gamma$ as measurable partitions of $\Sigma\times M$. Where the atoms $r\in\Gamma$ are given by $r=\Sigma^+(x)\times\{z\}$
    and the conditional measures are given by
    $\mu_r=\mu_x^+\times\delta_z$
\end{lemma}
Second, this implies that we only need to find a matching between the atoms of the past partition of the base.
\begin{lemma}\label{2vwb}
    $\mathcal{P}\times Q$ is VWB if for every $\epsilon>0$ there exists $N\in\mathbb{N}$ and $G\subset\Gamma$ such that $\mu\times\nu(G)>1-\epsilon$ and for any two atoms $r,\bar r\in G$
    such that $r=\Sigma^+(x)\times\{z\}$, $\bar r=\Sigma^+(\bar x)\times\{\bar z\}$
    there is a measurable isomorphism $\phi_{r,\bar r}^N:\Sigma^+( x)\rightarrow \Sigma^+(\bar x)$ and $U\subset r$ with $\mu_r(U)>1-\epsilon$ such that
    $$1/N\#\{1\le i\le N|\ T^i(\phi_{r,\bar r}^N(x'),\bar z)\ \text{and}\ T^i(x',z)\ \text{are in the same atom of}\ \mathcal{P}\times \mathcal{Q}\}>1-\epsilon$$
    for $(x',z)\in U$
\end{lemma}
We will see that for sufficiently big $N$ there exists such $\phi_{r,\bar r}^N$, which we construct throughout the rest of the paper. 
\section{Statistical properties of $\varphi$}
In this section we start the proof of Theorem 1.2, however we only focus of the space $\Sigma$ and the cocycle $\varphi$.
In \cite{22} we can find the following useful lemma.
\begin{lemma}\label{3past}
    For any H\"older continuous $\varphi:\Sigma\rightarrow\mathbb{R}$ there is a H\"older continuous $h:\Sigma\rightarrow\mathbb{R}$ such that $\varphi+h\circ\sigma-h$ depends only on the past, that is the coordinates $\dots,x_{-1},x_0$ of $x\in\Sigma$.
\end{lemma}
Therefore up to an isomorphism we can assume that $\varphi$ depends only on the past, which we do in the rest of the paper.
\begin{lemma}\label{3geo}
    For every $\delta>0$ there exist $C_1,\ s_1>0$ such that for any $w\in \Sigma$, integers $m\ge k$
    and $r\le m-k$ and $y\in [w_{k},\dots,w_{m}]$, we have
\begin{equation*}
    |S_{r}\varphi(\sigma^{k}(y))-S_{r}\varphi(\sigma^{k}(w))|\le C_1
\end{equation*}
and if $y\in [w_{k-s_1},\dots,w_{m}]$, then
\begin{equation*}
    |S_{r}\varphi(\sigma^{k}(y))-S_{r}\varphi(\sigma^{k}(w))|\le \delta.
\end{equation*}
\end{lemma}
\begin{proof}
    Let $k_1$ and $\gamma<1$ be such that 
    $$|\varphi(x_1)-\varphi(x_2)|\le k_1 d(x_1,x_2)^\gamma$$
    Then notice that for each $i\le m-k$ the words $\sigma^{k+i}(y)$ and $\sigma^{k+i}(w)$ do not differ on coordinates $0,\dots,-i$. Therefore 
    $$|\varphi(\sigma^{k+i}(y))-\varphi(\sigma^{k+i}(w))|\le k_1 (2^{-\gamma})^i$$
    and summing the geometric series $C_1$ can be chosen as $k_1\sum_{i=0}^{\infty}(2^{-\gamma})^i$. If $s_1$ is big enough then $k_1\sum_{i=s_1}^{\infty}(2^{-\gamma})^i<\delta$
\end{proof}
Another useful fact \cite{17} is the exponential rate of mixing for Gibbs measures.
\begin{lemma}\label{3mix}
    There exist $1>\alpha>0$ and $c_1>0$ such that for any two Lipschitz continuous $v,w:\Sigma\rightarrow\mathbb{R}$ we have 
    $$\left|\int_\Sigma w(v\circ\sigma^{-n})\ d\mu-\int_\Sigma w\ d\mu\int_\Sigma v\ d\mu\right|\le c_1\alpha^n\rVert v\rVert\rVert w\rVert$$
    where $\rVert f\rVert$ is the sum of $\sup_{\Sigma}|f|$ and the Lipschitz constant of $f$.
\end{lemma}
From this we obtain
\begin{lemma}\label{3equi}
    There exists $d>0$ such that for every $k\in\mathbb{N},x\in\Sigma$ and every word $w_0w_1\dots w_k$  
    $$\mu_x^+\{y\in\Sigma^+(x):\ \sigma^{dk}(y)\in[w_0w_1\dots w_k]\}\in[(1-e^{-k})\mu([w_0w_1\dots w_k]),(1+e^{-k})\mu([w_0w_1\dots w_k])]$$
\end{lemma}
\begin{proof}
    One of the characterizations of the Gibbs measure associated to the potential $\psi$ among all the invariant measures is that there exist constants $c>0$ and $P$ such that
    $$c^{-1}\le\frac{\mu([x_{-m},x_{-m+1},\dots,x_m])}{\exp(-Pm+\sum_{i=-m}^{m}\psi\circ\sigma^i(x))}\le c$$
    for all $x\in\Sigma$, $m\in\mathbb{N}$. Thus there exists $d_0>0$ such that for every $d_1,\ k\in\mathbb{N},\ x\in\Sigma$ and every word $w_0w_1\dots w_k$ we have
    $$\mu([x_{-d_1k},\dots,x_{d_1k}])\mu([w_0,\dots,w_k])\ge 2^{-d_0d_1k}.$$ Taking $w$ and $v$ in Lemma \ref{3mix} to be the characteristic functions of $[x_{-d_1k},\dots,x_{0}]$ and $[w_0,\dots,w_k]$ respectively we get $\rVert w\rVert\le2^{2k}$ and $\rVert v\rVert\le2^{2d_1k}$. Using the above inequality we obtain after dividing by the integrals of $w$ and $ v$ that
    \begin{equation}\label{3equi1}
        \left|\frac{\mu([x_{-d_1k},\dots,x_{0}]\cap\sigma^{-n}[w_0,\dots,w_k])}{\mu([x_{-d_1k},\dots,x_{0}])\mu([w_0,\dots,w_k])}-1\right|\le c_1\alpha^n2^{2d_1k+2k+d_0d_1k}.
    \end{equation}
    By local product structure (or Lemma \ref{2hol}) we have
    \begin{equation}\label{3equi2}
        \left|\frac{\frac{\mu([x_{-d_1k},\dots,x_{0}]\cap\sigma^{-n}[w_0,\dots,w_k])}{\mu([x_{-d_1k},\dots,x_{0}])}}{\mu_x^+(\Sigma^+(x)\cap\sigma^{-n}[w_0,\dots,w_k])}-1\right|\le C_02^{-\beta d_1k}.
    \end{equation}
    Choosing $d_1$ such that $C_02^{-\beta d_1}\le \frac{1}{3e}$ and then choosing
     $n= dk$, where $d$ is such that $c_1\alpha^d2^{2d_1+2+d_0d_1}\le \frac{1}{3e}$, we see that \eqref{3equi1} and \eqref{3equi2} get bounded by $\frac{e^{-k}}{3^k}$, thus
     $$\left|\frac{\mu_x^+(\Sigma^+(x)\cap\sigma^{-n}[w_0,\dots,w_k])}{\mu([w_0,\dots,w_k])}-1\right|\le e^{-k},$$
     which ends the proof
\end{proof}

\begin{df}
    A function $f : \Sigma \rightarrow \mathbb{R}$ is periodic if there exist $\rho \in \mathbb{R},\ g : \Sigma \rightarrow
\mathbb{R}$ measurable, $\lambda > 0$ and $q : \Sigma \rightarrow \mathbb{Z}$, such that $f= \rho +g-g \circ\sigma +\lambda q$ almost everywhere. Otherwise, it is aperiodic.
\end{df}
In Theorem 1.2 we assume the following.
  \begin{as}
      $\varphi$ is aperiodic and $\int_\Sigma\varphi\ d\mu=0$
  \end{as}
  Under the above assumption it is known \cite{17} that the sums $S_n\varphi:=\sum_{i=0}^{n-1}\varphi\circ\sigma^i$ satisfy the CLT with the same rate of convergence as in the case of sums of i.i.d. random variables.
\begin{lemma}\label{3clt}
    $(\Sigma,\sigma,\mu)$ satisfies the $O(1/\sqrt{n})$ rate of convergence CLT, i.e. there exists $\rho>0$ such that
    $$\mu\{w\in\Sigma:\ \frac{S_n\varphi(w)}{\sqrt{n}}\le a\}=\mathfrak{R}(a)+O(1/\sqrt{n})$$
    uniformly in $a\in\mathbb{R}$, where we let $\mathfrak{R}$ be the cumulative distribution function of the normal distribution $\mathcal{N}(0,\rho^2)$.
\end{lemma}

\subsection{Limit theorems for conditional measures}

Let $\mathfrak{R}$ be the cumulative distribution function of the normal distribution $\mathcal{N}(0,\rho^2)$ as in Lemma \ref{3clt}.

 \begin{lemma}\label{31clt}
    For any $r\ge 0, z\in\Sigma$
     $$\mu_x^+\{w\in\Sigma^+(x)\cap\sigma^{-n}([z_0,\dots,z_r]):\ \frac{S_n\varphi(w)}{\sqrt{n}}\le a\}=\mu([z_0,\dots,z_r])\mathfrak{R}(a)+O(\log n/\sqrt{n})$$
     uniformly in $a\in\mathbb{R},x\in\Sigma$.
\end{lemma}
\begin{proof}
We first show that
$$\mu_x^+\{w\in\Sigma^+(x):\ \frac{S_n\varphi(w)}{\sqrt{n}}\le a\}=\mathfrak{R}(a)+O(\log n/\sqrt{n})$$
uniformly in $a\in\mathbb{R},x\in\Sigma$.

Let $\Sigma^+(x)=\bigcup_{i\in I}\Sigma^+_{\frac{d}{\beta}\log n}(y_i)$ where $y_i\in\Sigma^+(x)$ are such that $\Sigma^+_{\frac{d}{\beta}\log n}(y_i)$ are disjoint. Let $c>0$ be such that $\varphi\le c$ and let $W$ be a family of cylinders $W:=\{[w_{-\frac{1}{\beta}\log n},\dots,w_0]\subset\Sigma\}$. We have
    \begin{equation}\label{31clt1}
        \begin{split}
            \mu_x^+\{w\in\Sigma^+(x):\ \frac{S_n\varphi(w)}{\sqrt{n}}\le a\}\le \mu_x^+\{w\in\Sigma^+(x):\ S_{n-\frac{d}{\beta}\log n}\varphi(\sigma^{\frac{d}{\beta}\log n}(w))\le a\sqrt{n}+c\frac{d}{\beta}\log n\}\\= \text{by Lemma \ref{2trans}}\\
            \sum_{C\in W}\sum_{y_i\in\sigma^{-\frac{d}{\beta}\log n}(C)}\mu_x^+(\Sigma^+_{\frac{d}{\beta}\log n}(y_i))\mu_{y_i}^{\frac{d}{\beta}\log n}\{w\in\Sigma^+_{\frac{d}{\beta}\log n}(y_i):\ S_{n-\frac{d}{\beta}\log n}\varphi(\sigma^{\frac{d}{\beta}\log n}(w))\le a\sqrt{n}+c\frac{d}{\beta}\log n\}
        \end{split}
    \end{equation}
    For $r>0$ define $H^r_{y,\bar y}:\Sigma^+_r(\bar y)\rightarrow\Sigma^+_r(y)$ by

   $$ H^r_{y,\bar y}(S)=\sigma^{-r}(H_{\sigma^r(y),\sigma^r(\bar y)}(\sigma^r(S)).$$

   Then for $y,\bar y\in\sigma^{-\frac{d}{\beta}\log n}([w_{-\frac{1}{\beta}\log n}\dots w_{0}])$, by Lemma \ref{2hol}
     \begin{equation}\label{31clt2}
         \left|\frac{\mu_y^{\frac{d}{\beta}\log n}(H^{\frac{d}{\beta}\log n}_{y,\bar y}(U))}{\mu_{\bar y}^{\frac{d}{\beta}\log n}(U)}-1\right|\le C_0/n.
     \end{equation}
   
    Moreover, by Lemma \ref{3geo} there is a constant $C_1>0$ such that for every $C\in W$ and $y,\bar y\in\sigma^{-\frac{d}{\beta}\log n}(C)$
    \begin{equation}\label{31clt3}
        \begin{split}
            \{w\in\Sigma^+_{\frac{d}{\beta}\log n}(y):\ S_{n-\frac{d}{\beta}\log n}\varphi(\sigma^{\frac{d}{\beta}\log n}(w))\le a\sqrt{n}+c\frac{d}{\beta}\log n\}\\
            \subset H^{\frac{d}{\beta}\log n}_{y,\bar y}(\{w\in\Sigma^+_{\frac{d}{\beta}\log n}(\bar y):\ S_{n-\frac{d}{\beta}\log n}\varphi(\sigma^{\frac{d}{\beta}\log n}(w))\le a\sqrt{n}+c\frac{d}{\beta}\log n+C_1\}).
        \end{split}
    \end{equation}
    Combining \eqref{31clt2} and \eqref{31clt3} we obtain for every $y_i\in\sigma^{-\frac{d}{\beta}\log n}(C)$ and every $y\in\sigma^{-\frac{d}{\beta}\log n}(C)$
    \begin{equation*}
    \begin{split}
         \mu_{y_i}^{\frac{d}{\beta}\log n}(\{w\in\Sigma^+_{\frac{d}{\beta}\log n}(y_i):\ S_{n-\frac{d}{\beta}\log n}\varphi(\sigma^{\frac{d}{\beta}\log n}(w))\le a\sqrt{n}+c\frac{d}{\beta}\log n\})\\
         \le(1+C_0/n)\mu_{y}^{\frac{d}{\beta}\log n}(\{w\in \Sigma^+_{\frac{d}{\beta}\log n}(y):\ S_{n-\frac{d}{\beta}\log n}\varphi(\sigma^{\frac{d}{\beta}\log n}(w))\le a\sqrt{n}+c\frac{d}{\beta}\log n+C_1\})
    \end{split}           
    \end{equation*}
    and using the disintegration theorem we get
    \begin{equation*}
    \begin{split}
         \mu_{y_i}^{\frac{d}{\beta}\log n}(\{w\in\Sigma^+_{\frac{d}{\beta}\log n}(y_i):\ S_{n-\frac{d}{\beta}\log n}\varphi(\sigma^{\frac{d}{\beta}\log n}(w))\le a\sqrt{n}+c\frac{d}{\beta}\log n\})\\
         \le\frac{\mu(\{w\in\sigma^{-\frac{d}{\beta}\log n}(C):\ S_{n-\frac{d}{\beta}\log n}\varphi(\sigma^{\frac{d}{\beta}\log n}(w))\le a\sqrt{n}+c\frac{d}{\beta}\log n+C_1\})}{\mu(C)}+C_0/n
    \end{split}           
    \end{equation*}
    Hence, together with
    \begin{equation*}
        \sum_{y_i\in\sigma^{-\frac{d}{\beta}\log n}(C)}\mu_x^+(\Sigma^+_{\frac{d}{\beta}\log n}(y_i))=\mu_x^+(\Sigma^+(x)\cap\sigma^{-\frac{d}{\beta}\log n}(C))=\mu(C)+O(1/n),
    \end{equation*}
    which comes from Lemma \ref{3equi}, we can bound \eqref{31clt1} by
    \begin{equation*}
        \sum_{C\in W}\mu(C)\frac{\mu\{w\in C:\ S_{n-\frac{d}{\beta}\log n}\varphi(\sigma^{\frac{d}{\beta}\log n}(w))\le a\sqrt{n}+c\frac{d}{\beta}\log n+C_1\}}{\mu(C)}+O(1/n),
    \end{equation*}
    which is the same as 
    \begin{equation*}
        \begin{split}
            \mu\{w\in\Sigma:\ S_{n-\frac{d}{\beta}\log n}\varphi(\sigma^{\frac{d}{\beta}\log n}(w))\le a\sqrt{n}+c\frac{d}{\beta}\log n+C_1\}\stackrel{\text{Lemma 3.7}}{=}\mathfrak{R}(\frac{a\sqrt{n}+c\frac{d}{\beta}\log n+C_1}{\sqrt{n-\frac{d}{\beta}\log n}})+O(1/\sqrt{n})\\
            =\mathfrak{R}(a)+O(\log n/\sqrt{n})
        \end{split}
    \end{equation*}
    and bounding from below is analogous.
    Now to show that for any $r\ge 0, z\in\Sigma$
     $$\mu_x^+\{w\in\Sigma^+(x)\cap\sigma^{-n}([z_0,\dots,z_r]):\ \frac{S_n\varphi(w)}{\sqrt{n}}\le a\}=\mu([z_0,\dots,z_r])\mathfrak{R}(a)+O(\log n/\sqrt{n})$$
     uniformly in $a\in\mathbb{R},x\in\Sigma$, notice that
     \begin{equation*}
         \begin{split}
             \mu_x^+\{w\in\Sigma^+(x)\cap\sigma^{-n}([z_0,\dots,z_r]):\ \frac{S_n\varphi(w)}{\sqrt{n}}\le a\}\\
             \le 
             \mu_x^+\{w\in\Sigma^+(x)\cap\sigma^{-n}([z_0,\dots,z_r]):\ S_{n-d\log n}\varphi(w)\le a\sqrt{n}+cd\log n\}\\=\sum_{y\in A(n)}\mu_x^+(\Sigma^+_{n-d\log n}(y))\mu_{y}^{n-d\log n}(\Sigma^+_{n-d\log n}(y)\cap\sigma^{-n}([z_0,\dots,z_r]))\\
             \stackrel{\text{Lemma 3.4}}{=}\sum_{y\in A(n)}\mu_x^+(\Sigma^+_{n-d\log n}(y))(\mu([z_0,\dots,z_r]+O(1/n))\\=
             \mu_x^+\{w\in\Sigma^+(x):\ S_{n-d\log n}\varphi(w)\le a\sqrt{n}+cd\log n\}(\mu([z_0,\dots,z_r]+O(1/n))
             =\mu([z_0,\dots,z_r])\mathfrak{R}(a)+O(\log n/\sqrt{n}),
         \end{split}
     \end{equation*}
     where $A(n)$ is a set of $y\in\Sigma^+(x)$ such that $$\bigcup_{y\in A(n)}\Sigma^+_{n-d\log n}(y)=\{w\in\Sigma^+(x):\ S_{n-d\log n}\varphi(w)\le a\sqrt{n}+cd\log n\}$$ and $\Sigma^+_{n-d\log n}(y)$ are disjoint. This gives the upper bound and bounding from below is analogous.
\end{proof}
As proven in \cite{7}, the mixing local limit theorem can be stated without changes for conditional measures. Let $\mathfrak{n}$ be the density function of $\mathcal{N}(0,\rho^2)$.
\begin{lemma}\label{31llt}
    $(\Sigma,\sigma,\mu)$ satisfies MLLT :\\
    for any $r>0$, interval $I\subset\mathbb{R}$ and $w\in\Sigma^+(x)$
$$\mu_x^+(y\in\Sigma^+(x)\cap[w_{N},\dots w_{N+r}]:S_N\varphi(y)\in k\sqrt{N}+I)=\frac{\mathfrak{n}(k)|I|\mu([w_{N},\dots w_{N+r}])}{\sqrt N}+o(1/\sqrt{N})$$
uniformly in $x\in\Sigma,\ k$ from a compact subset of $\mathbb{R}$.
\end{lemma}
Similarly to Kolmogorov's maximal inequality for i.i.d. random variables one can prove such result for conditional measures. For sufficiently big $N$
     $$\mu_x^+(w\in\Sigma^+(x):\ \max_{1\le k\le N}|S_k\varphi(w)|\ge a)\le 2\mu_x^+(w\in\Sigma^+(x):\ |S_N\varphi(w)|\ge a-2\sqrt{N})$$
     for all $a\in\mathbb{R}$.

Then by the above using Lemma \ref{31clt}, since $\mathfrak{R}(\log N-2)$ is smaller than the $O(\log N/\sqrt{N})$ error term, we obtain the following bound
\begin{lemma}\label{31kol}
    $$\mu_x^+(w\in\Sigma^+(x):\ \max_{1\le k\le N}|S_k\varphi(w))|\ge\sqrt{N}\log N)=O(\log N/\sqrt{N})$$
\end{lemma}

\subsection{Defining the matching on a single block}
\subsubsection{Outline of the construction}
To make use of Lemma \ref{2vwb} we would like to find a measurable isomorphism $\phi:\Sigma^+(x)\rightarrow\Sigma^+(\bar x)$, which would cause the fiber coordinates of $T^i(\phi(x),\bar z)\ \text{and}\ T^i(x,z)$ to get as close to each other as possible so as to maximize the chance that they are in the same atom of $\mathcal{P}\times\mathcal{Q}$. Since $\frac{S_n\varphi}{\sqrt{n}}$ follows a normal distribution (Lemmas \ref{31clt} and \ref{31llt}), the best we could possibly do when mapping $y\in \Sigma^+(x)$ to $\bar y\in\Sigma^+(\bar x)$ with $S_n\varphi(y)-S_n\varphi(\bar y)=t$, is to achieve that for a set of good control $D$ with $\mu_x^+(D)\sim\frac{t}{\sqrt{n}}$. This is what we achieve through the construction in the next subsection (up to a factor of $\log^{100}n$, which from the point of view of our application is irrelevant). One way to approach this problem, which appeared in \cite{7}, is to cut the space $\Sigma^+(x)$ into small sets of the form  $AB(i):=(S_n\varphi)^{-1}([i-\frac{1}{2}\delta,i+\frac{1}{2}\delta])$ and then map each $AB(i)$ into $AB(i+\frac{t}{\delta})$. In a relevant range (cutting off the tails of the distribution) there is roughly $\sqrt{n}$ of such $A_i$, each of $\mu_x^+$ size $\sim\frac{1}{\sqrt{n}}$. The $o(1/\sqrt{n})$ error in MLLT (Lemma \ref{31llt}) implies then an $\epsilon$ loss of measure, i.e. the set of good control $D=\{y\in\Sigma^+(x):\quad |S_n\varphi(y)-S_n\varphi(\bar y)-t|\le\delta\}$ satisfies $\mu_x^+(D)=1-o(1)$. This is already enough to use Lemma \ref{2vwb} and prove VWB for some systems. For some systems however, when nearby points diverge too fast, we would like to use $\phi$ repeatedly at different points in time with different $t$. This is what we do in Sections 4 and 5. The asymptotic size of the set of good control $D$ becomes then crucially important. We willl need to make sure that at each step when we are defining a map on some part of $\Sigma^+(x)$, the domain intersects the cylinders $[w_0,\dots,w_{L(N)}]$ in the same proportions as $\Sigma^+(x)$, where $L(N)$ is some large time beyond our considerations in this section. This leads to the definition of $(k,N)*$ sets which is needed to apply the results from previous subsection to the subsequent iterations of the matchings we will describe, here Lemma \ref{3past} plays a critical role.

In Lemma \ref{32main}. we do not specify a single set of good control immediately. We make a distinction between a set of good control $D_0$ and a set of very good control $D_1$. The whole construction assumes sufficiently big $n$ which is the size of the block in $\Sigma$ on which the construction essentially takes place. To find our map between two $(k,N)*$ sets $A\subset\Sigma^+(x)$ and $\bar A\subset\Sigma^+(\bar x)$, we first consider their partitions with respect to Birkhoff sums $S_m\varphi$ at an intermediate time $m=n/\log n$ $$A=AB_{m}(-\infty,-\sqrt{m}\log m)\sqcup AB_{m}[-\sqrt{m}\log m,-l\sqrt{m})\sqcup AB_{m}[-l\sqrt{m},l\sqrt{m})\\
        \sqcup AB_{m}[l\sqrt{m},\sqrt{m}\log m)\sqcup AB_{m}[\sqrt{m}\log m,\infty),$$
        $$ \bar A=\bar AB_{m}(-\infty,-\sqrt{m}\log m)\sqcup \bar AB_{m}[-\sqrt{m}\log m,-l\sqrt{m})\sqcup \bar AB_{m}[-l\sqrt{m},l\sqrt{m})\\
        \sqcup \bar AB_{m}[l\sqrt{m},\sqrt{m}\log m)\sqcup \bar AB_{m}[\sqrt{m}\log m,\infty),$$
        the leftmost and rightmost sets are negligible as implied by Lemma \ref{31clt}. 
        
The map between $AB_{m}[-\sqrt{m}\log m,-l\sqrt{m})$ and $\bar AB_{m}[-\sqrt{m}\log m,-l\sqrt{m})$ as well as between $AB_{m}[\sqrt{m}\log m,l\sqrt{m})$ and $\bar AB_{m}[-\sqrt{m}\log m,-l\sqrt{m})$ is essentially identity. Due to the lack of product measure we need to use exponential equidistribution (Lemma \ref{3equi}) along with local product structure to first condition on some cylinder of length $\sim\log n$ and then copy the remaining symbols. These sets fall into $D_0$ and they are the reason why in \eqref{32mainstatement1} we have $|t|$ on the RHS.

The map between $AB_{m}[-l\sqrt{m},l\sqrt{m})$ and $\bar AB_{m}[-l\sqrt{m},l\sqrt{m})$ is constructed in 2 steps given by Sublemma \ref{32main1} and Sublemma \ref{32main2}. In Sublemma \ref{32main1} we condition on a future cylinder $[w_{k+m-s},\dots,w_{k+m}]$ of length $s$, where $s$ is some large constant. Sublemma \ref{32main2} then is concerned with subsets of $[w_{k+m-s},\dots,w_{k+m}]$, which as we will see are taken to be the atoms of partitions $\mathcal{P}_w$ from Sublemma \ref{32main1}. The role of Sublemma \ref{32main1} is to shift most Birkhoff sums by $t$ and then Sublemma \ref{32main2}. is just as close to identity as we could get without the product measure.

More specifically in Sublemma \ref{32main1} we consider two types of partitions. The first, already mentioned, into subsets $AB(i)$, corresponds to cutting the graph of the cumulative distribution function $f(y)=\mu_x^+(\{w\in A:\quad \frac{S_m\varphi(w)}{\sqrt{m}}\le y\})$ using vertical lines and then looking at how the resulting probabilities compare to those of $\mathcal{N}(0,\rho^2)$, this is \eqref{32main1proof1}. Mapping $AB(i)$ to $AB(i+\frac{t}{\delta})$ then gives precise control on the Birkhoff sums and the points at which $\phi^{t}_w$ is defined this way comprise the set of very good control $D_1^w$ in Sublemma \ref{32main1}. As mentioned this method leaves behind a set of $\mu_x^+$ size $o(1)$. 

For that reason we use a second type of partition (Subsublemma \ref{32main1s}). Namely we cut 
the graph of the CDF of the sums $\frac{S_m\varphi}{\sqrt{m}}$ using horizontal lines as if we assumed the distribution was normal and wanted to obtain the partition $(S_m\varphi)^{-1}([i-\frac{\log^2m}{2},i+\frac{\log^2m}{2}])$. Then as a consequence of Lemma \ref{31clt} it follows that we in fact nearly obtain the subsets $(S_m\varphi)^{-1}([i-\frac{\log^2m}{2},i+\frac{\log^2m}{2}])$, this is \eqref{32main1sfirstclaim1} and \eqref{32main1sfirstclaim1'}. We then show how mapping the atoms of the first partition can be extended to mapping the atoms of the second partition. Because of \eqref{32main1sfirstclaim1} and \eqref{32main1sfirstclaim1'} and the relatively large $\log^2m$ size of the atoms of the second partition, the first partition forms its near refinement (with an $\epsilon^3$ loss). Cutting along horizontal lines allows for the optimal loss of measure $\sim\frac{t}{\sqrt{m}}$, which is how we get a large set of good control $D_0^w$ \eqref{32main1statement1}.  

Finally, in Lemma \ref{32final} we use exponential rate of convergence in the weak law of large numbers (Azuma's inequality) to specify a single set of good control $D$ which is obtained by repeatedly using Lemma \ref{32main} (order of $\log n$ times). In $D$ we only take those $y\in\Sigma^+(x)$ which satisfy two criteria. First, they stayed in the sets of good control $D_0$ through all $\sim \log n$ iterations and second, they stayed in the set of very good control $D_1$ in at least $1-\epsilon$ proportion of iterations. By \eqref{32mainstatement1} each iteration will be used with a time $t$ increased by at most $\log^6n$ and therefore not increasing the error at each step significantly. The first criterion for $D$ multiplies the loss in measure by at most $\log^k n$ for some $k$ and the second criterion can be made to only add an exponentially small error. As a result for such $D$ \eqref{32mainstatement2} and \eqref{32mainstatement3} will be satisfied for an arbitrarily large proportion of times and we will be able to get rid of $|t|$ on the RHS of \eqref{32mainstatement1}.

\subsubsection{The construction}
Let $L(N):=N^2\log^{210}N+N^2\log^{211}N$ ($L(N)$ will be the time we search for in VWB)
\begin{df}
    For a measurable $A\subset\Sigma^+(x)$ and an integer $k$ we say that $A$ is $(k,N)*$ if 
    $\exists y\in\Sigma^+(x)$, such that $A\subset\Sigma^+_k(y)$ and $\forall w\in\Sigma$  
    $\mu_y^k(A\cap[w_0,\dots,w_{L(N)}])=\mu_y^k([w_0,\dots,w_{L(N)}]\cap\Sigma^+_k(y))\mu_y^k(A)$
\end{df}

\begin{remark}
    For any SFT the whole $\Sigma^+(x)$ is always a $(0,N)*$ set. In the case of Bernoulli shifts any subset of $\Sigma_k^+(y)$, where $y\in\Sigma^+(x)$, given by fixing some coordinates beyond $L(N)$, i.e. 
    $\Sigma_k^+(y)\cap[w_{L(N)+k},\dots w_{L(N)+l}],$
    where
    $0<k<l$, is a $(k,N)*$ subset.
\end{remark}
\begin{remark}\label{32rem}
    If $A$ is $(k,N)*$, $A\subset\Sigma_k^+(y)$ then we have a similar identity also for shorter cylinders ie. if $0\le l\le d<L(N)$ then for any $w\in\Sigma$ the cylinder $[w_l,\dots,w_d]$ is a union of cylinders of length $L(N)+1$ and so it follows that we also have
    $\mu_y^k(A\cap[w_l,\dots,w_{d}])=\mu_y^k([w_l,\dots,w_{d}]\cap\Sigma^+_k(y))\mu_y^k(A)$.
\end{remark}
\begin{lemma}\label{32knstar1}
    For any  $k,n>0$ with $k+n\le L(N)$, $\ B\subset\Sigma^+(x)$ which is a $(k,N)*$ subset, there is a partition of $B$ into $(k+n,N)*$ subsets. Any set of the form $B\cap [w_0\ldots w_{k+n}]$ is $(k+n,N)*$
\end{lemma}
\begin{proof}
     Suppose $B\subset\Sigma^+_k(y)$ for some $y\in\Sigma^+_k(x)$. It is sufficient to consider the partition given by cylinders with first $k+n$ symbols fixed
    $B=\bigcup_{w\in \Sigma^+_k(y)}B\cap [w_0\ldots w_{k+n}]$.
    We will prove that for any $w\in\Sigma_k^+(y)$ the intersection $B\cap [w_0\ldots w_{k+n}]$ is a $(k+n,N)*$ set. 
    By the definition we want to show that for any $z\in \Sigma$ 
$$\mu_w^{k+n}(B\cap [w_0\ldots w_{k+n}]\cap[z_0,\dots,z_{L(N)}])=\mu_w^{k+n}([z_0,\dots,z_{L(N)}]\cap\Sigma^+_{k+n}(w))\mu_w^{k+n}(B\cap [w_0\ldots w_{k+n}]).$$
 Applying Lemma \ref{2trans}. to $U= B\cap [w_0\ldots w_{k+n}]\cap[z_0,\dots,z_{L(N)}]$ the sum on the RHS in in this lemma consists of a single number, which yields
\begin{equation}\label{32knstar1.1}
    \mu_y^k(B\cap [w_0\ldots w_{k+n}]\cap[z_0,\dots,z_{L(N)}])=\mu_y^k(\Sigma_k^+(y)\cap [w_0\ldots w_{k+n}])\mu_w^{k+n}(B\cap [z_0,\dots,z_{L(N)}]\cap\Sigma^+_{k+n}(w)).
\end{equation}
Analogously we obtain
\begin{equation} \label{32knstar1.2}
    \mu_y^k(\Sigma_{k+n}^+(w)\cap[z_0,\dots,z_{L(N)}])=\mu_y^k(\Sigma_k^+(x)\cap [w_0\ldots w_{k+n}])\mu_w^{k+n}([z_0,\dots,z_{L(N)}]\cap\Sigma^+_{k+n}(w))
\end{equation}

\begin{equation}\label{32knstar1.3}
    \mu_y^k(B\cap [w_0\ldots w_{k+n}])=\mu_y^k([w_0\ldots w_{k+n}]\cap\Sigma_k^+(y))\mu_w^{k+n}(B\cap [w_0\ldots w_{k+n}]).
\end{equation}
Since $w\in\Sigma_k^+(y)$, we have $B\cap [w_0\ldots w_{k+n}]\cap[z_0,\dots,z_{L(N)}]=B\cap [z_0,\dots,z_{L(N)}]\cap\Sigma^+_{k+n}(w)$ and we can write
\begin{equation*}
\begin{split}
    \mu_w^{k+n}(B\cap [w_0\ldots w_{k+n}]\cap[z_0,\dots,z_{L(N)}])\stackrel{\eqref{32knstar1.1}}{=}\frac{\mu_y^{k}(B\cap [w_0\ldots w_{k+n}]\cap[z_0,\dots,z_{L(N)}])}{\mu_y^k(\Sigma_k^+(y)\cap [w_0\ldots w_{k+n}])}\\
    \stackrel{B\ is\ (k,N)*}{=}\frac{\mu_y^k(\Sigma_k^+(y)\cap[z_0,\dots,z_{L(N)}])\mu_y^k(B)}{\mu_y^k(\Sigma_k^+(y)\cap [w_0\ldots w_{k+n}])}\stackrel{B\ is\ (k,N)*}{=}\\\frac{\mu_y^k(\Sigma_k^+(y)\cap[z_0,\dots,z_{L(N)}])\mu_y^k(B\cap[w_0\ldots w_{k+n}])}{\mu_y^k(\Sigma_k^+(y)\cap [w_0\ldots w_{k+n}])^2}\stackrel{by\ \eqref{32knstar1.2},\eqref{32knstar1.3}}{=}\\
    \mu_w^{k+n}([z_0,\dots,z_{L(N)}]\cap\Sigma^+_{k+n}(w))\mu_w^{k+n}(B\cap [w_0\ldots w_{k+n}])
\end{split}
\end{equation*}
    
\end{proof}
\begin{lemma}\label{32knstar2}
    For any  $k>0$ with $k\le L(N)$, $\ B\subset\Sigma^+(x)$ which is a $(k,N)*$ subset with $\mu_x^+(B)>0$ and for any finite sequence of non-negative numbers $(p_j)$ with $\sum_j p_j=1$, there is a partition of $B$ into $(k,N)*$ subsets $C_j$ which satisfy
    $\mu_y^k(C_j)=p_j\mu_y^k(B)$
\end{lemma}
\begin{proof}
    Let $y\in\Sigma_k^+(x)$ be such that $B\subset \Sigma_k^+(y)$ and note that $B=\bigcup_{w\in \Sigma^+_k(y)}B\cap [w_0\ldots w_{L(N)}]$ (where the sum is finite since there are finitely many cylinders of length $L(N)$). For each $w\in \Sigma_k^+(y)$ the set $B\cap [w_0\ldots w_{L(N)}]$ is a standard Borel probability space and so for any non-negative  numbers $(p_j)$ with $\sum_j p_j=1$, there is a partition of $B\cap [w_0\ldots w_{L(N)}]$ into sets $\{C^{w}_j\}$, where $\mu_y^k(C^{w}_j)=p_j\mu_y^k(B\cap [w_0\ldots w_{L(N)}])$. Define $C_j= \bigcup_{w\in\Sigma_k^+(y)}C^{w}_j$. Then $\mu_y^k(C_j)= p_j\mu_y^k(B)$ and we claim that if $B$ is $(k,N)*$ then $C_j$ are also $(k,N)*$.
    Indeed, we can assume $ z\in\Sigma^+_k(y)$  and we want to show that 
$$\mu_y^{k}(C_j\cap[z_0,\dots,z_{L(N)}])=\mu_y^{k}([z_0,\dots,z_{L(N)}]\cap\Sigma^+_k(y))\mu_y^{k}(C_j),$$
but this follows easily from the definition of $C_j$ and the fact that $B$ is a $(k,N)*$ subset since both sides are equal to $\mu_y^k(C_j^z)=p_j\mu_y^k(B\cap[z_0,\dots,z_{L(N)}])=p_j\mu_y^k(B)\mu_y^k([z_0,\dots,z_{L(N)}]\cap\Sigma_k^+(y))$.
\end{proof}
\begin{lemma}\label{32knstar3}
   For any  $k,n>0$ with $k+n\le L(N)$, $\ A\subset\Sigma^+(x),\ \bar A\subset\Sigma^+(\bar x)$ such that $A$ is partitioned into finitely many $(k,N)*$ subsets $A=\bigcup_{i\in I}B_i$ and $\bar A$ is partitioned into finitely many $(k,N)*$ subsets $\bar A=\bigcup_{i\in \bar I}\bar B_i$, there are finite refinements $C,\ \bar C$ of these partitions into $(k+n,N)*$ subsets $A=\bigcup_{j\in J}C_j$ and $\bar A=\bigcup_{j\in \bar J}\bar C_j$ and subsets $A_0\subset A,\ \bar A_0\subset A$ with $\mu_x^+(A_0)=\mu_{\bar x}^+(\bar A_0)=\min(\mu_x^+(A),\mu_{\bar x}^+(\bar A))$ such that $A_0=\bigcup_{j\in J_0}C_j$ and $\bar A_0=\bigcup_{j\in \bar J_0}\bar C_j$ and there is a measurable isomorphism $\phi:A_0\rightarrow\bar A_0$ such that $\phi(C)=\bar C$, that is for any $e\in \bar J_0$ there is exactly one $j\in J_0$ with $\phi(C_j)=\bar C_e$
\end{lemma}
\begin{proof}
First we use Lemma \ref{32knstar1}. applied to every atom $B_i$ and every $\bar B_i$ to obtain refinements $D,\ \bar D$ of given partitions into $(k+n,N)*$ sets. Then we use Lemma \ref{32knstar2} for every atom of these refinements to find the partitions $C,\ \bar C$. Since, as standard Borel probability spaces, $A$ and $\bar A$ are isomorphic to some subintervals of $[0,1]$ of Lebesgue measures $\mu_x^+(A)$ and $\mu_{\bar x}^+(\bar A)$ respectively, we only show how to find such refinements and a map $\phi$ in this context.
\begin{sublemma}
    Suppose we have two increasing sequences of points in $a_1,\dots ,a_{k}$ and $b_1,\dots, b_l$, $a_k\le b_l$, which partition the intervals $[0,a_k)$ and $[0,b_l)$ into subintervals $[a_i,a_{i+1})$ and $[b_i,b_{i+1})$ respectively. Then we can find finite refinements of these partitions given by two increasing sequences $A_1,\dots, A_n$ and $B_1,\dots, B_m$ such that for all $i\le n$
    $A_i-A_{i-1}=B_i-B_{i-1}$
\end{sublemma}
\begin{proof}
    Let $\{p_1,\dots,p_d\}$=$\{a_1,\dots, a_k\}\cup\{b_1,\dots b_l\}$ be such that $p_1\le\dots\le p_d$. Then it is enough to take $A_i$=$p_i$ for $p_i\le a_k$ and $B_i=p_i$ for all $i$.
\end{proof}
In the situation from this sublemma $\phi:[0,A_n)\rightarrow[0,B_n)$ equal to identity  would have the desired properties but we see that the existence of such maps depends only on the measures of atoms of $C$ and $\bar C$ and therefore by the use of Lemma \ref{32knstar2}. it is also possible in our case.
\end{proof}

\begin{lemma}\label{32main}
    $\forall\epsilon\ \forall\delta\ \exists n_1$ such that for $n\ge n_1$, any $N>0,\ \sqrt{n}/\log n\ge |t|$ and $A\subset\Sigma^+(x),\ \bar A\subset\Sigma^+(\bar x)$ which are $(k,N)*$ with $k+n\le L(N)$  and $\mu_x^+(A)=\mu_{\bar x}^+(\bar A)$ there are finite partitions $\mathcal{P},\ \mathcal{\bar P}$  of $A,\ \bar A$ into $((k+n),N)*$ sets and a measurable isomorphism $\bar\phi^{t,n}_{A,\bar A}:A\rightarrow\bar A$ such that \\
    1. $\bar\phi_{A,\bar A}^{t,n}(\mathcal{P})=\mathcal{\bar P}$\\
    2. there is a set $$D_0:=\bigcup_{\text{good}\ P_i\in \mathcal{P}}P_i\text{\quad with\quad}\mu_x^+(D_0)\ge \left(1-\log^{10}n\frac{|t|+\log^3 n}{\sqrt{n}}\right)\mu_x^+(A)$$
    such that for any $y\in D_0$
    \begin{equation}\label{32mainstatement1}
        |S_n\varphi(\sigma^k(\bar\phi^{t,n}_{A,\bar A}(y)))-S_n\varphi(\sigma^k(y))-t|\le |t|+\log^6 n
    \end{equation}
    3. there is a set $$D_1:=\bigcup_{\text{very good}\ P_i\in \mathcal{P}}P_i\text{\quad with\quad}\mu_x^+(D_1)\ge (1-\epsilon)\mu_x^+(A)$$
    such that for any $y\in D_1$
    \begin{equation}\label{32mainstatement2}
        |S_n\varphi(\sigma^k(\bar\phi^{t,n}_{A,\bar A}(y)))-S_n\varphi(\sigma^k(y))-t|\le 10\delta 
    \end{equation}
    and for $1-\epsilon$ proportion of $ 1\le r\le n$ we have
    \begin{equation}\label{32mainstatement3}
        |S_r\varphi(\sigma^k(\bar\phi^{t,n}_{A,\bar A}(y)))-S_r\varphi(\sigma^k(y))-t|\le 10\delta\text{\quad and\quad}(\bar\phi^{t,n}_{A,\bar A}(y))_{k+r}=y_{k+r}
    \end{equation}

\end{lemma}
\begin{proof}
By the definition of $(k,N)*$ there are $y\in\Sigma^+(x)$ and $\bar y\in\Sigma^+(\bar x)$ such that $A\subset\Sigma^+_k(y)$ and $\bar A\subset\Sigma_k^+(\bar y)$. For simplicity we will assume in this proof that $y=x$ and $\bar y=\bar x$.
Let 
$$AB_r[a,b):=\{y\in A:S_r\varphi(\sigma^k(y)\in[a,b)\}$$
$$\bar AB_r[a,b):=\{y\in \bar A:S_r\varphi(\sigma^k(y)\in[a,b)\}$$
Take $l>0$ such that $\mathfrak{R}(l)\ge1-\epsilon^{10}$. We will write $m$ for $m(n):=n/\log n$. Notice that 
\begin{equation*}
    \begin{split}
        A=AB_{m}(-\infty,-\sqrt{m}\log m)\sqcup AB_{m}[-\sqrt{m}\log m,-l\sqrt{m})\sqcup AB_{m}[-l\sqrt{m},l\sqrt{m})\\
        \sqcup AB_{m}[l\sqrt{m},\sqrt{m}\log m)\sqcup AB_{m}[\sqrt{m}\log m,\infty)
    \end{split}
\end{equation*}
and
\begin{equation*}
    \begin{split}
        \bar A=\bar AB_{m}(-\infty,-\sqrt{m}\log m)\sqcup \bar AB_{m}[-\sqrt{m}\log m,-l\sqrt{m})\sqcup \bar AB_{m}[-l\sqrt{m},l\sqrt{m})\\
        \sqcup \bar AB_{m}[l\sqrt{m},\sqrt{m}\log m)\sqcup \bar AB_{m}[\sqrt{m}\log m,\infty)
    \end{split}
\end{equation*}
Since $\varphi$ depends only on the past (Lemma \ref{3past}), it follows that $AB_r[a,b)$ and $\bar AB_r[a,b)$ are of the form $\bigcup_{w\in S}A\cap[w_k,\dots w_{k+r-1}]$ and $\bigcup_{w\in S}\bar A\cap[w_k,\dots w_{k+r-1}]$ respectively, where $S\subset\Sigma^+(x),\ \bar S\subset\Sigma^+(\bar x)$. Therefore, since $A$ and $\bar A$ are $(k,N)*$, we know from Remark \ref{32rem} that for $r+k\le L(N)$ 
\begin{equation}\label{32mainfirst1}
    \mu_x^k(AB_r[a,b))=\mu_x^k(\{y\in \Sigma_k^+(x):S_r\varphi(\sigma^k(y)\in[a,b)\})\mu_x^k(A)
\end{equation}
and
\begin{equation}\label{32mainfirst1'}
    \mu_{\bar x}^k(\bar AB_r[a,b))=\mu_{\bar x}^k(\{y\in \Sigma_k^+(\bar x):S_r\varphi(\sigma^k(y)\in[a,b)\})\mu_{\bar x}^k(\bar A)
\end{equation}
Then Lemma \ref{31clt} in particular tells us that 
\begin{equation}\label{32mainfirst2}
    \mu_x^k(AB_{m}(-\infty,-\sqrt{m}\log m)\sqcup AB_{m}[\sqrt{m}\log m,\infty))\le \frac{\log m}{\sqrt{m}}\mu_x^k(A),
\end{equation}
\begin{equation}\label{32mainfirst2'}
    \mu_{\bar x}^k(\bar AB_{m}(-\infty,-\sqrt{m}\log m)\sqcup \bar AB_{m}[\sqrt{m}\log m,\infty))\le \frac{\log m}{\sqrt{m}}\mu_{\bar x}^k(\bar A),
\end{equation}
for big $n$ so we can discard these sets.
Another similar property of $(k,N)*$ which we recall from Remark \ref{32rem} concerns intersections with arbitrary cylinders: 
\begin{equation*}
    \mu_x^k(A\cap[w_r,\dots w_{r+s}])=\mu_x^k(A)\mu_x^k(\Sigma_k^+(x)\cap[w_r,\dots w_{r+s}])
\end{equation*}
and 
\begin{equation*}
    \mu_x^k(\bar A\cap[w_r,\dots w_{r+s}])=\mu_x^k(\bar A)\mu_x^k(\Sigma_k^+(\bar x)\cap[w_r,\dots w_{r+s}]),
\end{equation*}
if $r+s\le L(N)$.
We first focus on defining our map between $AB_{m}[-\sqrt{m}\log m,-l\sqrt{m})$ and
$\bar AB_{m}[-\sqrt{m}\log m,-l\sqrt{m})$. 
Let $C:= \sup|\varphi|$ and recall that $\beta,d>0$ come from Lemma \ref{2hol} and Lemma \ref{3equi}. 
Notice that for any $w\in \Sigma$, $m_1\ge k+\frac{d}{\beta}\log n$ and $y\in [w_{k+\frac{d}{\beta}\log n},\dots,w_{m_1}]$ and $r\le m_1-\frac{d}{\beta}\log n-k$, by Lemma \ref{3geo} we get
\begin{equation}\label{32mainfirst3}
    |S_{r}\varphi(\sigma^{k+\frac{d}{\beta}\log n}(y))-S_{r}\varphi(\sigma^{k+\frac{d}{\beta}\log n}(w))|\le C_1,
\end{equation}

where $C_1>0$ is a constant depending only on $\varphi$.
Consider the set
$$W:=\{w\in \Sigma: S_{m-\frac{d}{\beta}\log n}\varphi(\sigma^{k+\frac{d}{\beta}\log n}(w))\in[-\sqrt{m}\log m+C\frac{d}{\beta}\log  n+C_1,-l\sqrt{m}-C\frac{d}{\beta}\log n-C_1]\}$$

 Then
$$AB_{m}[-\sqrt{m}\log m+2C\frac{d}{\beta}\log  n+C_1,-l\sqrt{m}-2C\frac{d}{\beta}\log n-C_1)\subset\bigcup_{w\in W}A\cap[w_{k+\frac{d}{\beta}\log n},\dots,w_{m}]\subset AB_{m}[-\sqrt{m}\log m,-l\sqrt{m})$$
Indeed, if $y\in \Sigma$ is in the LHS set, then by
$$S_{m-\frac{d}{\beta}\log n}\varphi(\sigma^{k+\frac{d}{\beta}\log n}(y))=S_{m}\varphi(\sigma^{k}(y))-S_{\frac{d}{\beta}\log n}\varphi(\sigma^{k}(y))\in[S_{m}\varphi(\sigma^{k}(y))-C\frac{d}{\beta}\log  n,\ S_{m}\varphi(\sigma^{k}(y))+C\frac{d}{\beta}\log  n]$$
we see that $y\in W$, which proves the first inclusion. For $y$ from the set in the middle we have $y\in[w_{k+\frac{d}{\beta}\log n},\dots,w_{m}]$ for some $w\in W$. Then again by the definition of $C$
\begin{equation*}
    \begin{split}
       S_{m}\varphi(\sigma^{k}(y))=S_{m-\frac{d}{\beta}\log n}\varphi(\sigma^{k+\frac{d}{\beta}\log n}(y))+S_{\frac{d}{\beta}\log n}\varphi(\sigma^{k}(y))\\
\in[S_{m-\frac{d}{\beta}\log n}\varphi(\sigma^{k+\frac{d}{\beta}\log n}(y))-C\frac{d}{\beta}\log  n,\ S_{m-\frac{d}{\beta}\log n}\varphi(\sigma^{k+\frac{d}{\beta}\log n}(y))+C\frac{d}{\beta}\log  n], 
    \end{split}
\end{equation*}
which by \eqref{32mainfirst3} is contained in 
$$[S_{m-\frac{d}{\beta}\log n}\varphi(\sigma^{k+\frac{d}{\beta}\log n}(w))-C\frac{d}{\beta}\log  n-C_1,\ S_{m-\frac{d}{\beta}\log n}\varphi(\sigma^{k+\frac{d}{\beta}\log n}(w))+C\frac{d}{\beta}\log  n+C_1].$$
By the definition of $W$, the above set is contained in $AB_{m}[-\sqrt{m}\log m,-l\sqrt{m})$. 
In the same way we obtain
$$\bar AB_{m}[-\sqrt{m}\log m+2C\frac{d}{\beta}\log  n+C_1,-l\sqrt{m}-2C\frac{d}{\beta}\log n-C_1)\subset\bigcup_{w\in W}\bar A\cap[w_{k+\frac{d}{\beta}\log n},\dots,w_{m}]\subset \bar AB_{m}[-\sqrt{m}\log m,-l\sqrt{m}).$$
Hence, 
\begin{equation*}
    \begin{split}
       AB_{m}[-\sqrt{m}\log m,-l\sqrt{m})\setminus\bigcup_{w\in W}A\cap[w_{k+\frac{d}{\beta}\log n},\dots,w_{m}]\\
    \subset AB_{m}[-\sqrt{m}\log m,-\sqrt{m}\log m+2C\frac{d}{\beta}\log  n+C_1)\cup AB_{m}[-l\sqrt{m}-2C\frac{d}{\beta}\log n-C_1,-l\sqrt{m}).
    \end{split}
\end{equation*}
Applying Lemma \ref{31clt}, \eqref{32mainfirst1}, \eqref{32mainfirst1'} to these two sets, we see that there is a constant $C_2>0$ such that for all $n$
\begin{equation}\label{32mainfirst4}
\begin{split}
    \mu_x^k(AB_{m}[-\sqrt{m}\log m,-l\sqrt{m})\setminus\bigcup_{w\in W}A\cap[w_{k+\frac{d}{\beta}\log n},\dots,w_{m}])\\
    \le (\mathfrak{R}(\frac{-\sqrt{m}\log m+2C\frac{d}{\beta}\log  n+C_1}{\sqrt{m}})-\mathfrak{R}(\frac{-\sqrt{m}\log m}{\sqrt{m}})+\mathfrak{R}(\frac{-l\sqrt{m}}{\sqrt{m}})-\mathfrak{R}(\frac{-l\sqrt{m}-2C\frac{d}{\beta}\log n-C_1}{\sqrt{m}}))\mu_x^k(A)
    \\
    \le\frac{C_2\log n}{\sqrt{m}}\mu_x^k(A)
\end{split}   
\end{equation}

and similarly
    \begin{equation}\label{32mainfirst5}
        \mu_{\bar x}^k(\bar AB_{m}[-\sqrt{m}\log m,-l\sqrt{m})\setminus\bigcup_{w\in  W}\bar A\cap[w_{k+\frac{d}{\beta}\log n},\dots,w_{m}])\le \frac{C_2\log n}{\sqrt{m}}\mu_{\bar x}^k(\bar A).
    \end{equation}
    
Lemma \ref{3equi} tells us that for any $w\in\Sigma$
\begin{equation}\label{32mainfirst6}
    \frac{\mu_x^k(\Sigma_k^+(x)\cap[w_{k+\frac{d}{\beta}\log n},\dots,w_{k+\frac{d}{\beta}\log n+\frac{1}{\beta}\log n}])}{\mu_{\bar x}^k(\Sigma_k^+(\bar x)\cap[w_{k+\frac{d}{\beta}\log n},\dots,w_{k+\frac{d}{\beta}\log n+\frac{1}{\beta}\log n}])}\in[1-\frac{1}{n^{1/\beta}},1+\frac{1}{n^{1/\beta}}].
\end{equation}
For $r>0$ we defined $H^r_{y,\bar y}:\Sigma^+_r(y)\rightarrow\Sigma^+_r(\bar y)$ in previous subsection. By Lemma \ref{2hol} we know that
\begin{equation}\label{32mainfirst7}
    \begin{split}
        \left|\frac{\mu^{k+\frac{d}{\beta}\log n+\frac{1}{\beta}\log n}_{\bar y}(H_{y,\bar y}^{k+\frac{d}{\beta}\log n+\frac{1}{\beta}\log n}(\Sigma^+_{k+\frac{d}{\beta}\log n+\frac{1}{\beta}\log n}(y)\cap[w_{k+\frac{d}{\beta}\log n},\dots,w_{k+n}]))}{\mu^{k+\frac{d}{\beta}\log n+\frac{1}{\beta}\log n}_y(\Sigma^+_{k+\frac{d}{\beta}\log n+\frac{1}{\beta}\log n}(y)\cap[w_{k+\frac{d}{\beta}\log n},\dots,w_{k+n}])}-1\right|\\
        \le C_0d(\sigma^{k+\frac{d}{\beta}\log n+\frac{1}{\beta}\log n}(y),\sigma^{k+\frac{d}{\beta}\log n+\frac{1}{\beta}\log n}(\bar y))^\beta\le C_0/n
    \end{split}
\end{equation}
for any $y\in A\cap[w_{k+\frac{d}{\beta}\log n},\dots,w_{k+\frac{d}{\beta}\log n+2\frac{1}{\beta}\log n}]$ and $\bar y\in \bar A\cap[w_{k+\frac{d}{\beta}\log n},\dots,w_{k+\frac{d}{\beta}\log n+2\frac{1}{\beta}\log n}]$, hence also for any $y\in A\cap[w_{k+\frac{d}{\beta}\log n},\dots,w_{k+\frac{d}{\beta}\log n+\frac{1}{\beta}\log n}]$ and $\bar y\in \bar A\cap[w_{k+\frac{d}{\beta}\log n},\dots,w_{k+\frac{d}{\beta}\log n+\frac{1}{\beta}\log n}]$ . 
Now we use Lemma \ref{2trans}. and Remark \ref{32rem} to write
\begin{equation}\label{32mainfirst8}
    \begin{split}
        \mu_x^k(A\cap[w_{k+\frac{d}{\beta}\log n},\dots,w_{k+n}])=\mu_x^k(A)\mu_x^k(\Sigma_k^+(x)\cap[w_{k+\frac{d}{\beta}\log n},\dots,w_{k+n}])\\
        =\mu_x^k(A)\sum_{y\in S}\mu_x^k(\Sigma_k^+(x)\cap[y_{k},\dots,y_{k+\frac{d}{\beta}\log n+\frac{1}{\beta}\log n}])\mu^{k+\frac{d}{\beta}\log n+\frac{1}{\beta}\log n}_y(\Sigma^+_{k+\frac{d}{\beta}\log n+\frac{1}{\beta}\log n}(y)\cap[w_{k+\frac{d}{\beta}\log n},\dots,w_{k+n}])
    \end{split}
\end{equation}
where $S$ is some set of $y\in\Sigma_k^+(x)$ such that 
$$\Sigma_k^+(x)\cap[w_{k+\frac{d}{\beta}\log n},\dots,w_{k+\frac{d}{\beta}\log n+\frac{1}{\beta}\log n}]=\bigsqcup_{y\in S}\Sigma_k^+(x)\cap[y_{k},\dots,y_{k+\frac{d}{\beta}\log n+\frac{1}{\beta}\log n}],$$
we also get
\begin{equation}\label{32mainfirst9}
    \begin{split}
        \mu_{\bar x}^k(\bar A\cap[w_{k+\frac{d}{\beta}\log n},\dots,w_{k+n}])= \mu_{\bar x}^k(\bar A) \mu_{\bar x}^k(\Sigma_k^+(\bar x)\cap[w_{k+\frac{d}{\beta}\log n},\dots,w_{k+n}])\\
        = \mu_{\bar x}^k(\bar A)\sum_{y\in \bar S}\mu_{\bar x}^k(\Sigma_k^+(\bar x)\cap[y_{k},\dots,y_{k+\frac{d}{\beta}\log n+\frac{1}{\beta}\log n}])\mu^{k+\frac{d}{\beta}\log n+\frac{1}{\beta}\log n}_y(\Sigma^+_{k+\frac{d}{\beta}\log n+\frac{1}{\beta}\log n}(y)\cap[w_{k+\frac{d}{\beta}\log n},\dots,w_{k+n}])
    \end{split}
\end{equation}
where $\bar S$ is some set of $y\in\Sigma_k^+(\bar x)$ such that 
$$\Sigma_k^+(\bar x)\cap[w_{k+\frac{d}{\beta}\log n},\dots,w_{k+\frac{d}{\beta}\log n+\frac{1}{\beta}\log n}]=\bigsqcup_{y\in \bar S}\Sigma_k^+(\bar x)\cap[y_{k},\dots,y_{k+\frac{d}{\beta}\log n+\frac{1}{\beta}\log n}].$$
Then by the definition of $S,\bar S$ and \eqref{32mainfirst6}
$$\frac{\sum_{y\in S}\mu_x^k(\Sigma_k^+(x)\cap[y_{k},\dots,y_{k+\frac{d}{\beta}\log n+\frac{1}{\beta}\log n}])}{\sum_{y\in \bar S}\mu_{\bar x}^k(\Sigma_k^+(\bar x)\cap[y_{k},\dots,y_{k+\frac{d}{\beta}\log n+\frac{1}{\beta}\log n}])}=
    \frac{\mu_x^k(\Sigma_k^+(x)\cap[w_{k+\frac{d}{\beta}\log n},\dots,w_{k+\frac{d}{\beta}\log n+\frac{1}{\beta}\log n}])}{\mu_{\bar x}^k(\Sigma_k^+(\bar x)\cap[w_{k+\frac{d}{\beta}\log n},\dots,w_{k+\frac{d}{\beta}\log n+\frac{1}{\beta}\log n}])}\in[1-\frac{1}{n^{1/\beta}},1+\frac{1}{n^{1/\beta}}].
$$
and \eqref{32mainfirst7} implies
\begin{equation*}
    \begin{split}
    \frac{\mu^{k+\frac{d}{\beta}\log n+\frac{1}{\beta}\log n}_y(\Sigma^+_{k+\frac{d}{\beta}\log n+\frac{1}{\beta}\log n}(y)\cap[w_{k+\frac{d}{\beta}\log n},\dots,w_{k+n}])}{\mu^{k+\frac{d}{\beta}\log n+\frac{1}{\beta}\log n}_{\bar y}(\Sigma^+_{k+\frac{d}{\beta}\log n+\frac{1}{\beta}\log n}(\bar y)\cap[w_{k+\frac{d}{\beta}\log n},\dots,w_{k+n}])}\\
=\frac{\mu^{k+\frac{d}{\beta}\log n+\frac{1}{\beta}\log n}_y(\Sigma^+_{k+\frac{d}{\beta}\log n+\frac{1}{\beta}\log n}(y)\cap[w_{k+\frac{d}{\beta}\log n},\dots,w_{k+n}])}{\mu^{k+\frac{d}{\beta}\log n+\frac{1}{\beta}\log n}_{\bar y}(H_{y,\bar y}^{k+\frac{d}{\beta}\log n+\frac{1}{\beta}\log n}(\Sigma^+_{k+\frac{d}{\beta}\log n+\frac{1}{\beta}\log n}(y)\cap[w_{k+\frac{d}{\beta}\log n},\dots,w_{k+n}]))}\in[1-C_0/n,1+C_0/n]
\end{split}
\end{equation*}
for any $y\in S,\ \bar y\in\bar S,$ 
which combined, after dividing \eqref{32mainfirst8} by \eqref{32mainfirst9}, gives us
\begin{equation}\label{32mainfirst10}
    \frac{\mu_x^k(A\cap[w_{k+\frac{d}{\beta}\log n},\dots,w_{k+n}])}{\mu_{\bar x}^k(\bar A\cap[w_{k+\frac{d}{\beta}\log n},\dots,w_{k+n}])}\in[(1-\frac{3C_0}{n})\frac{\mu_{x}^k( A)}{\mu_{\bar x}^k(\bar A)},(1+\frac{3C_0}{n})\frac{\mu_{x}^k( A)}{\mu_{\bar x}^k(\bar A)}].
\end{equation}
By Lemma \ref{32knstar1}. $A\cap[w_{k+\frac{d}{\beta}\log n},\dots,w_{k+n}]$ and $\bar A\cap[w_{k+\frac{d}{\beta}\log n},\dots,w_{k+n}]$ are both certain unions of $(k+n,N)*$ sets, so by Lemma \ref{32knstar3} we can refine these partitions (again into $(k+n,N)*$ sets) and get a map
$$\phi_w:A_w\rightarrow\bar A\cap[w_{k+\frac{d}{\beta}\log n},\dots,w_{k+n}]$$
defined on a subset $A_w\subset A\cap[w_{k+\frac{d}{\beta}\log n},\dots,w_{k+n}]$, $\mu_x^+(A_w)=\min(\mu_x^+(A\cap[w_{k+\frac{d}{\beta}\log n},\dots,w_{k+n}]),\mu^+_{\bar x}({\bar A}\cap[w_{k+\frac{d}{\beta}\log n},\dots,w_{k+n}]))$ with the associated partition $\mathcal{P}_w$ into $(k+n,N)*$ subsets and $\bar A_w:=\phi_w(A_w)$ with the associated partition $\mathcal{\bar P}_w$. By $\mu_x^+(A)=
\mu_{\bar x}^+(\bar A)$, Remark \ref{rem}, 2. and \eqref{32mainfirst10}
$$\frac{\mu_x^+(A\cap[w_{k+\frac{d}{\beta}\log n},\dots,w_{k+n}])}{\mu_{\bar x}^+(\bar A\cap[w_{k+\frac{d}{\beta}\log n},\dots,w_{k+n}])}\in[1-\frac{3C_0}{n},1+\frac{3C_0}{n}].$$
Thus
$$\mu_x^+(A_w)\ge (1-\frac{3C_0}{n})\mu_x^+(A\cap[w_{k+\frac{d}{\beta}\log n},\dots,w_{k+n}])$$
and $$\mu_{\bar x}^+(\bar A_w)\ge (1-\frac{3C_0}{n})\mu_{\bar x}^+(\bar A\cap[w_{k+\frac{d}{\beta}\log n},\dots,w_{k+n}])$$
so, again by Remark \ref{rem}, 1., we have
$$\mu_x^k(A_w)\ge (1-\frac{3C_0}{n})\mu_x^k(A\cap[w_{k+\frac{d}{\beta}\log n},\dots,w_{k+n}])$$
and $$\mu_{\bar x}^k(\bar A_w)\ge (1-\frac{3C_0}{n})\mu_{\bar x}^k(\bar A\cap[w_{k+\frac{d}{\beta}\log n},\dots,w_{k+n}])$$
Gluing $\phi_w$ for every $w\in W$ and combining partitions $\mathcal{P}_w$, $\mathcal{\bar P}_w$ into $\mathcal{P}_1,\ \mathcal{\bar P}_1$, we obtain $A_1:=\bigcup_{w\in W}A_w$, $\bar A_1:=\bigcup_{w\in W}\bar A_w$ and a measurable isomorphism 
$$\phi_1:A_1\rightarrow\bar A_1$$
such that $\phi_1(\mathcal{P}_1)=\mathcal{\bar P}_1,$
\begin{equation}\label{32mainfirst11}
    \mu_x^k(\bigcup_{w\in W}A\cap[w_{k+\frac{d}{\beta}\log n},\dots,w_{m}]\setminus A_1)\le\frac{3C_0}{n}\mu_x^k(\bigcup_{w\in W}A\cap[w_{k+\frac{d}{\beta}\log n},\dots,w_{m}])
\end{equation}
\begin{equation}\label{32mainfirst11'}
    \mu_{\bar x}^k(\bigcup_{w\in W}\bar A\cap[w_{k+\frac{d}{\beta}\log n},\dots,w_{m}]\setminus\bar A_1)\le\frac{3C_0}{n}\mu_x^k(\bigcup_{w\in W}\bar A\cap[w_{k+\frac{d}{\beta}\log n},\dots,w_{m}])
\end{equation}
and
\begin{equation}\label{32mainfirst12}
    \phi_1(A_1\cap[w_{k+\frac{d}{\beta}\log n},\dots,w_{n+k}])=\bar A_1\cap[w_{k+\frac{d}{\beta}\log n},\dots,w_{n+k}].
\end{equation}
\eqref{32mainfirst4} together with \eqref{32mainfirst11} and \eqref{32mainfirst5} together with \eqref{32mainfirst11'} imply that there is $C_4>0$ such that
\begin{equation}\label{32mainfirst13}
    \mu_x^k(AB_{m}[-\sqrt{m}\log m,-l\sqrt{m})\setminus A_1)\le \frac{C_4\log^3 n}{\sqrt{n}}\mu_x^k(A)
\end{equation}
and 
\begin{equation}\label{mainfirst13'}
        \mu_{\bar x}^k(\bar AB_{m}[-\sqrt{m}\log m,-l\sqrt{m})\setminus\bar A_1)\le \frac{C_4\log^3 n}{\sqrt{n}}\mu_{\bar x}^k(\bar A).
\end{equation}
For $y\in A_1$ we define $\bar\phi^{t,n}_{A,\bar A}(y):=\phi_1(y)$. Then \eqref{32mainfirst12} together with \eqref{32mainfirst3} applied to $m_1=n$ imply \eqref{32mainstatement1} for $y\in A_1$:
\begin{equation*}
    \begin{split}
       |S_{n}\varphi(\sigma^{k}(y))-S_{n}\varphi(\sigma^{k}(\bar\phi^{t,n}_{A,\bar A}(y)))|=|S_{n-\frac{d}{\beta}\log n}\varphi(\sigma^{k+\frac{d}{\beta}\log n}(y))+S_{\frac{d}{\beta}\log n}\varphi(\sigma^{k}(y))\\
       -(S_{n-\frac{d}{\beta}\log n}\varphi(\sigma^{k+\frac{d}{\beta}\log n}(\bar\phi^{t,n}_{A,\bar A}(y)))+S_{\frac{d}{\beta}\log n}\varphi(\sigma^{k}(\bar\phi^{t,n}_{A,\bar A}(y))))|\\
       \le|S_{n-\frac{d}{\beta}\log n}\varphi(\sigma^{k+\frac{d}{\beta}\log n}(y))
       -S_{n-\frac{d}{\beta}\log n}\varphi(\sigma^{k+\frac{d}{\beta}\log n}(\bar\phi^{t,n}_{A,\bar A}(y)))|
       +2C\frac{d}{\beta}\log n\le C_1+2C\frac{d}{\beta}\log n\le C_5\log n.
    \end{split}
\end{equation*}

Analogously, considering 
$AB_{m}[l\sqrt{m},\sqrt{m}\log m)$ and $\bar AB_{m}[l\sqrt{m},\sqrt{m}\log m)$, we obtain $A_2\subset AB_{m}[l\sqrt{m},\sqrt{m}\log m)$, $\bar A_2\subset \bar AB_{m}[l\sqrt{m},\sqrt{m}\log m)$, partitions $\mathcal{P}_2$ and $\mathcal{\bar P}_2$ of $A_2$ and $\bar A_2$ into $(k+n,N)*$ subsets and a measurable isomorphism $\phi_2:A_2\rightarrow\bar A_2$ such that $\phi_2(\mathcal{P}_2)=\mathcal{\bar P}_2$. Moreover
\begin{equation}\label{32mainfirst14}
    \mu_x^k(AB_{m}[l\sqrt{m},\sqrt{m}\log m)\setminus A_2)\le \frac{C_4 \log^3 n}{\sqrt{n}}\mu_x^k(A)
\end{equation}
and similarly
 \begin{equation}\label{32mainfirst14'}
        \mu_{\bar x}^k(\bar AB_{m}[l\sqrt{m},\sqrt{m}\log m)\setminus\bar A_2)\le \frac{C_4\log^3 n}{\sqrt{n}}\mu_{\bar x}^k(\bar A).
    \end{equation}
For $y\in A_2$ we define $\bar\phi^{t,n}_{A,\bar A}(y):=\phi_2(y)$ and then, just like before, we get \eqref{32mainstatement1} for $y\in A_2$.

Now we want to map $AB_{m}[-l\sqrt{m},l\sqrt{m})$ into $\bar AB_{m}[-l\sqrt{m},l\sqrt{m})$.
Fix $s>0$ such that $C_02^{-s\beta}<\epsilon^3$ and $s>s_1$ ($s_1$ from Lemma \ref{3geo}), $w\in\Sigma$ and consider the sets 
$$AB^w_{m}[a,b):=AB_{m}[a,b)\cap[w_{k+m-s},\dots,w_{k+m}]$$
$$\bar AB^w_{m}[a,b):=\bar AB_{m}[a,b)\cap[w_{k+m-s},\dots,w_{k+m}]$$
Then
$$AB_{m}[-l\sqrt{m},l\sqrt{m})=\bigcup_{w\in\Sigma}AB^w_{m}[-l\sqrt{m},l\sqrt{m})$$
and
$$\bar AB_{m}[-l\sqrt{m},l\sqrt{m})=\bigcup_{w\in\Sigma}\bar AB^w_{m}[-l\sqrt{m},l\sqrt{m}).$$
\begin{sublemma}\label{32main1}
    $\exists n_2$ (independent of $N,k, A,\bar A,w$) such that for $n\ge n_2$, $\epsilon^4\sqrt{m}\ge |t|$, ($m=n/\log n$) there are subsets of good control $D^w_0\subset AB^w_{m}[-l\sqrt{m},l\sqrt{m}),\ \bar D^w_0\subset\bar AB^w_{m}[-l\sqrt{m},l\sqrt{m})$ with $\mu_x^+(D^w_0)=\mu_{\bar x}^+(\bar D^w_0)$ and finite partitions $\mathcal{P}_w,\ \mathcal{\bar P}_w$  of $D^w_0,\bar D^w_0$ into $((k+m),N)*$ sets and there is a measurable isomorphism $\phi_w^t:D^w_0\rightarrow\bar D^w_0$ such that \\
    1. $\phi_{w}^{t}(\mathcal{P}_w)=\mathcal{\bar P}_w$\\
    2. $D^w_0$ satisfies
    $$\mu_x^k(D^w_0)\ge (1-\frac{10|t|+\log^{5} n}{\sqrt{m}})\mu_x^k(AB^w_{m}[-l\sqrt{m},l\sqrt{m}))$$
    and for any $y\in D^w_0$
    \begin{equation}\label{32main1statement1}
        |S_{m}\varphi(\sigma^k(\phi_{w}^t(y)))-S_{m}\varphi(\sigma^k(y))-t|\le \log^{5}n
    \end{equation}
    3. there is a set of very good control $$D^w_1:=\bigcup_{\text{very good}\ P_i\in \mathcal{P}_w}P_i\text{\quad with\quad}\mu_x^k(D^w_1)\ge (1-\epsilon^2)\mu_x^k(AB^w_{m}[-l\sqrt{m},l\sqrt{m}))$$
    such that for any $y\in D^w_1$
    \begin{equation}\label{32main1statement2}
        |S_{m}\varphi(\sigma^k(\phi_{w}^t(y)))-S_{m}\varphi(\sigma^k(y))-t|\le 5\delta 
    \end{equation}
\end{sublemma}
\begin{proof}
    We will consider two types of partitions.
    
    The first one into small sets where Birkhoff sums are controlled up to a small error $\delta$ specified in the statement of Lemma \ref{32main}. 
    Namely we let $AB^w_{m}(i):=AB^w_{m}[(i-\frac{1}{2})\delta,(i+\frac{1}{2})\delta)$ and $\bar AB^w_{m}(i):=\bar AB^w_{m}[(i-\frac{1}{2})\delta,(i+\frac{1}{2})\delta )$, for $i\in[ l\sqrt{m}/\delta+1, l\sqrt{m}/\delta-1]$. These form near partitions of the sets $AB^w_{m}[-l\sqrt{m},l\sqrt{m})$ and $\bar AB^w_{m}[-l\sqrt{m},l\sqrt{m})$ (leaving small ranges of Birkhoff sums near the endpoints of $[-l\sqrt{m},l\sqrt{m})$ uncovered, which as one may expect will not cause any problem).
    
By Lemma \ref{31llt} for big enough $n$, for all $w,\ x,\ \bar x\in\Sigma,\ k\in\mathbb{N},\ i\in[-l\sqrt{m}/\delta,l\sqrt{m}/\delta], |t|\le\epsilon^4\sqrt{m}$ and since $A, \bar A$ are $(k,N)*$ using the same observation as in \eqref{32mainfirst1} and \eqref{32mainfirst1'} we have
\begin{equation}\label{32main1proof1}
    \frac{\mu_x^k(AB^w_{m}(i))}{\mu_{\bar x}^k(\bar AB^w_{m}(i+\frac{t}{\delta}))}=\frac{\frac{\mathfrak{n}(i\delta)|\delta|}{\sqrt m}+o(1/\sqrt{m})}{\frac{\mathfrak{n}(i\delta+t)|\delta|}{\sqrt m}+o(1/\sqrt{m})}\in[(1-\epsilon^3)\frac{\mu_{x}^k( A)}{\mu_{\bar x}^k(\bar A)},(1+\epsilon^3)\frac{\mu_{x}^k( A)}{\mu_{\bar x}^k(\bar A)}].
\end{equation}
The second partition is with respect to probability.
\begin{subsublemma}\label{32main1s}
 There exists a partition
    $$A=\bigcup_{w\in \Sigma}\bigsqcup_{i\in\mathbb{Z}}A^w_{m}(i),$$
where $A^w_{m}(i)\subset\Sigma^+_k(x)\cap[w_{k+m-s},\dots,w_{k+m}]$ satisfy $$\frac{\mu_x^k(y\in[w_{k+m-s},\dots,w_{k+m}]\cap A\setminus A^w_{m}(i):\ S_{m}\varphi(\sigma^k(y))\le\inf_{z\in A^w_{m}(i)}  S_{m}\varphi(\sigma^k(z)))}{\mu_x^k(A\cap[w_{k+m-s},\dots,w_{k+m}])}\ge\mathfrak{R}(\frac{(i-\frac{1}{2})\log^2m}{\sqrt{m}})$$
$$\frac{\mu_x^k(A^w_{m}(i))}{\mu_x^k(A\cap[w_{k+m-s},\dots,w_{k+m}])}=\mathfrak{R}(\frac{(i+\frac{1}{2})\log^2 m}{\sqrt{m}})-\mathfrak{R}(\frac{(i-\frac{1}{2})\log^2m}{\sqrt{m}})$$
$$\frac{\mu_x^k(y\in[w_{k+m-s},\dots,w_{k+m}]\cap A\setminus A^w_{m}(i):\ S_{m}\varphi(\sigma^k(y))\ge\sup_{z\in A^w_{m}(i)}  S_{m}\varphi(\sigma^k(z)))}{\mu_x^k(A\cap[w_{k+m-s},\dots,w_{k+m}])}\ge1-\mathfrak{R}(\frac{(i+\frac{1}{2})\log^2m}{\sqrt{m}})$$
and for $i\le i'$ and $y\in A^w_{m}(i)$, $y'\in A^w_{m}(i')$
we have 
$$S_{m}\varphi(\sigma^k(y))\le S_{m}\varphi(\sigma^k(y'))$$
additionally each $A^w_{m}(i)$ is a union of $(k+m,N)*$ sets.
\end{subsublemma}
\begin{proof}
    We fix some $w$ for this proof. To show that such a partition exists
take $y^1,\dots,y^d\in\Sigma_k^+(x)$ such that $$A\cap[w_{k+m-s},\dots,w_{k+m}]=\bigsqcup_iA\cap[y^i_1,\dots,y_{k+m}^i]$$ and $S_{m}\varphi(\sigma^k(y^i))\le S_{m}\varphi(\sigma^k(y^{i+1}))$ for every $i$.
Then take $a_i\in\{1,\dots,d\}$ to be the biggest number such that $$\frac{\mu_x^k(\bigcup_{j\le a_i-1}A\cap[y^j_1,\dots,y^j_{k+m}])}{\mu_x^k(A\cap[w_{k+m-s},\dots,w_{k+m}])}\le\mathfrak{R}(\frac{(i-\frac{1}{2})\log^2m}{\sqrt{m}}).$$ By Lemma \ref{32knstar1} since $A$ is $(k,N)*$, $A\cap[y^{a_i}_1,\dots,y^{a_i}_{k+m}]$ is $(k+m,N)*$. By the definition of $a_i$
$$\frac{\mu_x^k(A\cap[y^{a_i}_1,\dots,y^{a_i}_{k+m}])}{\mu_x^k(A\cap[w_{k+m-s},\dots,w_{k+m}])}\ge\mathfrak{R}(\frac{(i-\frac{1}{2})\log^2m}{\sqrt{m}})-\frac{\mu^k_x(\bigcup_{j\le a_i-1}A\cap[y^j_1,\dots,y^j_{k+m}])}{\mu_x^k(A\cap[w_{k+m-s},\dots,w_{k+m}])}.$$
By Lemma \ref{32knstar2} we can therefore partition
$A\cap[y^{a_i}_1,\dots,y^{a_i}_{k+m}]$ into two $(k+m,N)*$ subsets $S$ and $D$ with $$\frac{\mu_x^k(S)}{\mu_x^k(A\cap[w_{k+m-s},\dots,w_{k+m}])}=\mathfrak{R}(\frac{(i-\frac{1}{2})\log^2m}{\sqrt{m}})-\frac{\mu^k_x(\bigcup_{j\le a_i-1}A\cap[y^j_1,\dots,y^j_{k+m}])}{\mu_x^k(A\cap[w_{k+m-s},\dots,w_{k+m}])}.$$
Similarly take $b_i\in\{1,\dots,d\}$ to be the biggest number such that $$\frac{\mu^k_x(\bigcup_{j\le b_i-1}A\cap[y^j_1,\dots,y^j_{k+m}])}{\mu_x^k(A\cap[w_{k+m-s},\dots,w_{k+m}])}\le\mathfrak{R}(\frac{(i+\frac{1}{2})\log^2m}{\sqrt{m}})$$ and partition
$A\cap[y^{b_i}_1,\dots,y^{b_i}_{k+m}]$ into two $(k+m,N)*$ subsets $U$ and $O$ such that $O\cap S=\emptyset$ (in the rare case when $A\cap[y^{b_i}_1,\dots,y^{b_i}_{k+m}]=A\cap[y^{a_i}_1,\dots,y^{a_i}_{k+m}]$ we apply Lemma \ref{32knstar2} to $D$ and partition it into $O$ and $O'$ and take $U=S\cup O'$) with $$\frac{\mu_x^k(U)}{\mu_x^k(A\cap[w_{k+m-s},\dots,w_{k+m}])}=\mathfrak{R}(\frac{(i+\frac{1}{2})\log^2m}{\sqrt{m}})-\frac{\mu^k_x(\bigcup_{j\le b_i-1}A\cap[y^j_1,\dots,y^j_{k+m}])}{\mu_x^k(A\cap[w_{k+m-s},\dots,w_{k+m}])}.$$
Then it is enough to let 
$$A^w_{m}(i)=[w_{k+m-s},\dots,w_{k+m}]\cap A\setminus((\bigcup_{j\ge b_i+1}A\cap[y^j_1,\dots,y^j_{k+m}])\cup(\bigcup_{j\le a_i-1}A\cap[y^j_1,\dots,y^j_{k+m}])\cup S\cup O).$$
In the case when $a_i\neq b_i$ the above is just
$$A^w_{m}(i)=(\bigcup_{a_i+1\le j\le b_i-1}A\cap[y^j_1,\dots,y^j_{k+m}])\cup D\cup U,$$
which as we noted, by Lemma \ref{32knstar1} and since $D$ and $U$ are $(k+m,N)*$, is a union of $(k+m,N)*$ sets. Otherwise it is just the mentioned $(k+m,N)*$ set $O'\subset D$.

It satisfies the required properties because the Birkhoff sums for $y_i$ are ordered i.e.
$S_{m}\varphi(\sigma^k(y))\le S_{m}\varphi(\sigma^k(y'))$ when $y\in\bigcup_{j\le a_i-1}A\cap[y^j_1,\dots,y^j_{k+m}])\cup S$ and $y'\in [w_{k+m-s},\dots,w_{k+m}]\cap A\setminus((\bigcup_{j\ge b_i+1}A\cap[y^j_1,\dots,y^j_{k+m}])\cup(\bigcup_{j\le a_i-1}A\cap[y^j_1,\dots,y^j_{k+m}])\cup S\cup O)$ (we use the fact that $\varphi$ depends only on the past here). Hence
\begin{equation*}
    \begin{split}
        \frac{\mu_x^k(y\in[w_{k+m-s},\dots,w_{k+m}]\cap A\setminus A^w_{m}(i):\ S_{m}\varphi(\sigma^k(y))\le\inf_{z\in A^w_{m}(i)}  S_{m}\varphi(\sigma^k(z)))}{\mu_x^k(A\cap[w_{k+m-s},\dots,w_{k+m}])}\\
        \ge\frac{\mu_x^k(\bigcup_{j\le a_i-1}A\cap[y^j_1,\dots,y^j_{k+m}])\cup S)}{\mu_x^k(A\cap[w_{k+m-s},\dots,w_{k+m}])}=\mathfrak{R}(\frac{(i-\frac{1}{2})\log^2m}{\sqrt{m}})
    \end{split}
\end{equation*}
by the definition of $S$.
Similarly 
\begin{equation*}
    \begin{split}
        \frac{\mu_x^k(y\in[w_{k+m-s},\dots,w_{k+m}]\cap A\setminus A^w_{m}(i):\ S_{m}\varphi(\sigma^k(y))\ge\sup_{z\in A^w_{m}(i)}  S_{m}\varphi(\sigma^k(z)))}{\mu_x^k(A\cap[w_{k+m-s},\dots,w_{k+m}])}\\
        \ge\frac{\mu_x^k(\bigcup_{j\ge b_i+1}A\cap[y^j_1,\dots,y^j_{k+m}])\cup O)}{\mu_x^k(A\cap[w_{k+m-s},\dots,w_{k+m}])}=1-\mathfrak{R}(\frac{(i+\frac{1}{2})\log^2m}{\sqrt{m}})
    \end{split}
\end{equation*}
\end{proof}

Analogously define $\bar A^w_{m}(\cdot)$. We will consider the partition
$$\bar A=\bigcup_{w\in \Sigma}\bigsqcup_{i\in\mathbb{Z}}\bar A^w_{m}(i+\frac{t}{\log^2m})$$
By Lemma \ref{31clt} for big enough $n$, for all $y\in A^w_{m}(i),\ \bar y\in \bar A^w_{m}(i),\ k\in\mathbb{N}\ (k+n\le L(N)),\ i\in[-\frac{l\sqrt{m}}{\log ^2m},\frac{l\sqrt{m}}{\log ^2m}]$ we have \\
CLAIM 1\\
\begin{equation}\label{32main1sfirstclaim1}
    |S_{m}\varphi(\sigma^k(y))-i\log^2m|< \frac{\log^2m}{2}+\log^{1.5}m
\end{equation}
and
\begin{equation}\label{32main1sfirstclaim1'}
    |S_{m}\varphi(\sigma^k(\bar y))-i\log^2m|< \frac{\log^2m}{2}+\log^{1.5}m
\end{equation}
Indeed, suppose that $|S_{m}\varphi(\sigma^k(y))-i\log^2m|\ge\frac{\log^2m}{2}+\log^{1.5}m$ and assume for example $S_{m}\varphi(\sigma^k(y))\le i\log^2m-\frac{\log^2m}{2}-\log^{1.5}m$,
then by the definition of $A^w_{m}(i)$ it would follow that \begin{equation}\label{32main1sfirstclaim2}
    \frac{\mu_x^k(y\in[w_{k+m-s},\dots,w_{k+m}]\cap A:\ S_{m}\varphi(\sigma^k(y))\le i\log^2m-\frac{\log^2m}{2}-\log^{1.5}m)}{\mu_x^k(A\cap[w_{k+m-s},\dots,w_{k+m}])}\ge \mathfrak{R}(\frac{(i-\frac{1}{2})\log^2m}{\sqrt{m}}).
\end{equation}
Then, since $\varphi$ depends only on the past and $A$ is $(k,N)*$ (as in observation \eqref{32mainfirst1}), by Lemma \ref{31clt} we know that
\begin{equation*}
    \begin{split}
        \mu_x^k(y\in[w_{k+m-s},\dots,w_{k+m}]\cap A:\ S_{m}\varphi(\sigma^k(y))\le i\log^2m-\frac{\log^2m}{2}-\log^{1.5}m)\\
        =\mu_x^k(A)\mu_x^k(y\in[w_{k+m-s},\dots,w_{k+m}]\cap \Sigma_k^+(x):\ S_{m}\varphi(\sigma^k(y))\le i\log^2m-\frac{\log^2m}{2}-\log^{1.5}m)\\
        =\mu_x^k(A)(\mu([w_{k+m-s},\dots,w_{k+m}])\mathfrak{R}(\frac{i\log^2m-\frac{\log^2m}{2}-\log^{1.5}m}{\sqrt{m}})+O(\log m/\sqrt{m})).
    \end{split}
\end{equation*}
Lemma \ref{3equi} combined with $A$ being $(k,N)*$ tells us in particular that
\begin{equation*}
    \begin{split}
        \mu_x^k(A\cap[w_{k+m-s},\dots,w_{k+m}])=\mu_x^k(A)\mu_x^k(\Sigma_k^+(x)\cap[w_{k+m-s},\dots,w_{k+m}])=\mu_x^k(A)\mu([w_{k+m-s},\dots,w_{k+m}](1+O(1/m)).
    \end{split}
\end{equation*}
Dividing both equations we get
\begin{equation*}
    \begin{split}
       \frac{\mu_x^k(y\in[w_{k+m-s},\dots,w_{k+m}]\cap A:\ S_{m}\varphi(\sigma^k(y))\le i\log^2m-\frac{\log^2m}{2}-\log^{1.5}m)}{\mu_x^k(A\cap[w_{k+m-s},\dots,w_{k+m}])}\\
       =\mathfrak{R}(\frac{i\log^2m-\frac{\log^2m}{2}-\log^{1.5}m}{\sqrt{m}})+O(\log m/\sqrt{m}), 
    \end{split}
\end{equation*}
\\
which for sufficiently big $m$ (depending only on $l$) is smaller than $\mathfrak{R}(\frac{(i-\frac{1}{2})\log^2m}{\sqrt{m}})$, contradicting \eqref{32main1sfirstclaim2}. Analogously we prove 
that  $S_{m}\varphi(\sigma^k(y))>i\log^2m+\frac{\log^2m}{2}+\log^{1.5}m$ is impossible and that \eqref{32main1sfirstclaim1'} is true, ending the proof of CLAIM 1.

Now we will show how both types of partitions essentially cover $AB^w_{m}[-l\sqrt{m},l\sqrt{m})$ and how the first one is essentially a refinement of the second one.

By \eqref{32main1sfirstclaim1} and \eqref{32main1sfirstclaim1'} 
\begin{equation}\label{32main1proof2}
    \bigcup_{i\in[-l\sqrt{m}/\log^2m+2,l\sqrt{m}/\log^2m-2]} A^w_{m}(i)\subset AB^w_{m}[-l\sqrt{m},l\sqrt{m}).
\end{equation}
and
\begin{equation}\label{32main1proof2'}
    \bigcup_{i\in[-l\sqrt{m}/\log^2m+2,l\sqrt{m}/\log^2m-2]} \bar A^w_{m}(i)\subset \bar AB^w_{m}[-l\sqrt{m},l\sqrt{m}).
\end{equation}
Notice that for any $y\in A^w_m(i')$ with $i'\le -l\sqrt{m}/\log^2m+2$, since the  Birkhoff sums for elements in the sets $A^w_m(i)$ are monotone in $i$ (the fourth condition in Subsublemma \ref{32main1s}) we have by \eqref{32main1sfirstclaim1} that
$$S_{m}\varphi(\sigma^k(y))\le-l\sqrt{m}+2\log^2m+\frac{\log^2 m}{2}+\log^{1.5}m$$
Similarly if $y\in A^w_m(i')$ with $i'\ge l\sqrt{m}/\log^2m-2$ then
$$S_{m}\varphi(\sigma^k(y))\ge l\sqrt{m}-2\log^2m-\frac{\log^2 m}{2}-\log^{1.5}m$$
and hence
\begin{equation*}
    AB^w_{m}[-l\sqrt{m}+4\log^2m,l\sqrt{m}-4\log^2m)\subset \bigcup_{i\in[-l\sqrt{m}/\log^2m+2,l\sqrt{m}/\log^2m-2]}A^w_m(i).
\end{equation*}

In general if $y\in AB_m^w(i\log^2 m-\frac{\log^2m}{2}+\log^{1.5} m,\ i\log^2m +\frac{\log^2m}{2}-\log^{1.5} m)$ for some $i\in[-l\sqrt{m}/\log^2m,l\sqrt{m}/\log^2m]$, then $y\in A_m^w(i)$ because otherwise we would have $y\in A_m^w(i')$ for some $i'\neq i$ and then either
\begin{equation*}
    S_{m}\varphi(\sigma^k(y))\ge i'\log^2m-\frac{\log^2 m}{2}-\log^{1.5}m\ge i\log^2m +\frac{\log^2m}{2}-\log^{1.5} m
\end{equation*}
or
\begin{equation*}
    S_{m}\varphi(\sigma^k(y))\le i'\log^2m+\frac{\log^2 m}{2}-\log^{1.5}m\le i\log^2m -\frac{\log^2m}{2}-\log^{1.5} m
\end{equation*}
depending on whether $i'>i$ or $i'<i$, so the sum falls out of the interval $(i\log^2 m-\frac{\log^2m}{2}+\log^{1.5} m,\ i\log^2m +\frac{\log^2m}{2}-\log^{1.5} m)$.
Thus we obtain for $i\in[-l\sqrt{m}/\log^2m,l\sqrt{m}/\log^2m]$ that
\begin{equation}\label{32main1proof3}
    AB_m^w(i\log^2 m-\frac{\log^2m}{2}+\log^{1.5} m,\ i\log^2m +\frac{\log^2m}{2}-\log^{1.5} m)\subset A_m^w(i)
\end{equation}
and similarly
\begin{equation}\label{32main1proof3'}
    \bar AB_m^w(i\log^2 m-\frac{\log^2m}{2}+\log^{1.5} m,\ i\log^2m +\frac{\log^2m}{2}-\log^{1.5} m)\subset \bar A_m^w(i).
\end{equation}
Analogously to \eqref{32main1proof2} and \eqref{32main1proof2'} we have
\begin{equation*}
    \bigcup_{i\in[-l\sqrt{m}/\log^2m+2,l\sqrt{m}/\log^2m-2]} \bar A^w_{m}(i)\subset \bar AB^w_{m}[-l\sqrt{m},l\sqrt{m})
\end{equation*}
and 
\begin{equation*}
    \bar AB^w_{m}[-l\sqrt{m}+4\log^2m,l\sqrt{m}-4\log^2m)\subset \bigcup_{i\in[-l\sqrt{m}/\log^2m+2,l\sqrt{m}/\log^2m-2]}\bar A^w_m(i).
\end{equation*}
Therefore if we consider $$S:=\{i\in\{-\lfloor l\sqrt{m}/\log^2m\rfloor,\dots,\lfloor l\sqrt{m}/\log^2m\rfloor\}:\ A^w_{m}(i)\subset AB^w_{m}[-l\sqrt{m},l\sqrt{m})\text{ and }\bar A^w_{m}(i+\frac{t}{\log^2m})\subset \bar AB^w_{m}[-l\sqrt{m},l\sqrt{m})\}$$
then
$$AB^w_{m}[-l\sqrt{m}+4\log^2 m+|t|,l\sqrt{m}-4\log^2 m-|t|)\subset \bigcup_{i\in S}A_{m}^w(i).$$
Indeed,
if $y\in AB^w_{m}[-l\sqrt{m}+4\log^2 m+|t|,l\sqrt{m}-4\log^2 m-|t|)$ then $y\in A_m^w(i)$ for some $i\in[-l\sqrt{m}/\log^2 m+2+|t|/\log^2 m,\ l\sqrt{m}/\log^2 m-2-|t|/\log^2 m]$. Thus both $i$ and $i+t/\log^2m$ are in $[-l\sqrt{m}/\log^2 m+2,\ l\sqrt{m}/\log^2 m-2]$. and by \eqref{32main1proof2} and \eqref{32main1proof2'} the inclusion follows.

Since $A$ is $(k,N)*$, remembering $t\le\epsilon^4\sqrt{m}$, observation \eqref{32mainfirst1} and Lemma \ref{31clt} give us
\begin{equation*}
    \begin{split}
        \frac{\mu_x^k(AB^w_{m}[-l\sqrt{m}+4\log^2 m+|t|,l\sqrt{m}-4\log^2 m-|t|))}{\mu_x^k(AB^w_{m}[-l\sqrt{m},l\sqrt{m})}\\
        =\frac{\mu_x^k(y\in[w_{k+m-s},\dots,w_{k+m}]\cap A:\ S_{m}\varphi(\sigma^k(y))\in[-l\sqrt{m}+4\log^2 m+|t|,l\sqrt{m}-4\log^2 m-|t|)}{\mu_x^k(y\in[w_{k+m-s},\dots,w_{k+m}]\cap A:\ S_{m}\varphi(\sigma^k(y))\in[-l\sqrt{m},l\sqrt{m})}\\
        =\frac{\mu_x^k(A)\mu_x^k(y\in[w_{k+m-s},\dots,w_{k+m}]\cap \Sigma_k^+(x):\ S_{m}\varphi(\sigma^k(y))\in[-l\sqrt{m}+4\log^2 m+|t|,l\sqrt{m}-4\log^2 m-|t|)}{\mu_x^k(A)\mu_x^k(y\in[w_{k+m-s},\dots,w_{k+m}]\cap \Sigma_k^+(x):\ S_{m}\varphi(\sigma^k(y))\in[-l\sqrt{m},l\sqrt{m})}\\
        =\frac{\mu_x^k(A)\mu_x^k([w_{k+m-s},\dots,w_{k+m}])}{\mu_x^k(A)\mu_x^k([w_{k+m-s},\dots,w_{k+m}])}\frac{\mathfrak{R}(l-\frac{4\log^2 m+|t|}{\sqrt{m}})-\mathfrak{R}(-l+\frac{4\log^2 m+|t|}{\sqrt{m}})+O(\log m/\sqrt{m})}{\mathfrak{R}(l)-\mathfrak{R}(-l)+O(\log m/\sqrt{m})}\\
        \ge(1-10\frac{\log^2 m+|t|}{\sqrt{m}})
    \end{split}
\end{equation*}
for sufficiently big $m$.
As a consequence we obtain
\begin{equation}\label{32main1proof4}
    \mu_x^k(AB^w_{m}[-l\sqrt{m},l\sqrt{m})\setminus\bigcup_{i\in S}A_{m}^w(i))\le10\frac{\log^2 m+|t|}{\sqrt{m}}\mu_x^k(AB^w_{m}[-l\sqrt{m},l\sqrt{m}))
\end{equation}
and similarly
\begin{equation}\label{32main1proof4'}
    \mu_{\bar x}^k(\bar AB^w_{m}[-l\sqrt{m},l\sqrt{m})\setminus\bigcup_{i\in S}\bar A_{m}^w(i+\frac{t}{\log^2m}))\le10\frac{\log^2 m+|t|}{\sqrt{m}}\mu_{\bar x}^k(\bar AB^w_{m}[-l\sqrt{m},l\sqrt{m}).
\end{equation}

Consider $G_i:=\{j\in\{-\lfloor l\sqrt{m}/\delta\rfloor+1,\dots,\lfloor l\sqrt{m}/\delta\rfloor-1\}: AB^w_{m}(j)\subset A^w_{m}(i)\text{ and }\bar AB^w_{m}(j+\frac{t}{\delta})\subset\bar A^w_{m}(i+\frac{t}{\log^2 m})\}$, where $i\in S$.
\\
CLAIM 2\\
For sufficiently big $n$
\begin{equation}\label{claim2}
    \mu_x^k(\bigcup_{j\in G_i}AB^w_{m}(j))\ge(1-\epsilon^3)\mu_x^k(A^w_{m}(i))
\end{equation}
and
\begin{equation}\label{claim2'}
     \mu_{\bar x}^k(\bigcup_{j\in G_i}\bar AB^w_{m}(j+\frac{t}{\delta}))\ge(1-\epsilon^3)\mu_{\bar x}^k(\bar A^w_{m}(i+\frac{t}{\log^2m}))
\end{equation}
Indeed, \eqref{32main1sfirstclaim1} and \eqref{32main1sfirstclaim1'} together with \eqref{32main1proof3} and \eqref{32main1proof3'} imply that
\begin{equation*}
    \begin{split}
        AB_m^w[i\log^2 m-\frac{\log^2m}{2}+\log^{1.5} m,\ i\log^2m +\frac{\log^2m}{2}-\log^{1.5} m)\subset A_m^w(i)\\
        \subset
        AB_m^w[i\log^2 m-\frac{\log^2m}{2}-\log^{1.5} m,\ i\log^2m +\frac{\log^2m}{2}+\log^{1.5} m)
    \end{split}
\end{equation*}
and 
\begin{equation*}
    \begin{split}
        \bar AB_m^w[i\log^2m+t-\frac{\log^2m}{2}+\log^{1.5} m,\ i\log^2m+t+\frac{\log^2m}{2}-\log^{1.5} m)\subset \bar A_m^w(i+\frac{t}{\log^2m})\\
        \subset
        \bar AB_m^w[i\log^2m+t-\frac{\log^2m}{2}-\log^{1.5} m,\ i\log^2m+t+\frac{\log^2m}{2}+\log^{1.5} m)
    \end{split}
\end{equation*}
thus by the first pair of inclusions, accounting for the diameters of $AB^w_{m}(j)$
$$AB_m^w[i\log^2 m-\frac{\log^2m}{2}+\log^{1.5} m+\delta,\ i\log^2m +\frac{\log^2m}{2}-\log^{1.5} m-\delta)\subset\bigcup_{j\in G_i}AB^w_{m}(j)$$
and
$$\bar AB_m^w[i\log^2 m+t-\frac{\log^2m}{2}+\log^{1.5} m+\delta,\ i\log^2m+t +\frac{\log^2m}{2}-\log^{1.5} m-\delta)\subset\bigcup_{j\in G_i}\bar AB^w_{m}(j+\frac{t}{\delta}).$$ 
Now we see that 
\begin{equation*}
    \begin{split}
        A_w^m(i)\setminus\bigcup_{j\in G_i}AB^w_{m}(j)\subset AB_m^w[i\log^2 m-\frac{\log^2m}{2}-\log^{1.5} m,\ i\log^2 m-\frac{\log^2m}{2}+\log^{1.5} m+\delta)\\
        \cup
       AB_m^w[i\log^2 +\frac{\log^2m}{2}-\log^{1.5} m-\delta,\ i\log^2 +\frac{\log^2m}{2}+\log^{1.5} m) 
    \end{split}
\end{equation*}
We cover the union of intervals appearing above with intervals of unit length \begin{equation*}
    \begin{split}
        [i\log^2 m-\frac{\log^2m}{2}-\log^{1.5} m+\delta,\ i\log^2 -\frac{\log^2m}{2}+\log^{1.5} m)\\\cup[i\log^2 +\frac{\log^2m}{2}-\log^{1.5} m-\delta,\ i\log^2 +\frac{\log^2m}{2}+\log^{1.5} m)
        \subset\bigcup_{j\in F}[j,j+1),
    \end{split}
\end{equation*} 
where  
$F\subset\{i\log^2 m-\frac{\log^2m}{2}-\log^{1.5} m,\dots,i\log^2 +\frac{\log^2m}{2}+\log^{1.5} m\}$,\ $|F|\le\log^{1.6}m$.
By Lemma \ref{31llt} and the same observation as in \eqref{32mainfirst1}
$$\mu_x^k(AB_m^w[j,j+1))=\frac{\mu_x^k(A)(\mu([w_{k+m-s},\dots,w_{k+m}])\mathfrak{n(\frac{j}{\sqrt{m}})}+o(1))}{\sqrt{m}}$$
so for any two $j,j'\in[i\log^2 m-\frac{\log^2m}{2}-\log^{1.5} m,\dots,i\log^2 +\frac{\log^2m}{2}+\log^{1.5} m]$
$$\frac{\mu_x^k(AB_m^w[j,j+1))}{\mu_x^k(AB_m^w[j',j'+1))}=\frac{\mathfrak{n}(\frac{j}{\sqrt{m}})}{\mathfrak{n}(\frac{j'}{\sqrt{m}})}+o(1)\le77$$ 
for sufficiently big $n$ depending only on $s$ (77 is arbitrary, greater than 1).
Hence if we take any disjoint family of intervals of unit length 
$$\bigsqcup_{j'\in F'}[j',j'+1)\subset[i\log^2 m-\frac{\log^2m}{2}+\log^{1.5} m,\ i\log^2 +\frac{\log^2m}{2}-\log^{1.5} m)$$
where $|F'|\ge\log^{1.9}m$, we can prove \eqref{claim2} by bounding
\begin{equation*}
    \begin{split}
        \frac{\mu_x^k( A_w^m(i)\setminus\bigcup_{j\in G_i}AB^w_{m}(j))}{\mu_x^k( A_w^m(i))}\le\frac{\mu_x^k( A_w^m(i)\setminus\bigcup_{j\in G_i}AB^w_{m}(j))}{\mu_x^k(AB_m^w[i\log^2 m-\frac{\log^2m}{2}+\log^{1.5} m,\ i\log^2 +\frac{\log^2m}{2}-\log^{1.5} m)}\\
        \le\frac{\sum_{j\in F}\mu_x^k(AB_m^w[j,j+1))}{\sum_{j'\in F'}\mu_x^k(AB_m^w[j',j'+1))}\le 77\frac{|F|}{|F'|}\le\epsilon^3
    \end{split}
\end{equation*}
and the proof of \eqref{claim2'} is analogous. This end the proof of CLAIM 2.

By CLAIM 2 we see that the first partition is up to $\epsilon^3$ a refinement of the second partition.
Now we find a map between the atoms $AB^w_m(\cdot),\ \bar AB^w_m(\cdot)$ of the first partition.
Fix $j\in S$.
Recall \eqref{32main1proof1} that
\begin{equation}\label{32main1proof5}
    \frac{\mu_{\bar x}^k(\bar AB^w_{m}(i+\frac{t}{\delta}))}{\mu_x^k(AB^w_{m}(i))}\in[(1-\epsilon^3)\frac{\mu_{\bar x}^k(\bar A)}{\mu_x^k(A)},(1+\epsilon^3)\frac{\mu_{\bar x}^k(\bar A)}{\mu_x^k(A)}]
\end{equation}
and by Lemma \ref{32knstar3} we have a map
$$\phi_{i}:A_{i}\rightarrow\bar AB^w_{m}(i+\frac{t}{\delta})$$
defined on a subset $A_{i}\subset AB^w_{m}(i),\ \mu_x^+(A_{i})=\min(\mu_x^+(AB^w_{m}(i)),\mu_{\bar x}^+(\bar AB^w_{m}(i+\frac{t}{\delta})))$  with the associated partition $\mathcal{P}_{i}$ into $(k+m,N)*$ subsets and $\bar A_{i}:=\phi_{i}(A_{i})$ with the associated partition $\mathcal{\bar P}_{i}$. By $\mu_x^+(A)=
\mu_{\bar x}^+(\bar A)$, Remark \ref{rem}, 2. and  \eqref{32main1proof5}
$$\frac{\mu_{\bar x}^+(\bar AB^w_{m}(i+\frac{t}{\delta}))}{\mu_x^+(AB^w_{m}(i))}\in[1-\epsilon^3,1+\epsilon^3]$$
so
\begin{equation*}
    \mu_x^+(A_{i})\ge (1-\epsilon^3)\mu_x^+(AB^w_{m}(i))
\end{equation*}
and 
\begin{equation*}
    \mu_{\bar x}^+(\bar A_{i})\ge (1-\epsilon^3)\mu_{\bar x}^+(AB^w_{m}(i+\frac{t}{\delta}))
\end{equation*}
thus by Remark \ref{rem}, 2.
\begin{equation}\label{32main1proof6}
    \mu_x^k(A_{i})\ge (1-\epsilon^3)\mu_x^k(AB^w_{m}(i))
\end{equation}
and 
\begin{equation}\label{32main1proof6'}
    \mu_{\bar x}^k(\bar A_{i})\ge (1-\epsilon^3)\mu_{\bar x}^k(AB^w_{m}(i+\frac{t}{\delta})).
\end{equation}

Gluing $\{\phi_{i}\}_{i\in G_j}$ and combining $\mathcal{P}_i,\ \mathcal{\bar P}_i$ we get a map
$$\phi^0_j:A^0_j\rightarrow \bar A^w_{m}(j+\frac{t}{\log^2m})\}$$
defined on a subset $A^0_j\subset A^w_{m}(j)$  with a partition $\mathcal{C}_0$ into $(k+m,N)*$ subsets and $\bar A_j^0:=\phi_j(A_j^0)$ with the associated partition $\mathcal{\bar C}_0$.
By \eqref{32main1proof6}
\begin{equation*}
    \mu_x^k(\bigcup_{i\in G_j}AB^w_{m}(i)\setminus A^0_j)\le \epsilon^3\mu_x^k(\bigcup_{i\in G_j}AB^w_{m}(i))
\end{equation*}
so by \eqref{claim2}
\begin{equation}\label{32main1proof7}
    \mu_x^k(A^0_j)\ge(1-3\epsilon^3)\mu_x^k(A^w_{m}(i))
\end{equation}

Now that we mapped the atoms from the first partition, leaving $3\epsilon^3$ proportion of points behind, we want to deal with the remaining points in the corresponding atoms from the second partition so as to obtain a loss of measure of order bounded by $\frac{|t|+\log^2m }{\sqrt{m}}$.

Now consider the sets 
$$A_{m}^w(j)\setminus A_j^0$$ and $$\bar A_{m}^w(j)\setminus \bar A^0_j$$
they are both unions of $(k+m,N)*$ sets so by Lemma \ref{32knstar3} we get a map
$$\phi^1_j:A^1_j:\rightarrow \bar A^w_{m}(j)\setminus\bar A_j^0$$
defined on a subset $A^1_j\subset A^w_{m}(j)\setminus A_j^0$  with a partition $\mathcal{C}_1$ into $(k+m,N)*$ subset and $\bar A^1_j:=\phi^1_j(A_j)$ with the associated partition $\mathcal{\bar C}_1$.
Then $\mu_x^+(A^0_j\sqcup A^1_j)=\mu^+_{\bar x}(\bar A^0_j\sqcup \bar A^1_j)=\min(\mu_x^+(A^w_{m}(j)),\mu_{\bar x}^+(\bar A^w_{m}(j+\frac{t}{\log^2m})))$.
Now we use the key property of $A^w_{m}(j)$ and $A^w_{m}(j+\frac{t}{\log^2m})$ along with the mean value theorem to obtain
\begin{equation}\label{32main1proof8}
    \begin{split}
        \frac{\frac{\mu_x^k(A^w_{m}(j))}{\mu_x^k(A\cap[w_{k+m-s},\dots,w_{k+m}])}}{\frac{\mu_{\bar x}^k(\bar A^w_{m}(j+\frac{t}{\log^2m}))}{\mu_x^k(\bar A\cap[w_{k+m-s},\dots,w_{k+m}])}}
    =\frac{\mathfrak{R}(\frac{(i+\frac{1}{2})\log^2m}{\sqrt{m}})-\mathfrak{R}(\frac{(i-\frac{1}{2})\log^2m}{\sqrt{m}})}{\mathfrak{R}(\frac{(i+\frac{1}{2})\log^2m+t}{\sqrt{m}})-\mathfrak{R}(\frac{(i-\frac{1}{2})\log^2m+t}{\sqrt{m}})}\in[1-c_l\frac{|t|+\log^2m}{\sqrt{m}},1+c_l\frac{|t|+\log^2m}{\sqrt{m}}],
    \end{split}
\end{equation}
where $c_l=\frac{sup_{\theta\in[-2l,2l]}\mathfrak{R}''(\theta)}{inf_{\gamma\in[-2l,2l]}\mathfrak{R}'(\gamma)}$. By the fact that $A,\bar A$ are $(k,N)*$ and Remark \ref{32rem}, Lemma \ref{3equi} tells us that (the rate of convergence is actually exponential in $m$, but we do not need that)
$$\frac{\mu_x^k(A\cap[w_{k+m-s},\dots,w_{k+m}])}{\mu_x^k(\bar A\cap[w_{k+m-s},\dots,w_{k+m}])}=\frac{\mu_x^k(A)\mu_x^k(\Sigma_k^+(x)\cap[w_{k+m-s},\dots,w_{k+m}])}{\mu_x^k(\bar A)\mu_x^k(\Sigma_k^+(x)\cap[w_{k+m-s},\dots,w_{k+m}])}\in[(1-\frac{1}{m})\frac{\mu_x^k(A)}{\mu_{\bar x}^k(\bar A)},(1+\frac{1}{m})\frac{\mu_x^k(A)}{\mu_{\bar x}^k(\bar A)}]$$
for sufficiently big $m$ (depending only on $s$),
so by \eqref{32main1proof8} and the above
we get 
$$\frac{\mu_x^k(A^w_{m}(j))}{\mu_{\bar x}^k(\bar A^w_{m}(j+\frac{t}{\log^2m}))}\in[(1-3c_l\frac{|t|+2\log^2m}{\sqrt{m}})\frac{\mu_x^k(A)}{\mu_{\bar x}^k(\bar A)},(1+3c_l\frac{|t|+2\log^2m}{\sqrt{m}})\frac{\mu_x^k(A)}{\mu_{\bar x}^k(\bar A)}].$$
By $\mu_x^+(A)=
\mu_{\bar x}^+(\bar A)$, Remark \ref{rem}, 2. and  the above
$$\frac{\mu_x^+(A^w_{m}(j))}{\mu_{\bar x}^+(\bar A^w_{m}(j+\frac{t}{\log^2m}))}\in[(1-3c_l\frac{|t|+2\log^2m}{\sqrt{m}}),(1+3c_l\frac{|t|+2\log^2m}{\sqrt{m}})],$$
thus
\begin{equation*}
    \mu_x^+(A_j^0\sqcup A^1_j)\ge(1-3c_l\frac{|t|+2\log^2m}{\sqrt{m}})\mu_x^+(A^w_{m}(j))
\end{equation*}
and
\begin{equation*}
    \mu_x^+(\bar A_j^0\sqcup \bar A^1_j)\ge(1-3c_l\frac{|t|+2\log^2m}{\sqrt{m}})\mu_x^+(\bar A^w_{m}(j)).
\end{equation*}
so by Remark \ref{rem} 1. we have
\begin{equation}\label{32main1proof9}
    \mu_x^k(A_j^0\sqcup A^1_j)\ge(1-\frac{C_l(|t|+\log^2m)}{\sqrt{m}})\mu_x^k(A^w_{m}(j))
\end{equation}
and
\begin{equation}\label{32main1proof9'}
    \mu_x^k(\bar A_j^0\sqcup \bar A^1_j)\ge(1-\frac{C_l(|t|+\log^2m)}{\sqrt{m}})\mu_x^k(\bar A^w_{m}(j)).
\end{equation}
for some constant $C_l$ depending only on $l$. Gluing $\phi^0_j$ and $\phi^1_j$ and combining partitions $\mathcal{C}_0$ with $\mathcal{C}_1$ and $\mathcal{\bar C}_0$ with $\mathcal{\bar C}_1$, we get $\mathcal{C}_j,\ \mathcal{\bar C}_j$,
$$\phi^2_j:A_j^0\sqcup A^1_j\rightarrow\bar A_j^0\sqcup \bar A^1_j$$
$\phi^2_j(\mathcal{C}_j)=\mathcal{\bar C}_j$.
\\
Gluing $\{\phi_j^2\}_{j\in S}$, combining partitions $\mathcal{C}_j,\ \mathcal{\bar C}_j$,
we get a measurable isomorphism
$$\phi_3: A_{good}\rightarrow\bar A_{good},$$ and invariant partitions into $(k+m,N)*$ sets $\mathcal{C}_3,\ \mathcal{\bar C}_3$,
where by \eqref{32main1proof9}
$$\mu_x^k(A_{good})\ge (1-\frac{C_l(|t|+\log^2m)}{\sqrt{m}})\mu_x^k(\bigcup_{i\in S}A_n(i))$$
so by \eqref{32main1proof4} there is some constant $\hat C_l>C_l$ such that
$$\mu_x^k(A_{good})\ge(1-\hat C_l\frac{|t|+\log^2 m}{\sqrt{m}})\mu_x^k(AB^w_{m}[-l\sqrt{m},l\sqrt{m})).$$
\\
Let $D^w_0:=A_{good},\ \mathcal{ P}_w:=\mathcal{ C}_3$ and $\bar D^w_0:= \bar A_{good},\ \mathcal{\bar P}_w:=\mathcal{\bar C}_3$\\
Let $D^w_1:=\bigsqcup_{j\in S}A^0_j$. Then by \eqref{32main1proof7}
$$\mu_x^k(D^w_1)\ge(1-3\epsilon^3)\mu_x^k(\bigcup_{i\in S}A^w_m(i)),$$
which by \eqref{32main1proof4} implies
$$\mu_x^k(D^w_1)\ge(1-\epsilon^2)\mu_x^k(AB^w_{m}[-l\sqrt{m},l\sqrt{m})).$$
\\
Define $\phi^t_w(y):=\phi_3(y)$ for $y\in D^w_0$. 
By the definition of $A_j^0$, $\phi_3$ and the set $S$, for $y\in D^w_0$ there is some $A^w_m(i)$ such that $y\in A^w_m(i)$ and $\phi^t_w(y)\in \bar A^w_m(i+\frac{t}{\log^2m})$, thus by \eqref{32main1sfirstclaim1} and \eqref{32main1sfirstclaim1'}
$$|S_{m}\varphi(\sigma^k(\phi_{w}^t(y)))-S_m\varphi(\sigma^k(y))-t|\le \log^3m $$
for $y\in D^w_0$. 
By the definition of $A_j^1$ and $\phi_3$, for $y\in D^w_1$ there is some $AB^w_m(i)$ such that $y\in AB^w_m(i)$ and $\phi^t(y)\in \bar AB^w_m(i+\frac{t}{\delta})$, thus by the definition of $AB^w_m(i)$ and $\bar AB^w_m(i+\frac{t}{\delta})$ 
$$|S_{m}\varphi(\sigma^k(\phi_{w}^t(y)))-S_m\varphi(\sigma^k(y))-t|\le 2\delta $$
for $y\in A_w$. 

Thus, such $D^w_0,\ \bar D^w_0,\ \mathcal{ P}_w,\ \mathcal{\bar P}_w $ and $D^w_1$ satisfy the requirements of our sublemma. 
\end{proof}
The above sublemma was all about the interval $[k,k+m]$ in the shift space and the range of Birkhoff sums in the interval $[-l\sqrt{m},l\sqrt{m})$.
Before constructing the final $\bar\phi_{A,\bar A}^{t,n}$ we need a simpler map which will serve the role of an identity map on the interval $[k+m,n]$ in the shift space.
\begin{sublemma}\label{32main2}
    We keep $x,\bar x$ from the original lemma and claim that $\exists n_3$ such that for $n'\ge n_3$, any $N,k'>0,\ w\in\Sigma$ and $A'\subset\Sigma_{k'}^+(y),\ \bar A'\subset\Sigma_{k'}^+(\bar y)$ which are $({k'},N)*$ with ${k'}+n'\le L(N)$, $y\in \Sigma^+(x)\cap[w_{k'-s},\dots,w_{k'}],\ \bar y\in \Sigma^+(\bar x)\cap[w_{k'-s},\dots,w_{k'}]$ and $\mu_x^+(A')=\mu_{\bar x}^+(\bar A')$, there are finite partitions $\mathcal{P}',\ \mathcal{\bar P}'$  of $A',\ \bar A'$ into $(({k'}+n'),N)*$ sets and a measurable isomorphism $\phi_{A',\bar A'}:A'\rightarrow\bar A'$ such that \\
    1. $\phi_{A',\bar A'}(\mathcal{P}')=\mathcal{\bar P}'$\\
    2. there is a set $$D'_0=D_0^{A',\bar A'}:=\bigcup_{\text{good}\ P'_i\in \mathcal{P}'}P'_i\text{\quad with\quad}\mu_y^{k'}(D'_0)\ge (1-\frac{\log^3n'}{\sqrt{n'}})\mu_y^{k'}(A')$$
    such that for any $z\in D'_0$
    \begin{equation}\label{32main2statement1}
        |S_{n'}\varphi(\sigma^{k'}(\phi_{A',\bar A'}(z)))-S_{n'}\varphi(\sigma^{k'}(z))|\le \log^2 n'
    \end{equation}
    3. there is a set $$D'_1=D_1^{A',\bar A'}:=\bigcup_{\text{very good}\ P'_i\in \mathcal{P}'}P'_i\text{\quad with\quad}\mu_y^{k'}(D'_1)\ge (1-\epsilon^2)\mu_y^{k'}(A')$$
    such that for any $z\in D'_1$
    and for all $q$ such that $ 1\le q\le n'$ we have
    \begin{equation}\label{32main2statement2}
        |S_q\varphi(\sigma^{k'}(\phi_{A',\bar A'}(z)))-S_q\varphi(\sigma^{k'}(z))|\le \delta\text{\quad and\quad}(\phi_{A',\bar A'}(z))_{{k'}+q}=z_{{k'}+q}
    \end{equation}
\end{sublemma}
\begin{proof}
We will write $k$ for $k'$ and $n$ for $n'$ only in this sublemma.\\
     Consider a similar partition to before
\begin{equation*}
    \begin{split}
        A'=A'B_{n}(-\infty,-\sqrt{n}\log n)\sqcup A'B_{n}[-\sqrt{n}\log n,-l\sqrt{n})\sqcup A'B_{n}[-l\sqrt{n},l\sqrt{n})\\
        \sqcup A'B_{n}[l\sqrt{n},\sqrt{n}\log n)\sqcup A'B_{n}[\sqrt{n}\log n,\infty)
    \end{split}
\end{equation*}
and
\begin{equation*}
    \begin{split}
        \bar A'=\bar A'B_{n}(-\infty,-\sqrt{n}\log n)\sqcup \bar A'B_{n}[-\sqrt{n}\log n,-l\sqrt{n})\sqcup \bar A'B_{n}[-l\sqrt{n},l\sqrt{n})\\
        \sqcup \bar A'B_{n}[l\sqrt{n},\sqrt{n}\log n)\sqcup \bar A'B_{n}[\sqrt{n}\log n,\infty)
    \end{split}
\end{equation*}
 
Using the same reasoning as earlier we can get for $i=1,2$ the subsets $A'_i\subset A'B_{n}[(-1)^il\sqrt{n},(-1)^i\sqrt{n}\log n)$, $\bar A'_i\subset \bar A'B_{n}[(-1)^il\sqrt{n},(-1)^i\sqrt{n}\log n)$, partitions $\mathcal{P}'_i$ and $\mathcal{\bar P}'_i$ of $A'_i$ and $\bar A_i$ into $(k+n,N)*$ subsets and a measurable isomorphism $\phi_i:A'_i\rightarrow\bar A'_i$ such that $\phi_i(\mathcal{P}'_i)=\mathcal{\bar P}'_i$. Moreover
\begin{equation}\label{32main2proof1}
    \mu_y^k(A'B_{n}[(-1)^il\sqrt{n},(-1)^i\sqrt{n}\log n)\setminus A'_i)\le \frac{C_4\log^{2}n}{\sqrt{n}}\mu_y^k(A')
\end{equation}
and $|S_{n}\varphi(\sigma^{k}(z))-S_{n}\varphi(\sigma^{k}(\phi_i(z)))|\le C_5\log n$ for $z\in A'_i$\\
and similarly
 \begin{equation}\label{32main2proof1'}
        \mu_{\bar y}^k(\bar A'B_{n}[(-1)^il\sqrt{n},(-1)^i\sqrt{n}\log n)\setminus\bar A'_i)\le \frac{C_4\log^{2}n}{\sqrt{n}}\mu_{\bar y}^k(\bar A').
    \end{equation}
 For $z\in A'_i$ we will define $\phi_{A',\bar A'}(z):=\phi_i(z)$ and then, just like before, we get \eqref{32main2statement1} for $z\in A'_i$.\\
Now we want to map $A'B_{n}[-l\sqrt{n},l\sqrt{n})$ into $\bar A'B_{n}[-l\sqrt{n},l\sqrt{n})$.

Since $A'$ is $(k,N)*$, it follows that we have a subsublemma analogous to Subsublemma \ref{32main1s} and CLAIM 1
\begin{subsublemma}
    For $k+n\le L(N)$ there exists a partition
$$A'=\bigsqcup_{i\in\mathbb{Z}}A'_{n}(i)$$
where $A'_{n}(i)\subset\Sigma^+_k(y)$ satisfy
$$\frac{\mu_y^k(z\in A'\setminus A'_{n}(i):\ S_{n}\varphi(\sigma^k(z))\le\inf_{w\in A'_{n}(i)}  S_{n}\varphi(\sigma^k(w)))}{\mu_y^k(A')}\ge\mathfrak{R}(\frac{(i-\frac{1}{2})\log^2n}{\sqrt{n}})$$
$$\frac{\mu_y^k(A'_{n}(i))}{\mu_y^k(A')}=\mathfrak{R}(\frac{(i+\frac{1}{2})\log^2 n}{\sqrt{n}})-\mathfrak{R}(\frac{(i-\frac{1}{2})\log^2n}{\sqrt{n}})$$
$$\frac{\mu_y^k(z\in A'\setminus A'_{n}(i):\ S_{n}\varphi(\sigma^k(z))\ge\sup_{w\in A'_{n}(i)}  S_{n}\varphi(\sigma^k(w))}{\mu_y^k(A')}\ge1-\mathfrak{R}(\frac{(i+\frac{1}{2})\log^2 n}{\sqrt{n}})$$
and for $i\le i'$ and $z\in A_{n}(i)$, $z'\in A_{n}(i')$
we have 
$$S_{n}\varphi(\sigma^k(z))\le S_{n}\varphi(\sigma^k(z'))$$
additionally each $A_{n}(i)$ is a union of $(k+n,N)*$ sets.
Analogously we get a partition
$$\bar A'=\bigsqcup_{i\in\mathbb{Z}}\bar A'_{n}(i).$$\\
For big enough $n$ and $i=-\lfloor\frac{l\sqrt{n}}{\log ^2n}\rfloor,-\lfloor\frac{l\sqrt{n}}{\log ^2n}\rfloor+1,\dots,\lfloor\frac{l\sqrt{n}}{\log ^2n}\rfloor$ we have 
\begin{equation}\label{32main2copyofclaim1}
    |S_{n}\varphi(\sigma^k(z))-i\log^2n|\le \frac{\log^2n}{2}+\log^{1.5}n
\end{equation}
and
\begin{equation}\label{32main2copyofclaim1'}
    |S_{n}\varphi(\sigma^k(\bar z))-i\log^2n|\le \frac{\log^2n}{2}+\log^{1.5}n
\end{equation}
if $z\in A'_{n}(i),\ \bar z\in \bar A'_{n}(i)$.
\end{subsublemma}

In the same way as \eqref{32main1proof3}, \eqref{32main1proof3'} we get  
\begin{equation}\label{32main2copyof1}
   \begin{split}
       A'B_n(i\log^2 n-\frac{\log^2n}{2}+\log^{1.5} n,\ i\log^2n +\frac{\log^2n}{2}-\log^{1.5} n)\subset A'_n(i)\\
    \subset
        A'B_n[i\log^2 n-\frac{\log^2n}{2}-\log^{1.5} n,\ i\log^2n +\frac{\log^2n}{2}+\log^{1.5} n)
   \end{split} 
\end{equation}
and 
\begin{equation}\label{32main2copyof1'}
   \begin{split}
       \bar A'B_n(i\log^2 n-\frac{\log^2n}{2}+\log^{1.5} n,\ i\log^2n +\frac{\log^2n}{2}-\log^{1.5} n)\subset \bar A'_n(i)\\
    \subset
        \bar A'B_n[i\log^2 n-\frac{\log^2n}{2}-\log^{1.5} n,\ i\log^2n +\frac{\log^2n}{2}+\log^{1.5} n)
   \end{split} 
\end{equation}
Recall that $s$ was taken to satisfy $C_02^{-s\beta}<\epsilon^3$, so that $C_0d(y,y')^\beta\le\epsilon^3$ for $y,\bar y$ satisfying the assumptions of our sublemma (Sublemma \ref{32main2}).
Then, for $z\in A'$ by the fact that $A,\bar A$ are $(k,N)*$, Remark \ref{32rem} and then Lemma \ref{2hol} we get
\begin{equation}\label{32main2proof2}
    \begin{split}
        \frac{\mu_{\bar y}^k(\bar A\cap[z_k,\dots,z_{k+n}])}{\mu_{ y}^k(A\cap[z_k,\dots,z_{k+n}])}=\frac{\mu_{\bar y}^k(\bar A)}{\mu_{ y}^k(A)}\frac{\mu_{\bar y}^k(\Sigma^+_k(\bar y)\cap[z_k,\dots,z_{k+n}])}{\mu_{ y}^k(\Sigma^+_k( y)\cap[z_k,\dots,z_{k+n}])}\\=\frac{\mu_{\bar y}^k(\bar A)}{\mu_{ y}^k(A)}\frac{\mu_{\bar y}^k(H^k_{\bar y,y}(\Sigma^+_k( y)\cap[z_k,\dots,z_{k+n}]))}{\mu_{ y}^k(\Sigma^+_k( y)\cap[z_k,\dots,z_{k+n}])}\in[(1-\epsilon^3)\frac{\mu_{\bar y}^k(\bar A)}{\mu_{ y}^k(A)},(1+\epsilon^3)\frac{\mu_{\bar y}^k(\bar A)}{\mu_{ y}^k(A)}].
    \end{split}
\end{equation}
We can find a measurable isomorphism $\phi_z:A'_z\rightarrow\bar A_z'$, where $A'_z\subset A'\cap[z_k,\dots,z_{k+n}]$ is $(k+n,N)*$, $\bar A'_z\subset \bar A'\cap[z_k,\dots,z_{k+n}]$ is $(k+n,N)*$ and 
    $\mu_x^+(A'_z)=\mu_{\bar x}^+(\bar A'_z)=\max(\mu_x^+(A'\cap[z_k,\dots,z_{k+n}]),\mu_{\bar x}^+(\bar A'\cap[z_k,\dots,z_{k+n}]))$.
    By $\mu_x^+(A)=
\mu_{\bar x}^+(\bar A)$, Remark \ref{rem}, 2. and \eqref{32main2proof2}
$$\frac{\mu_x^+(A'\cap[z_k,\dots,z_{k+n}])}{\mu_{\bar x}^+(\bar A'\cap[z_k,\dots,z_{k+n}])}\in[1-\epsilon^3,1+\epsilon^3]$$
so
\begin{equation*}
\mu_x^+(A'_z)\ge (1-\epsilon^3)\mu_x^+(A'\cap[z_k,\dots,z_{k+n}])
\end{equation*}
and
\begin{equation*}
    \mu_{\bar x}^+(\bar A_z')\ge(1-\epsilon^3)\mu_{\bar x}^+(\bar A'\cap[z_k,\dots,z_{k+n}]).
\end{equation*}
thus by Remark \ref{rem}, 1.
\begin{equation}\label{32main2proof3}
\mu_y^k(A'_z)\ge (1-\epsilon^3)\mu_{ y}^k(A'\cap[z_k,\dots,z_{k+n}])
\end{equation}
and
\begin{equation}\label{32main2proof3'}
    \mu_{\bar y}^k(\bar A_z')\ge(1-\epsilon^3)\mu_{\bar y}^k(\bar A'\cap[z_k,\dots,z_{k+n}]).
\end{equation}
Consider $G_i:=\{z\in A'_n(i)\quad \text{such that}\quad\bar A'\cap[z_k,\dots,z_{k+n}]\subset\bar A'_n(i)\}$, where $i\in\{-\lfloor\frac{l\sqrt{n}}{\log ^2n}\rfloor,\dots,\lfloor\frac{l\sqrt{n}}{\log ^2n}\rfloor\}$.
Then by \eqref{32main2copyof1} and \eqref{32main2copyof1'}
\begin{equation}\label{32main2proof4}
    A'B_n(i\log^2 m-\frac{\log^2m}{2}+\log^{1.5} m+\delta,\ i\log^2m +\frac{\log^2m}{2}-\log^{1.5} m-\delta)\subset G_i.
\end{equation}
Indeed by Lemma \ref{3geo} 
applied to $z\in A'B_n(i\log^2 m-\frac{\log^2m}{2}+\log^{1.5} m+\delta,\ i\log^2m +\frac{\log^2m}{2}-\log^{1.5} m-\delta)$ and any $\bar z\in \Sigma^+_k(\bar y)\cap[z_k,\dots,z_{k+n}]$ we have by our choice of $s$ and conditions on $y,\bar y$ that for $1\le q\le n$
\begin{equation}\label{32main2proof5}
    |S_{q}\varphi(\sigma^{k}(z))-S_{q}\varphi(\sigma^{k}(\bar z))|\le \delta.
\end{equation}
so $\bar z\in \bar A'B_n(i\log^2 n-\frac{\log^2n}{2}+\log^{1.5} n,\ i\log^2n +\frac{\log^2n}{2}-\log^{1.5} n)$, which by \eqref{32main2copyof1'} is contained in $\bar A'_n(i)$, proving \eqref{32main2proof4}.
Now \eqref{32main2copyof1} and \eqref{32main2proof4} combine into
\begin{equation*}
   \begin{split}
       A'B_n(i\log^2 n-\frac{\log^2n}{2}+\log^{1.5} n+\delta,\ i\log^2n +\frac{\log^2n}{2}-\log^{1.5} n-\delta)\subset G_i \subset A'_n(i)\\
    \subset
        A'B_n[i\log^2 n-\frac{\log^2n}{2}-\log^{1.5} n,\ i\log^2n +\frac{\log^2n}{2}+\log^{1.5} n).
   \end{split} 
\end{equation*}
By the same reasoning as in CLAIM 2 this implies for sufficiently big $n$ (depending only on $l$) that
\begin{equation}\label{32main2proof6}
    \mu_y^k(G_i)\ge(1-\epsilon^3)\mu_y^k(A'_n(i))
\end{equation}
and
\begin{equation}\label{32main2proof6'}
     \mu_{\bar y}^k(\bigcup_{z\in G_i}\bar A'_n(i)\cap[z_k,\dots,z_{k+n}])\ge(1-\epsilon^3)\mu_{\bar y}^k(\bar A'_n(i))
\end{equation}
Gluing $\{\phi_{z}\}_{z\in G_i}$ (when $z,z'\in G_i$ describe the same cylinders and $A\cap[z_k,\dots,z_{k+n}]=A\cap[z'_k,\dots,z'_{k+n}]$, we pick one of them arbitrarily) we get a map
$$\phi_i':A'_i\rightarrow \bar A'_n(i)$$
defined on a subset $A_i'\subset A'_{n}(i)$ consisting of those $A_z'$ which took part in gluing. We define a partition $\mathcal{D}_i$ into $(k+n,N)*$ subsets, whose atoms are just $A_z'$, and similarly we let $\bar A_i':=\phi_i'(A_i')$ with the associated partition $\mathcal{\bar D}_i$.
By \eqref{32main2proof3} and \eqref{32main2proof6}
\begin{equation}\label{32main2proof7}
    \mu_y^k(A_i')\ge (1-3\epsilon^3)\mu_{ y}^k(A'_{n}(i))
\end{equation}
and by \eqref{32main2proof3'} and \eqref{32main2proof6'}
$$\mu_{\bar y}^k(A_i')\ge (1-3\epsilon^3)\mu_{\bar y}^k(\bar A'_{n}(i)).$$
Now, since by definition 
$$\frac{\mu_{ y}^k(A'_{n}(i))}{\mu_{ y}^k(A')}=\frac{\mu_{\bar y}^k(\bar A'_{n}(i))}{\mu_{ \bar y}^k(\bar A')}$$ we can, using Lemma \ref{32knstar3}, construct a measurable isomorphism (by $\mu_{ x}^+(A')=\mu_{\bar x}^+(\bar A')$ and Remark \ref{rem}, 1. we have $\mu_{ x}^+(A'_{n}(i))=\mu_{\bar x}^+(\bar A'_n(i))$)
$$\phi_i:A'_n(i)\rightarrow \bar A'_n(i)$$
extending $\phi_i'$, with partitions $\mathcal{C}'_i,\ \mathcal{\bar C}_i'$ of $A'_n(i)$ and $\bar A'_n(i)$ containing $\mathcal{D}'_i,\ \mathcal{\bar D}_i'$. We take $D^i$, the set of very good control in $A_n'(i)$ to be the union of those atoms in $\mathcal{C}'_i$ which are from $\mathcal{D}'_i$, which by definition are just the sets $A_i'$ above
$$D^i:=A_i'$$
Consider $S:=\{i\in\{-\lfloor\frac{l\sqrt{n}}{\log^2 n}\rfloor,\dots,\lfloor \frac{l\sqrt{n}}{\log^2n}\rfloor\}:\ A_n(i)\subset A'B_{n}[-l\sqrt{n},l\sqrt{n})\text{ and }\bar A_n(i)\subset A'B_{n}[-l\sqrt{n},l\sqrt{n})\}$
then in the same way as \eqref{32main1proof4} and \eqref{32main1proof4'} (Indeed, this is a special case with $t=0$) we obtain
\begin{equation}\label{32main2proof8}
    \mu_y^k(A'B_{n}[-l\sqrt{n},l\sqrt{n})\setminus\bigcup_{i\in S}A_n(i))\le\frac{10\log^2n}{\sqrt{n}}\mu_y^k(A'B_{n}[-l\sqrt{n},l\sqrt{n}))
\end{equation}
and
\begin{equation}\label{32main2proof8'}
    \mu_{\bar y}^k(\bar A'B_{n}[-l\sqrt{n},l\sqrt{n})\setminus\bigcup_{i\in S}A_n(i))\le\frac{10\log^2n}{\sqrt{n}}\mu_{\bar y}^k(\bar A'B_{n}[-l\sqrt{n},l\sqrt{n}))
\end{equation}
Now 
gluing $\{\phi_{i}\}_{i\in S}$ we get a measurable isomorphism
$$\phi:A_0'\rightarrow \bar A'B_{n}[-l\sqrt{n},l\sqrt{n})$$
defined on a subset $A_0'\subset A'B_{n}[-l\sqrt{n},l\sqrt{n})$  with a partition $\mathcal{B}$ into $(k+n,N)*$ subset and $\bar A_0':=\phi_0'(A_0')$ with the associated partition $\mathcal{\bar B}$.
By \eqref{32main2proof8}
\begin{equation}\label{32main2proof9}
    \mu_y^k(A_0')\ge (1-10\log^2n/\sqrt{n})\mu_{ y}^k(A'B_{n}[-l\sqrt{n},l\sqrt{n}))
\end{equation}
Let $D_0':=A_0'\sqcup A_1'\sqcup A_2'$, where we recall that $A'_1,A'_2$ are from \eqref{32main2proof1}. Then by \eqref{32main2proof9} and \eqref{32main2proof1} 
\begin{equation*}
    \begin{split}
        \mu_y^k(D_0')\ge(1-10\log^2n/\sqrt{n})\mu_{ y}^k(A'B_{n}[-l\sqrt{n},l\sqrt{n}))+(1-C_4\log^2n/\sqrt{n})\mu_y^k(A'B_{n}[-\sqrt{n}\log n,-l\sqrt{n}))\\
        +(1-C_4\log^2n/\sqrt{n})\mu_y^k(A'B_{n}[l\sqrt{n},\sqrt{n}\log n))\ge(1-\log^3n/\sqrt{n})\mu_y^k(A')
    \end{split}
\end{equation*}
and by \eqref{32main2proof1} and \eqref{32main2copyofclaim1}, \eqref{32main2copyofclaim1'}
$$|S_{n}\varphi(\sigma^{k}(z))-S_{n}\varphi(\sigma^{k}(\phi(z)))|\le 2\log^2n$$
for $z\in D_0,$ giving \eqref{32main2statement1}.
Let $D_1'=\bigsqcup_{i\in S}D^i$.
By \eqref{32main2proof8} and \eqref{32main2proof7}
$$\mu_y^k(D_1')\ge(1-\epsilon^2)\mu_y^k(A')$$
and by \eqref{32main2proof5} for $1\le q\le n$ we have
$$|S_{q}\varphi(\sigma^{k}(z))-S_{q}\varphi(\sigma^{k}(\phi(z)))|\le\delta$$
for $z,\bar z\in D_1'$ and the symbols are copied by the definition of $\phi_z$, giving \eqref{32main2statement2}.
We can find refinements $\mathcal{P'}, \mathcal{\bar P'}$ of $\mathcal{B}$ and $\mathcal{\bar B}$ and consider any measurable isomorphism $\phi_{A',\bar A'}$ which is also defined between
\begin{equation}\label{32main2proof10}
    \begin{split}
        A'B_{n}[-l\sqrt{n},l\sqrt{n})\setminus A_0'\ \sqcup (A'B_{n}[-\sqrt{n}\log n,-l\sqrt{n})\setminus A'_1)\\
        \sqcup(A'B_{n}[l\sqrt{n},\sqrt{n}\log n)\setminus A'_2)\sqcup A'B_{n}(-\infty,-\sqrt{n}\log n)\sqcup A'B_{n}[\sqrt{n}\log n,\infty)
    \end{split}
\end{equation}and 
\begin{equation}\label{32main2proof10'}
    \begin{split}
        \bar A'B_{n}[-l\sqrt{n},l\sqrt{n})\setminus \bar A_0'\ \sqcup (\bar A'B_{n}[-l\sqrt{n},-\sqrt{n}\log n)\setminus \bar A'_1)\\
        \sqcup(\bar A'B_{n}[l\sqrt{n},\sqrt{n}\log n)\setminus \bar A'_2)\sqcup \bar A'B_{n}(-\infty,-\sqrt{n}\log n)\sqcup \bar A'B_{n}[\sqrt{n}\log n,\infty),
    \end{split}
\end{equation}
and which extends $\phi$. Such $D_0',\ D_1',\ \phi_{A,\bar A}$ satisfy the requirements of the sublemma.
\end{proof}
Recall that we have already defined $\bar\phi^{t,n}_{A,\bar A}(y)$ on $A_1\subset AB_{m}[-\sqrt{m}\log m,-l\sqrt{m})$ and $A_2\subset AB_{m}[l\sqrt{m}, \sqrt{m}\log m)$. We are now ready to fully construct $\bar\phi^{t,n}_{A,\bar A}$ and finish the proof of Lemma \ref{32main}.

We take $n$ big enough to ensure $n/\log n\ge n_2$ and $n-n/\log n\ge n_1$. Then we can use Sublemma \ref{32main1} applied to $m=m(n)$. Then we will take $k'$ from Sublemma \ref{32main2} to be $m(n)$ and take $n'$ from Sublemma \ref{32main2} to be equal to $n-n/\log n$ for our $n$ from Lemma \ref{32main}. To every pair of atoms $R,\ \bar R$ of partitions $\mathcal{P}_w,\ \mathcal{\bar P}_w$ from Sublemma \ref{32main1} (they are $(k+m,N)*$), such that $\phi^t_w(R)=\bar R$, we apply Sublemma \ref{32main2}, that is we take $A'=R$, $\bar A'=\bar R$, $k'=m,\ n'= n-m$. As a result for each $R$ we obtain a measurable isomorphism $\phi_{R,\bar R}:R\rightarrow\bar R$ and partitions $\mathcal{ P}_R,\mathcal{ \bar P}_R$ of $R$ and $\bar R$ into $(k+m+n-m,N)*$, that is $(k+n,N)*$, sets. Doing this for all atoms $R$ of all partitions $\mathcal{ P}_w$ and $\mathcal{\bar P}_w$ and taking the union of all $\mathcal{ P}_R, \mathcal{ \bar P}_R$, we obtain partitions $\mathcal{ P}_3,$ and $ \mathcal{ \bar P}_3$ of $F\subset AB_{m}[-l\sqrt{m},l\sqrt{m})$ and $\bar F\subset \bar AB_{m}[-l\sqrt{m},l\sqrt{m})$ into $(k+n,N)*$ sets, where
$$F:=\bigcup_{w\in\Sigma}D_0^w,\quad \bar F:=\bigcup_{w\in\Sigma}\bar D_0^w$$
For $y\in F$ let $\mathcal{P}_w(y)$ be the atom of $\mathcal{P}_w$ containing $y$.
Define $$\bar\phi^{t,n}_{A,\bar A}(y):=\phi_{\mathcal{P}_w(y),\mathcal{\bar P}_w(\phi^t_w(y))}(y)$$ for $y\in F$, where $\phi_{\mathcal{P}_w(y),\mathcal{\bar P}_w(\phi^t_w(y))}$ comes from Sublemma \ref{32main2}.
Then $\bar\phi^{t,n}_{A,\bar A}(\mathcal{P}_3)=\mathcal{\bar P}_3$.
This way we have $\bar\phi^{t,n}_{A,\bar A}$ defined for  $y\in A_1\sqcup A_2\sqcup F$ and we extend it to an arbitrary measurable isomorphism $\bar\phi^{t,n}_{A,\bar A}:A\rightarrow\bar A$ with $\mathcal{P},\ \mathcal{\bar P}$ being some $(k+n,N)*$ refinements of $\mathcal{P}_1\sqcup \mathcal{P}_2\sqcup \mathcal{P}_3 \sqcup \{A\setminus (A_1\sqcup A_2\sqcup F)\}$ such that $\bar\phi^{t,n}_{A,\bar A}(\mathcal{P})=\mathcal{\bar P}$ obtained using Lemma \ref{32knstar3}. Let $$A_3:=\bigcup_{w\in\Sigma}\{y\in F:\quad y\in D^{\mathcal{P}_w(y),\mathcal{\bar P}_w(\phi^t_w(y))}_0\}$$ be the set of words which fall into the sets of good control at both the stage of Sublemma \ref{32main1} and the stage of Sublemma \ref{32main2}. Finally, let $$D_0:=A_1\sqcup A_2\sqcup A_3$$ ($A_1,\ A_2$ from \eqref{32mainfirst13} and \eqref{32mainfirst14}) and let $$D_1:=\bigcup_{w\in\Sigma}\{ y\in D_1^w :\quad y\in D^{\mathcal{P}_w(y),\mathcal{\bar P}_w(\phi^t_w(y))}_1\}$$ be the set of words which fall into the sets of very good control at both the stage of Sublemma \ref{32main1} and the stage of Sublemma \ref{32main2}. By these definitions $D_0$ and $D_1$ indeed form the required sets of good and very good control respectively. Sublemma \ref{32main1} and Sublemma \ref{32main2}, \eqref{32mainfirst13}, \eqref{32mainfirst14} and \eqref{32mainfirst2} imply that
$$\mu_x^+(D_0)\ge \left(1-\log^{10}n\frac{t+\log^3 n}{\sqrt{n}}\right)\mu_x^+(A)$$
and, together with the choice of $l$, Sublemma \ref{32main1} and Sublemma \ref{32main2} imply that
$$\mu_x^+(D_1)\ge (1-\epsilon)\mu_x^+(A).$$
 \eqref{32main1statement1} together with \eqref{32main2statement1} imply \eqref{32mainstatement1}. \eqref{32main1statement2} together with \eqref{32main2statement2} and $\frac{n}{m}>(1-\epsilon^2)$ for big $n$ imply \eqref{32mainstatement2} and \eqref{32mainstatement3}.
\end{proof}

\begin{lemma}\label{32final}
    $\forall\epsilon\ \forall\delta\ \exists n_4$ such that for $n\ge n_4$, any $N>0,\ \sqrt{n}/\log^2 n\ge |t|$ and $A\subset\Sigma^+(x),\ \bar A\subset\Sigma^+(\bar x)$ which are $(k,N)*$ with $k+n\le L(N)$ there are finite partitions $\mathcal{P},\ \mathcal{\bar P}$  of $A,\ \bar A$ into $((k+n),N)*$ sets and a measurable isomorphism $\phi_{A,\bar A}^{t,n}:A\rightarrow\bar A$ such that \\
    1. $\phi_{A,\bar A}^{t,n}(\mathcal{P})=\mathcal{\bar P}$\\
    2. there is a set $$D:=\bigcup_{\text{good}\ P_i\in \mathcal{P}}P_i\text{\quad with\quad}\mu_x^+(D)\ge \left(1-\log^{100} n\frac{t+\log n}{\sqrt{n}}\right)\mu_x^+(A)$$
    such that for any $y\in D$
    \begin{equation}\label{32finalstatement1}
        |S_n\varphi(\sigma^k(\phi_{A,\bar A}^{t,n}(y)))-S_n\varphi(\sigma^k(y))-t|\le\log^{10} n
    \end{equation}
     and for $1-\epsilon^2$ proportion of $ 1\le r\le n$ we have
    \begin{equation}\label{32finalstatement2}
        |S_r\varphi(\sigma^k(\phi_{A,\bar A}^{t,n}(y)))-S_r\varphi(\sigma^k(y))-t|\le \delta\text{\quad and\quad}(\phi_{A,\bar A}^{t,n}(y))_{k+r}=y_{k+r}
    \end{equation}
\end{lemma}
\begin{proof}
    define $\phi_1:=\bar\phi_{A,\bar A}^{t,m}$ where $\bar\phi_{A,\bar A}^{t,m}$ is from Lemma \ref{32main} applied to $\epsilon'$ and $\delta/10$, $m=m(n)=\frac{\epsilon'^2n}{\log n}$ sufficiently large (if $t\le \sqrt{n}/\log^2n$ then for all $\epsilon'$ and sufficiently big $n$ we have $\sqrt{m}/\log m\ge t$ so Lemma \ref{32main} applies). We will show that for sufficiently small $\epsilon'$ and sufficiently large $m$ we can take a certain composition of $\frac{\log n}{\epsilon'^2}$ maps of this form which will satisfy the statement of the lemma. For $\phi_1$ let $\mathcal{P}_1$ be the associated partition into $(k+m,N)*$ sets, let $D^1_0$ and $D^1_1$ be the sets of good and very good control from 2. and 3. in Lemma \ref{32main}. Let
    $$t_1=t_1(y):=S_m\varphi(\sigma^k(\phi_1(y)))-S_m\varphi(\sigma^k(y))-t$$
    and notice that for any two words $y,\ y'$ in the same atom of $\mathcal{P}_1$ we have $t_1(y)=t_1(y')$ because $\varphi$ depends only on the past.
    Then define inductively 
    $$\phi_{i+1}(y):=\bar\phi_{\mathcal{P}_{i}(y),\mathcal{\bar P}_i(\phi_i(y))}^{t_i,m}(y),$$
    $\mathcal{P}_{i+1},\mathcal{\bar P}_{i+1}$ the refinements of $\mathcal{P}_{i},\mathcal{\bar P}_i$ which we obtain by considering every atom $\mathcal{P}_i(y)$ and taking the union of the associated partitions into $(k+(i+1)m,N)*$ sets,
    $$D^{i+1}_0:=\bigsqcup_{P\in \mathcal{P}_i}D_0^P$$ and $$D^{i+1}_1:=\bigsqcup_{P\in \mathcal{P}_i}D_1^P,$$
    where $D_0^P$ and $D_1^P$ are the sets of good and very good control inside $P$ from 2. and 3. in Lemma \ref{32main}.,
    $$t_{i+1}:=t_i-S_m\varphi(\sigma^{k+im}(\phi_{i+1}(y)))-S_m\varphi(\sigma^{k+im}(y)).$$
    Let $$D_{i,0}:=\bigcap_{k\le i}D^k_0$$
    then by \eqref{32mainstatement1} for $y\in D_{i,0}$ we get $|t_{i+1}|\le|t|+i\log^6n$ and so for every $i$
    $$\mu_x^+(D_{i,0})\ge \left(1-2i^2\log^{10} n\frac{|t|+\log^7 n}{\sqrt{n}}\right)\mu_x^+(A)$$ 
    For any $\epsilon'>0$ let $$D_{i,1}:=\{y\in A|\quad y\in D^h_1 \quad \text{for at least}\ 1-2\epsilon' \text{proportion of}\ 1\le h\le i \}$$
     define random variables $X_j$ on $(A,\mu_x^+(\cdot)/\mu_x^+(A))$ by $X_j=0$ if $y\in D_1^j$ and $X_j=1$ otherwise. Let $S_j:=X_1+\dots+X_j$
    then for sufficiently big $n$, by 3. in Lemma \ref{32main}, the sequence $(S_j-j\epsilon')_j$ is a supermartingale and therefore by Azuma's inequality $$\mathbb{P}(S_j-j\epsilon'\ge j\epsilon')\le e^{-j\epsilon'^2/2}.$$
    Thus $$\mu_x^+(D_{i,1})\ge (1-e^{-i\epsilon'^2/2})\mu_x^+(A)$$
    and 
    defining $\phi_{A,\bar A}^{t,n}:=\phi_i,\ \mathcal{P}=\mathcal{P}_i,\ \mathcal{\bar P}=\mathcal{\bar P}_i$ for $i=\log n/\epsilon'^2$, sufficiently small $\epsilon'$ and sufficiently big $n$, the statement is satisfied for $D:=D_{i,0}\cap D_{i,1}$. Indeed, by the definition of $D_{i,1}$, if $y\in D$ then for $1-\epsilon'$ proportion of $k\in \{1,\dots,\log n/\epsilon'^2\}$ and for $1-\epsilon'$ proportion of $1\le r\le j$
    we have
    \begin{equation*}
        \begin{split}
            |S_{jm+r}\varphi(\sigma^{k}(\phi_{A,\bar A}^{t,n}(y)))-S_{jm+r}\varphi(\sigma^{k}(y))-t|=\\|S_{r}\varphi(\sigma^{k+jm}(\phi_{A,\bar A}^{t,n}(y)))-S_{r}\varphi(\sigma^{k+jm}(y))+(\sum_{i=0}^{j-1} S_{m}\varphi(\sigma^{k+im}(\phi_{A,\bar A}^{t,n}(y)))-S_{m}\varphi(\sigma^{k+im}(y)))-t|
        \end{split}
    \end{equation*}
    which telescopes to 
    $$ |S_{r}\varphi(\sigma^{k+jm}(\phi_{A,\bar A}^{t,n}(y)))-S_{r}\varphi(\sigma^{k+jm}(y))-t_j|\le\delta$$
    where the bound comes from \eqref{32mainstatement3}, which also gives 
    $$(\phi_{A,\bar A}^{t,n}(y))_{k+mj+r}=y_{k+mj+r}$$
    proving \eqref{32finalstatement2}. We also obtain \eqref{32finalstatement1} because for $y\in D_{i,0}$ we get $|t_{j+1}|\le|t_j|+\log^6n\dots\le|t|+j\log^6n$
    and as we have seen, some $t_h$ above is smaller than $\delta$ meaning that $|t_{i+1}|\le\delta+\log^7n/\epsilon'^2\le \log^{10}n.$
\end{proof}

\section{Fiber dynamics}
Throughout the rest of the paper We use $K_t:=K_t^{\gamma,\alpha}$ and write $\gamma=\frac{1}{2}-\eta$ for some $\eta>0$. We use the identification $[0,1)\cong\mathbb{R}/\mathbb{Z}$.\\
Let $e_i:=1-\frac{i\eta}{10}$ for $i=1,\dots,k$, where $k$ is such that $e_k\in[\frac{2}{3}+\frac{2\eta}{10},\frac{2}{3}+\frac{\eta}{10})$.\\
Let
$$Z_0(N):=\bigcup_{t\in [-N\log^{250}N,N\log^{250}N]}K_t([-1/N\log^{300}N,1/N\log^{300}N]^f)$$
and for $i=1,\dots, k$ let

$$Z_i(N):=\bigcup_{t\in [-N^{e_i}\log^3N,N^{e_i}\log^3N]}K_t([-1/N^{e_i}\log^6N,1/N^{e_i}\log^6N]^f)$$
  
$$U(N):=\{z=(z_1,z_2)\in M|\quad \rVert n\alpha+z_1\rVert\ge1/N\log^3 N\quad \text{for}\quad  |n|\le N\log^2N\}$$

For any two points $z,\ \bar z\in M$ $z=(a,b),\ \bar z =(\bar a, \bar b)$ we 
define\\ 
$d_V(z,\bar z):=|b-\bar b|$\\
$d_H(z,\bar z):=\rVert a-\bar a\rVert$\\
$d(z,\bar z):=d_V(z,\bar z)+d_H(z,\bar z)$\\
Let $S_{\bar z,z}^+(N):=\{t\ge0:\quad |\bar a+N(\bar z,t)\alpha-a|\le1/N\  \text{and} \ |N(\bar z,t)|\  \text{is minimal}\}$\\
Let $S_{\bar z, z}^-(N):=\{t\le 0:\quad |\bar a+N(\bar z,t)\alpha-a|\le1/N\  \text{and} \ |N(\bar z,t)|\  \text{is minimal}\}$\\
Let $n_{\bar z, z}^{\pm}(N):=N(\bar z,t)$ for $t\in S_{\bar z, z}^{\pm}(N)$\\
Let $t_{\pm}:=\inf S_{\bar z, z}^{\pm}(N)+\min(f(\bar a+n_{\bar z, z}^{\pm}(N)\alpha),b)$\\
Let 
\begin{equation}\label{defdn}
   d_{N}(\bar z, z):=
    \begin{cases}
      t_+, & \text{if}\  t_+\le -t_- \\
      t_-  & \text{otherwise}
    \end{cases}
  \end{equation}
This way $d_N(\bar z, z)$ gives us the time $t$ of smallest absolute value, which is needed to ensure that for some $t'$ in the neighborhood of $t$ we have $d_H(K_{t'}(\bar z),z)\le1/N$ and $d_V(K_{t'}(\bar z),z)$ is as small as possible.

For $d_{N}(\bar z, z)\in\mathbb{R_{\pm}}$ let $n_{\bar z, z}(N)=n_{\bar z, z}^{\pm}(N)$.\\
For any $\epsilon>0$ if $N$ is sufficiently large, then $|f'|<\delta(\epsilon)N$ outside of $[-\frac{1}{N^{e_k}\log^6N}-\frac{1}{N},\frac{1}{N^{e_k}\log^6N}+\frac{1}{N}]$, where $\delta(\cdot)>0$ is a certain function (see Lemma \ref{51good}).
Therefore, if additionally $N>1/\delta(\epsilon)$ and $z\notin E_{\delta(\epsilon)}(N)\cup Z_k(N)$, then 
\begin{equation}\label{distances}
    |d_{N}(\bar z, z)|< \delta(\epsilon)\implies d(\bar z,z)< 2\delta(\epsilon)
\end{equation}
where
$$E_\delta(N):=\{z=(z_1,z_2)\in M|\quad z_2\ge f(z_1)-\delta\quad \text{or}\quad  z_2\le \delta\}.$$

We will need the following property of $d_{N}$:\\
assuming $z=(a,b),\ \bar z =(\bar a, \bar b)$ we have
\begin{equation}\label{dnchange}
|d_{N}(K_t(\bar z),K_t(z))|\le |f(\bar a+n_{\bar z, z}(N)\alpha)-f( a)|+|S_{\bar N(z,t)}f(\bar a+n_{\bar z, z}(N)\alpha)-S_{\bar N(z,t)}f( a)|+|d_{N}(\bar z,z)|
\end{equation}

\begin{lemma}\label{4ini}
     $\exists N_0\ \forall N\ge N_0,\ z,\ \bar z\in U(N)$ there exists $T_0$ with $|T_0|<N\log^6N$ such that 
    $d_H(K_{T_0}(\bar z),z)\le1/N$
\end{lemma}
\begin{proof}
    Let  $z=(a,b),\ \bar z =(\bar a, \bar b)$. Recall that $\frac{p_i}{q_i}$ are the canonical approximants of $\alpha$. Let $q_{n-1}\le N\le q_{n}$, then by \eqref{lancuch} we get $q_{n}\le N\log^2N$. Partition $[0,1)$ into $q_{n}$ subintervals $[\frac{i}{q_n},\frac{i+1}{q_n})$ and suppose that $a-\bar a\in [\frac{i}{q_n},\frac{i+1}{q_n})\ (\textrm{mod}\ 1)$ and without loss of generality that $\alpha\in[\frac{p_n}{q_n},\frac{p_n+1}{q_n})$. Then take $j\le q_n$ such that $jp_n\equiv i\ (\textrm {mod}\ q_n)$. This way for any $N$ there is $j\le N\log^2N$ such that $|\bar a +j\alpha-a|\le 1/N$, then by Denjoy-Koksma and the definition of $U(N)$: $S_{j}f(\bar a) \le N\log^5N$, so it is enough to take $T_0=S_{j}f(\bar a)-\bar b$.
\end{proof}
\begin{remark}\label{4init}
    The above implies that for sufficiently big $N$, $z,\bar z\notin U(N)$ we have $|d_N(\bar z,z)|\le N\log^7N$. Indeed, if $d_H(K_{T_0}(\bar z),z)\le1/N$, then it is enough to add to $T_0$ some $s$ with $|s|\le f(N\log^3N)\le N$ to get arbitrarily close to $z$ in the vertical direction.
\end{remark}
\begin{lemma}\label{4appro}
   $\exists N_0\in\mathbb{N}\ \forall N\ge N_0,\ z,\bar z\notin Z_0(N),\ h\ge0,\ |s_1|,|s_2|,|d_{N}(\bar z,z)|\le N\log^{150}N$ such that
   $$|s_2-s_1-d_{N}(\bar z,z)|\le h,$$
   we have
   $$|d_{N}(K_{s_2}(\bar z),K_{s_1}(z))|\le h+N^{\frac{1}{2}-\eta/3}$$
\end{lemma}
\begin{proof}
    Let $T:=d_{N}(\bar z, z)$ and let  $z=(a,b),\ \bar z =(\bar x, \bar b)$ and $\bar a=\bar x+n_{\bar z,z}(N)$. Notice that if $a+k_1\alpha\in[-\log N/N,\log N/N]$ and  $a+k_2\alpha\in[-\log N/N,\log N/N]$ then $|k_1-k_2|\ge N/\log^6N$, indeed if we write $q_n\le |k_1-k_2|\le q_{n+1}$, then $\rVert |k_1-k_2|\alpha\rVert\ge\frac{1}{q_{n+2}}\ge\frac{1}{|k_1-k_2|\log^5|k_1-k_2|}$ by our diophantine assumptions on $\alpha$. Thus for any $|k|\le N\log^{151}N$ the orbit
    $\{a+i\alpha\}_{i=0}^k$ visits the set $[-\log N/N,\log N/N]$ in at most $\log^{157}N$ times $k_1,\dots, k_j$. Moreover, if $a+m\alpha\notin[-\log N/N,\log N/N]$ then $\bar a+m\alpha\notin[-1/N,1/N]$, because $|\bar a-a|\le 1/N$. Hence, for $0\le k\le N\log^{151}N$ we can bound $|S_{k}f(\bar a)-S_{k}f(a)|$ by breaking the orbit into at most $\log^{157}N$ intervals where we can use the bound from Denjoy-Koksma on $S_{k_{i+1}}f'(\theta_i)-S_{k_{i}+1}f'(\theta_i)$ for some point $\theta_i\in[\bar a+(k_{i}+1)\alpha,a+(k_i+1)\alpha]$: 
    \begin{equation}\label{4appro1}
        S_{k_{i+1}}f'(\theta_i)-S_{k_i+1}f'(\theta_i)\le N\log^{154}N\cdot N^{\frac{1}{2}-\eta}.
    \end{equation}
    Assuming $0\le k\le N\log^{151}N$ and
    \begin{equation}\label{4appro2}
        \{a+i\alpha\}_{i=0}^k\cap[-1/N\log^{300}N,1/N\log^{300}N]=\emptyset\ \text{and}\  \{\bar a+i\alpha\}_{i=0}^k\cap[-1/N\log^{300}N,1/N\log^{300}N]=\emptyset
    \end{equation}
    and by the mean value theorem \eqref{4appro1} yields for any such $k$:
\begin{equation}\label{4appro3}
    \begin{split}
        |S_{k}f(\bar a)-S_{k}f(a)|\le\\
        |S_{k_1}f(\bar a)-S_{k_1}f(a)|+
        \dots+|S_{k}f(\bar a)-S_{k_j+1}f(\bar a)-(S_{k}f(a)-S_{k_j+1}f(a))|
        +\sum_{i=1}^{j}|f(\bar a+k_i\alpha)-f(a+k_i\alpha)|\le\\
        \log^{157}N(\rVert\bar a-a\rVert N\log^{154}N\cdot N^{\frac{1}{2}-\eta}+(N\log^{300} N)^{\frac{1}{2}-\eta})\le N^{\frac{1}{2}-\eta/2}
    \end{split}
\end{equation}
for sufficiently big $N$. Analogously \eqref{4appro3} holds for $-N\log^{151}N\le k\le 0$. Making $N$ possibly even larger we see that
    $|N(z,s_1)|\le |\frac{s_1}{\min f}|\le N\log^{151}N$. By the definition of $Z_0(N)$ and $|T|+|s_1|\le N\log^{151}N$ it follows that $K_l(z),K_l(\bar z)\notin [-1/N\log^{300}N,1/N\log^{300}N]^f$ for $|l|\le|T|+|s_1| $. Therefore $k=N(z,s_1)$ satisfies \eqref{4appro2} and we can use \eqref{4appro3} to obtain
    $$|S_{N(z,s_1)}f(\bar a)-S_{N(z,s_1)}f(a)|\le N^{\frac{1}{2}-\eta/2}.$$
 Now using observation \eqref{dnchange} and the above we have
    \begin{equation*}
        \begin{split}
            |d_{N}(K_{s_2}(\bar z),K_{s_1}(z))|
            =|d_{N}(K_{s_2-T}(K_{T}(\bar z)),K_{s_1}(z))|
             = |d_{N}(K_{s_2-s_1-T}(K_{s_1}(K_{T}(\bar z))),K_{s_1}(z))|  \\
            \le|s_2-s_1-T|+|f(\bar a)-f(a)|+ |S_{N(z,s_1)}f(\bar a)-S_{N(z,s_1)}f(a)|+d_{N}(K_{T}(\bar z),z)\\
            \le h+N^{\frac{1}{2}-\eta/2}+N^{\frac{1}{2}-\eta/2}+0\\
            \le h+N^{\frac{1}{2}-\eta/3}
        \end{split}
    \end{equation*}
    for sufficiently large $N$.
    
\end{proof}

The next lemma is about precise control of the orbits.

\begin{lemma}\label{4precel}
    $\forall \epsilon>0\ \exists N_0\in\mathbb{N}\ \forall N\ge N_0,\ z\notin Z_i(N),\bar z\in M\ h\ge0,\ |s_1|,|s_2|\le N^{e_i}\log^{3}N$ such that
   $$|s_2-s_1-d_{N}(\bar z,z)|\le h,$$
   we have
   $$|d_{N}(K_{s_2}(\bar z),K_{s_1}(z))|\le h+N^{(\frac{3}{2}-0.9\eta)e_i-1}$$
\end{lemma}
\begin{proof}
     Let $t:=d_{N}(\bar z, z)$ and let  $z=(a,b),\ \bar z =(\bar x, \bar b)$ and $\bar a=\bar x+n_{\bar z,z}(N)$. Similarly as before the key is to bound $|S_{N(z,s_1)}f(\bar a)-S_{N(z,s_1)}f(a)|$ using Denjoy-Koksma. Since $s_1\le N^{e_i}\log^3N$ and $ z\notin Z_i(N)$, we know for sufficiently big $N$ that $|N(z,s_1)|\le N^{e_i}\log^4 N$ and $a+k\alpha\notin [-1/N^{e_i}\log^6N,1/N^{e_i}\log^6N]$ for $|k|\le |N(z,s_1)|$ and since $\rVert a+k\alpha-(\bar a+k\alpha)\rVert\le1/N$, also $\bar a+k\alpha\notin [-1/N^{e_i}\log^7N,1/N^{e_i}\log^7N]$. Hence
    \begin{equation*}
        \begin{split}
            |d_{N}(K_{s_2}(\bar z),K_{s_1}(z))|=|d_{N}(K_{s_2-s_1-t}(K_{s_1}(K_t(\bar z))),K_{s_1}(z))|\\
            \le |s_2-s_1-t|+|S_{N(a,s_1)}f(\bar a)-S_{N(a,s_1)}f(a)|+|f(\bar a)-f(a)|
            \le h+|\bar a-a|(|S_{N(a,s_1)}f'(\theta_1)|+|f'(\theta_2)|)\\
            \le h+1/N(N^{e_i}\log^6 N\cdot (N^{e_i}\log^7)^{\frac{1}{2}-\eta}+2(N^{e_i}\log^7)^{\frac{3}{2}-\eta}+(N^{e_i}\log^7)^{\frac{3}{2}-\eta})\le h+N^{(\frac{3}{2}-0.9 \eta)e_i-1}
        \end{split}
    \end{equation*}
    for some $\theta_1,\theta_2$ between $a$ and $\bar a$ and sufficiently big $N$.
\end{proof}

\section{Establishing Bernoullicity}
In this section we specify a set of good atoms of $\bigvee_{i=0}^{\infty}T^i(\mathcal{P}\times\mathcal{Q})$. For any two such atoms $r, \bar r$ we construct a matching $\phi_N^{r, \bar r}$ between $r=\Sigma^+(x)\times\{z\}$ and $\bar r=\Sigma^+(\bar x)\times\{\bar z\}$ which makes $T^i(y,z)$ and $T^i(\phi_N^{r, \bar r}(y,z))$ stay  in the same atom of $\mathcal{P}\times\mathcal{Q}$ for an arbitrarily large proportion of times $i=1,\dots,n$ for a set of $y$ of arbitrarily large measure. Staying in the same atom of $\mathcal{P}$ on the $y$ coordinate will be an immediate consequence of Lemma \ref{32final}. \eqref{32finalstatement2} which will play a key role in our construction. The harder part is to make sure that the points on the fiber are in the same atom of $\mathcal{Q}$. We obtain it by first restricting to the set of good trajectories $(y,z)$ for which the $z$ coordinate of $T^i(y,z)$ stays separated from $\partial \mathcal{Q}$ by some fixed distance for an arbitrarily large proportion of times $i$ (property 5 from Lemma \ref{51good}). This way, using \eqref{distances} and 3 and 5 from Lemma \ref{51good}, it is enough to ensure that the $z$ coordinates are close with respect to $d_N$.
The first step of the construction is a map $\phi_0$ which uses Lemma \ref{32final} and Remark \ref{4init}. Then as our points might get separated we use Lemma \ref{32final} repeatedly on blocks of varying length so as to maximize the chance that our trajectory belongs to the set of good control from Lemma \ref{32final}. This is obtained by specifying different levels of closeness to the singularity and using a longer block of symbols when close to the singularity and a shorter block of symbols when further away.

Aside from the formula for $\phi_N^{r, \bar r}$, we put a description of how the lengths of the blocks are chosen for each $y\in \Sigma, z\in M, N\in \mathbb{N}$ under Lemma \ref{52good}, it is the partition $S(y,z,N)$ of $[N^2\log^{210}N+1,L(N)]$.

Let $e_i:=1-\frac{i\eta}{10}$ for $i=1,\dots,k$, where $k$ is such that $e_k\in[\frac{2}{3}+\frac{2\eta}{10},\frac{2}{3}+\frac{\eta}{10})$.\\
Let $l(N):=N^2\log^{210}N$\\
Let $L(N):=N^2\log^{210}N+N^2\log^{211}N$ be the time we look for in the definition of VWB.\\
    Let $H_0(N)=N^2\log^{211}N$\\
    Let $H_i(N)\in [N^{2 e_i}/\log^2N,2N^{2 e_i}/\log^2N]$ be the length of a block at level $i$. Additionally we choose these numbers so that $H_{i+1}(N)\mid H_{i}(N)$ for all $i$, this is just to ensure that the blocks at level $i+1$ form a refinement of the blocks at level $i$. \\
    Let $h_i(N):=\frac{H_{i-1}(N)}{H_{i}(N)}$ be the number of blocks at level $i$ which fit into a block at level $i-1$\\
    Let $R_i(N):=N^{2}\log^{211}N/H_i(N)$ be the number of blocks at level $i$.\\
    Let $l_i^j(N):= N^2\log^{210}N+jH_i(N)$ be the times at the end of each block at level $i$.\\
    Let $W_i(N):=\bigcup_{t\in [-N^{e_i}\log^3N,N^{e_i}\log^3N]}K_t(Z_i(N))$ for $i=1,\dots, k$\\
    Let $W_0(N):=\bigcup_{t\in [-N\log^{250}N,N\log^{250}N]}K_t(Z_0(N))$

\subsection{Good trajectories}
The following lemma will serve as a definition of a set of good trajectories $B(N)$
\begin{lemma}\label{51good}

    $\forall\epsilon>0\;\exists \delta(\epsilon)>0\;\exists N_0\:\forall N>N_0\;\exists B(N)\subset\Sigma\times M$ such that $\mu\times\nu(B(N))>1-\epsilon$ and for any $(x,z)\in B(N)$:\\
    \\
    1. for any $h,i,j,\ 1\le h\le L(N),\ 1\le i\le H_j(N),\ 1\le j\le k$
    $$|S_{i}\varphi(\sigma^h(x))|\le N^{e_j}$$
    and
     $$|S_{l(N)}\varphi(x)|\le N\log^{110}N$$
    2. for any $i=1,\dots, k$
    $$\#\{1\le j\le R_i(N):\ K_{S_{l^j_i(N)}\varphi(x)}(z)\notin W_i(N)\}>(1-\epsilon){R_i(N)}$$
    3. 
    $$\#\{1\le j\le L(N):\ K_{S_{j}\varphi(x)}(z)\notin W_k(N)\}>(1-\epsilon){L(N)}$$
    4.for $1\le i\le L(N)$\\ 
    $$K_{S_i\varphi(x)}(z)\notin Z_0(N)$$
    5.
    $$\#\{1\le i\le L(N)|\quad T^i(x,z)\notin\Sigma\times (V_{2\delta(\epsilon)}(\partial \mathcal{Q})\cup E_{\delta(\epsilon)})\}>(1-\epsilon)L(N)$$
    6.
    $$z\notin U(N)$$
\end{lemma}
\begin{proof}
    It is enough to ensure each of these properties for a set of arbitrarily large measure. The conditions 2,3,5,6 all follow from Markov's inequality which implies that for any sequence of sets $A_1,\dots, A_n$ with $\mu (A_i)\le \epsilon^2$ there exists $B$ with $\mu(B)\ge 1-\epsilon$ such that for $y\in B$ we have
    $\#\{1\le j\le n:\ y\in A_j\}\le\epsilon n$.
To ensure 1 notice that
    for sufficiently big $N$, any $x\in\Sigma$ and any $j=1,\dots, k$ there is a $D_j\subset\Sigma^+(x)$ with $\mu_x^+(D_j)\ge1-10\epsilon/\eta$ such that for any $y\in D_j,\ 1\le h\le L(N),\ 1\le i\le H_j(N)$ we have
    $|S_i\varphi(\sigma^h(y))|\le N^{e_j}$.
Indeed, by Lemma \ref{31kol} applied to $N=H_j(N)$  
     $$\mu_x^+(w\in\Sigma^+(x):\ \max_{1\le i\le H_j(N)}|S_i\varphi(\sigma^n(w))|\ge \frac{1}{2}N^{e_j})=O(\log^2N/N^{e_j})$$
     uniformly in $n,x$. Thus for sufficiently big $N$ taking intersections of the above sets for $n=l_j^i(N),\ i=0, 1,\dots,R_j(N)$
     yields a set of arbitrarily small measure as $\frac{R_j(N)\log^2 N}{N^{e_j}}\to0$ and we can take $D_j$ to be its complement. Then  $\bigcap_{j=1}^k D_j$ is of arbitrarily large measure and to get 1. it is enough to take $B(N)$ such that $B(N)\cap\Sigma^+(x)\times M\subset D_j\times M$.
    In the same way we can make sure that
    $$|S_{l(N)}\varphi(x)|\le N\log^{110}N$$
     
     Similarly, to ensure 4 notice that
$$\mu_x^+(w\in\Sigma^+(x):\ \max_{1\le i\le L(N)}|S_j\varphi(w)|\ge N\log^{250}N)=O(\log^2N/N^{e_j})$$ and intersect the complement of above set with the complement of $\Sigma^+(x)\times W_0(N)$ to obtain a set of arbitrarily large measure inside each $\Sigma^+(x)\times M$.
\end{proof}

As in Lemma \ref{2par}, let $\Gamma\subset\bigvee_{i=0}^{\infty}T^i(\mathcal{P}\times\mathcal{Q})$ be a set of atoms with $\mu\times\nu(\bigcup\Gamma)=1$ such that for any atom $r\in\Gamma$ there exist $x\in\Sigma$ and $z\in M$ such that 

$$r=\Sigma^+(x)\times\{z\}$$
Let $$\Gamma_1(N)=\{r=\Sigma^+(x)\times\{z\}\in\Gamma: (x,z)\in B(N),\ \mu_x^+(r\setminus B(N))<\epsilon^2\}$$
then $$\mu\times\nu(\bigcup\Gamma_1(N))>1-\epsilon.$$\\

 For any $\epsilon>0$ we will find a sufficiently big $N\in\mathbb{N}$, such that 
 for any two atoms $r, \bar r\in\Gamma_1(N)$
$$r=\Sigma^+(x)\times\{z\}\quad\text{and}\quad\bar r=\Sigma^+(\bar x)\times\{\bar z\}$$
 we will find a measure preserving $\phi_N^{(x,z),(\bar x, \bar z)}:\Sigma^+(x)\rightarrow \Sigma^+(\bar x)$ and $U\subset\Sigma^+(x),\ \mu^+_{x}(U)>1-\epsilon$ such that for $x\in U$:\\
 \begin{equation}\label{cond1}
     1/ L(N)\#\{1\le i\le L(N)\ |\ d(K_{S_i\varphi(x^-,x^+)}(z),K_{S_i\varphi(\bar x^-,\phi(x^+))}(\bar z))<2\delta(\epsilon)\}>1-\epsilon
 \end{equation}
\begin{equation}\label{cond2}
    1/ L(N)\#\{1\le i\le L(N)\ |\ x^+_i=\phi(x^+)_i\}>1-\epsilon
\end{equation}
\\
Then \eqref{cond1} and 5. from Lemma \ref{51good} imply that the fiber coordinates of $T^i(x,z)$ and $T^i(\phi(x),\bar z)$ are in the same $\mathcal{Q}$ atom for at least $1-2\epsilon$ proportion of $i$, so combined with \eqref{cond2} and the definition of $\Gamma_1(N)$ and its arbitrarily large total measure, we obtain that $\mathcal{P}\times\mathcal{Q}$ is a VWB partition for $T$.

   \subsection{Construction of the final matching}

    We now fix an $\epsilon>0$ and show that for sufficiently big $N$ we can define the matching $\phi_N^{r,\bar r}:\Sigma^+(x)\rightarrow \Sigma^+(\bar x)$
    for any two atoms $r, \bar r\in\Gamma_1(N)$
$$r=\Sigma^+(x)\times\{z\}\quad\text{and}\quad \bar r =\Sigma^+(\bar x)\times\{\bar z\}$$
which satisfies \eqref{cond1} and \eqref{cond2}. We first use Lemma \ref{32final}. to define 
    $\psi_{0}:\Sigma^+(x)\rightarrow \Sigma^+(\bar x)$ and  partitions $\mathcal{P}_{0}, \mathcal{\bar P}_0$ of $\Sigma^+(x)$ and $\Sigma^+(\bar x)$ by the formula 
    $$\psi_{0}(y)=\phi^{T_0,N^2\log^{210}N}_{\Sigma^+(x),\Sigma^+(\bar x)}(y)$$
    for $y\in\Sigma^+(x)$, where $T_0:=d_N(\bar z ,z)$ and $\phi^{t_0,N^2\log^{210}N}_{\Sigma^+(x),\Sigma^+(\bar x)}$ is from Lemma \ref{32final}. applied to $n=N^2\log^{210}N$ for sufficiently big $N$ and $(0,N)*$ sets $\Sigma^+(x),\Sigma^+(\bar x)$. This way we obtain partitions $\mathcal{P}_{0},\ \mathcal{\bar P}_{0}$ into $(N^2\log^{210}N,N)*$ subsets, and a set of good control, which we denote by $D_0$ and which for sufficiently big $N$ satisfies
    \begin{equation}\label{firstgoodcontrol}
        \mu_x^+(D_0)\ge1-\log^{100} (N^2\log^{210}N)\frac{|T_0|+\log (N^2\log^{210}N)}{\sqrt{(N^2\log^{210}N)}}\ge 1-\frac{1}{\log N},
    \end{equation}
    where we used Remark \ref{4init} applied to sufficiently big $N$ and $z$, $\bar z$ (which by the definition of $B(N)$, 6. do not belong to $U(N)$). 
    Then,
    denoting by $\mathcal{P}(y)$ the atom of partition $\mathcal{P}$ containing $y$ and 
    setting
    $$\bar z_0:=K_{S_{N^2\log^{210}N}\varphi(\psi_0(y))}(\bar z)$$
    $$z_0:=K_{S_{N^2\log^{210}N}\varphi(y)}(z),$$
    we will define     
 $$\phi_N^{r,\bar r}(y):=\hat\psi_{1,\bar z_0,z_0}^{N^2\log^{210}N,\mathcal P_0,\mathcal{\bar P}_0,\psi_0}(y)$$
    where for $q\in\mathbb{N}$, finite families $\mathcal{P},\mathcal{\bar P}$ of $(q,N)*$ subsets of $\Sigma^+(x),\Sigma^+(\bar x)$, any $w,\bar w\in M,$ and a measurable isomorphism $\psi:\bigcup \mathcal{P}\rightarrow\bigcup\mathcal{\bar P},$ such that $ \psi(\mathcal{P})=\mathcal{\bar P}$, for $i=1,\dots, k$ we define $\hat\psi_{i,\bar w,w}^{q,\mathcal{P},\mathcal{\bar P},\psi}$ along with refining families $i(\mathcal{P}),i(\mathcal{\bar P})$ and sets of good control $D^\mathcal{P}_i\subset\Sigma^+(x)$. The construction proceeds by recursion on $i$.
      The base case for the definition of $\hat\psi_{i,\bar w,w}^{q,\mathcal{P},\mathcal{\bar P},\psi}$ is $i=k$, where we let 
    $$\hat\psi_{k,\bar w,w}^{q,\mathcal{P},\mathcal{\bar P},\psi}(y):=\psi_{k,h_k}(y)$$
    $$k(\mathcal{P}):=P^{h_k},\quad k(\mathcal{\bar P}):=\mathcal{\bar P}^{h_k}$$
    Here for $j=1,\dots,h_k(N)$ we again define $\psi_{k,j}$ and $\mathcal{P}^{j},\mathcal{\bar P}^{j}$ inductively. The base case for this induction is $j=1$, where we set
    $$\psi_{k,1}(y):=\phi^{t_0,H_k(N)}_{P(y),\mathcal{\bar P}(\psi(y))}(y)$$
    \begin{equation}\label{constructionind0}
        t_0:=\begin{cases}
           d_N(\bar w,w)&\text{if}\  |d_N(\bar w,w)|\le N^{(\frac{3}{2}-\frac{\eta}{100})e_i-1}\\
           0 &\text{otherwise}
        \end{cases}
    \end{equation}
    and if $\mathcal{P},\mathcal{\bar P}$ are families of
    $(q,N)*$ subsets we let $\mathcal{P}^1$ and $\mathcal{\bar P}^1$ be their refinements into $(q+H_k(N),N)*$ subsets coming from Lemma \ref{32final}. for $\phi^{t,H_k(N)}_{P(y),\mathcal{\bar P}(\psi(y))}(y)$ applied to every $y$.
    
    Assume we have defined $\psi_{k,j}$ and $\mathcal{P}^{j},\mathcal{\bar P}^{j}$
    for some $1\le j\le h_k(N)-1$. Then we define
     $$w_j=K_{S_{jH_{k}(N)}\varphi(\sigma^q(y))}(w)$$
    $$\bar w_j=K_{S_{jH_{k}(N)}\varphi(\sigma^q(\psi_{k,j}(y)))}(\bar w)$$
  \begin{equation}\label{constructionindj}
      t_j(y):=\begin{cases}
          d_N(\bar w_j,w_j)&\text{if}\ |d_N(\bar w_j,w_j)|\le  N^{(\frac{3}{2}-\frac{\eta}{100})e_i-1}\\
          0&\text{otherwise}
      \end{cases}
  \end{equation}
 $$\psi_{k,j+1}:=\phi^{t_{j}(y),H_k(N)}_{\mathcal{P}_{j}(y),\mathcal{\bar P}_{j}(\psi_{k,j}(y))}(y)$$   
    along with refining families $\mathcal{P}^{j+1},\mathcal{\bar P}^{j+1}$ of $(q+(j+1)H_{k}(N),N)*$ subsets which we obtain from Lemma \ref{32final}. This ends the inductive definition of $\psi_{k,j}$ and $\mathcal{P}^{j},\mathcal{\bar P}^{j}$ and we obtain $\psi_{k,h_k}$ and $k(\mathcal{P})$ and $k(\mathcal{\bar P})$. Now we need to specify the set of good control $D_k^\mathcal{P}$
    
    For any $j$ and for any atom $A\in \mathcal{P}^j$ (letting $\mathcal{P}^0:=\mathcal{P}$) Lemma \ref{32final}. implies along with the existence of $\psi_{k,j+1}$ as defined above, that there exist sets $D^A\subset A$, on which we have good control, with $$\mu_x^+(D^A)\ge (1-\log^{100}( H_k(N))\frac{|t_j|+\log H_k(N)}{\sqrt{H_k(N)}})\mu_x^+(A)\ge(1-\frac{N^{\eta/99}}{\sqrt{H_k(N)}})\mu_x^+(A).$$
    We set
      \begin{equation}\label{constructionindgoodsetk}
    D^i_k:=\bigcup_{A\in \mathcal{P}^i}D^A,\qquad
    D_k:=\bigcap_{0\le i\le h_k-1}D_k^i,\qquad \text{then}\quad \mu_x^+(D_k)\ge (1-h_k(N)\frac{N^{\eta/99}}{\sqrt{H(N)}})\mu_x^+(\bigcup \mathcal{P})   
    \end{equation}
    Letting $D_k^\mathcal{P}:=D_k$ as defined above completes all the definitions at level $k$ (the base case of recursion where $i=k$). We can go back to the recursion now.
    
    Suppose that for any $q, w,\bar w, \mathcal{P},\mathcal{\bar P},\psi$ we have defined $\hat\psi_{i,\bar w,w}^{q,\mathcal{P},\mathcal{\bar P},\psi}(y)$ together with $i(P),\ i(\mathcal{\bar P})$ and $D_i$.    
    Now we fix the parameters and define $\hat\psi_{i-1,\bar w,w}^{q,\mathcal{P},\mathcal{\bar P},\psi}$ by
    $$\hat\psi_{i-1,\bar w,w}^{q,\mathcal{P},\mathcal{\bar P},\psi}(y):=\psi_{i-1,h_{i-1}(N)}(y)$$
    where for a fixed $i=2,\dots,k$, the map $\psi_{i-1,j+1}$ is defined inductively for $j=0,\dots,h_{i-1}(N)-1$ by the formula
    \begin{equation}\label{constructionrecmap}
   \psi_{i-1,j+1}(y):=
   \begin{cases}
   \phi^{t_j(y),H_{i-1}(N)}_{\mathcal{P}^{i-1}_{j,1}(y),\mathcal{\bar P}^{i-1}_{j,1}(\psi_{{i-1},j}(y))}(y), & \text{if}\ j=0\ \text{or}\  K_{S_{q+jH_{i-1}(N)}\varphi(y)}(w)\in W_{i-1}(N)\\
      \hat\psi_{i,w_j,\bar w_j}^{q+jH_{i-1}(N), \mathcal{P}^{i-1}_{j,2},\mathcal{\bar P}^{i-1}_{j,2},\psi_{{i-1},j}}(y)  & \text{otherwise}
    \end{cases}
  \end{equation}
    $$w_j=K_{S_{jH_{i-1}(N)}\varphi(\sigma^q(y))}(w)$$
    $$\bar w_j=K_{S_{jH_{i-1}(N)}\varphi(\sigma^q(\psi_{i,j}(y)))}(\bar w)$$
  \begin{equation}\label{constructionrectime}
      t_j(y):=\begin{cases}
          d_N(\bar w_j,w_j)&\text{if}\ |d_N(\bar w_j,w_j)|\le N^{(\frac{3}{2}-\frac{\eta}{100})e_i-1}\\
          0&\text{otherwise}
      \end{cases}
  \end{equation}
    along with refining families $\mathcal{P}^{i-1}_{j+1},\mathcal{\bar P}^{i-1}_{j+1}$ of $(q+(j+1)H_{i-1}(N),N)*$ subsets which we split into two groups
$$\mathcal{P}_{j+1}^{i-1}= \mathcal{P}_{j+1,1}^{i-1}\cup \mathcal{P}_{j+1,2}^{i-1}$$
    $$\mathcal{\bar P}_{j+1}^{i-1}= \mathcal{\bar P}_{j+1,1}^{i-1}\cup \mathcal{\bar P}_{j+1,2}^{i-1}$$
    For $j=0$ in the formula for $\psi_{i-1,j+1}$ we let
    $$\mathcal{P}^{i-1}_{0,1}:=\mathcal{P}$$
    $$\mathcal{\bar P}^{i-1}_{0,1}:=\mathcal{\bar P}.$$
    Next, $\mathcal{P}^{i-1}_1$ and $\mathcal{\bar P}_1^{i-1}$ are the refinements of $\mathcal{P}$ and $\mathcal{\bar P}$ obtained from Lemma \ref{32final}. by partitioning all the atoms $\mathcal{P}(y)$. For $j>0$ we divide the atoms (elements of the family $\mathcal{P}^{i-1}_{j}$) into two groups
    $$\mathcal{P}_j^{i-1}= \mathcal{P}_{j,1}^{i-1}\cup \mathcal{P}_{j,2}^{i-1}$$
    An atom $A:=\mathcal{P}^{i-1}_{j}(y)$ will either belong to $ \mathcal{P}_{j,1}^{i-1}$
    or $ \mathcal{P}_{j,2}^{i-1}$
    depending on which of the two conditions in the definition of $\psi_{i-1,j+1}$ does $y$, and thus all the elements of $A$ satisfy. Then
    $\mathcal{\bar P}_{j,1}^{i-1} $and $\mathcal{\bar P}_{j,2}^{i-1}$ are defined by $\psi_{i-1,j}(\mathcal{P}_{j,1}^{i-1})$ and
    $\psi_{i-1,j}(\mathcal{P}_{j,2}^{i-1})$ respectively.
    Having defined $ \mathcal{P}_{j,1}^{i-1}, \mathcal{P}_{j,2}^{i-1},\mathcal{\bar P}_{j,1}^{i-1}$ and $\mathcal{\bar  P}_{j,2}^{i-1}$,
    we define
    $$\mathcal{P}^{i-1}_{j+1}:=r(\mathcal{P}^{i-1}_{j,1})\cup i(\mathcal{P}^{i-1}_{j,2})$$
    and
    $$\mathcal{\bar P}^{i-1}_{j+1}:=r(\mathcal{\bar P}^{i-1}_{j,1})\cup i(\mathcal{\bar P}^{i-1}_{j,2})$$
    where $r(\mathcal{P}^{i-1}_{j,1})$ and $r(\mathcal{\bar P}^{i-1}_{j,1})$ are defined by considering each pair $A= \mathcal{P}^{i-1}_{j}(y)\in \mathcal{P}^{i-1}_{j,1},\ \bar A= \mathcal{\bar P}^{i-1}_{j}(\psi_{i-1,j}(y))\in\mathcal{\bar P}^{i-1}_{j,1}$ and letting $r(\mathcal{P}^{i-1}_{j,1})_{|A}$ and $r(\mathcal{\bar P}^{i-1}_{j,1})_{|\bar A}$ be the partitions of of $A,\bar A$ coming from Lemma \ref{32final}. for the maps $\phi^{t_j(y),H_{i-1}(N)}_{A,\bar A}$, 
    and
    where $i(P^{i-1}_{j,2}),i(\mathcal{\bar P}^{i-1}_{j,2})$ are the refining families associated with $\hat\psi_{i,z_j,\bar z_j}^{q+jH_{i-1}(N), P^{i-1}_{j,2},\mathcal{\bar P}^{i-1}_{j,2},\psi_{{i-1},j}}$.
(This way we let $\mathcal{P}_{j+1}^{i-1}$ be the union of the two refining families.)
    
    This ends the definition of $\hat\psi_{i-1,\bar z,z}^{q,\mathcal{P},\mathcal{\bar P},\psi}(y)$. The associated refining partitions $(i-1)(\mathcal{P}),(i-1)(\mathcal{\bar P})$ are defined by just setting
    $$(i-1)(\mathcal{P}):=\mathcal{P}^{i-1}_{h_{i-1}(N)}\quad\text{and}\quad (i-1)(\mathcal{\bar P}):=\mathcal{\bar P}^{i-1}_{h_{i-1}}(N).$$
    The set of good control $D^{\mathcal{P}}_{i-1}$ is an intersection 
    $$D^{\mathcal{P}}_{i-1}:=\bigcap_{1\le j\le h_{i-1}}D_{i-1}^{j}$$
    of the sets of good control for the maps $\psi_{i-1,j}$, which again, are the sums of the sets $D^A$ coming from Lemma \ref{32final}., for $A\in \mathcal{P}^{i-1}_{j-1,1}$, and the sets $D_i^{\mathcal{P}^{i-1}_{j-1,2}}$ defined in the previous recursive step for the map $\hat\psi_{i,w_j,\bar w_j}^{q+jH_{i-1}(N), \mathcal{P}^{i-1}_{j,2},\mathcal{\bar P}^{i-1}_{j,2},\psi_{{i-1},j}}$
    $$D_{i-1}^{j}=\bigcup_{A\in \mathcal{P}^{i-1}_{j-1,1}}D^A\quad\sqcup\quad D_i^{\mathcal{P}^{i-1}_{j-1,2}}$$
    This finishes our recursive definition of $\phi_N^{r,\bar r}$ and provides us with $D_1:=D^{\mathcal{P}_0}_1$ - a set of good control for the whole $\hat\psi_{1,\bar z_0,z_0}^{N^2\log^{210}N,\mathcal{P}_0,\mathcal{\bar P}_0,\psi_0}(y)$.

\begin{lemma}\label{52good}
 $\forall \epsilon'>0\ \exists N_0\in\mathbb{N}\ \forall N\ge N_0,\ \forall r,\ \bar r\in\Gamma_1(N),$ we have $\mu_x^+(D_1)\ge1-\epsilon'$
\end{lemma}
\begin{proof}
   We prove inductively that for any map $\hat\psi_{k-i,\bar w,w}^{q,\mathcal{P},\mathcal{\bar P},\psi}$ its set of good control $D_{k-i}^\mathcal{P}$ satisfies 
   $$\mu_x^+(D^\mathcal{P}_{k-n})\ge (1-\sum_{i=0}^{n}\frac{H_{k-n-1}(N)}{H_{k-i}(N)}\frac{N^{\frac{3}{2}e_{k-i}-1}}{\sqrt{H_{k-i}(N)}})\mu_x^+(\bigcup P)$$
   then the statement follows by taking $n=k-1$, taking sufficiently big $N_0$ and recalling that for the original $P_0$ $\bigcup P_0=\Sigma^+(x)$.
   We know \eqref{constructionindgoodsetk} that 
   $$\mu_x^+(D^\mathcal{P}_k)\ge (1-h_k(N)\frac{N^{\eta/99}}{\sqrt{H(N)}}) \mu_x^+(\bigcup P) $$
   For any family $P$ and any map at level $i-1$ we have
   $$D^\mathcal{P}_{i-1}=\bigcap_{1\le j\le h_{i-1}}D_{i-1}^{j}$$
   and if 
   $$\mathcal{P}_{j-1}^{i-1}=\mathcal{P}_{j-1,1}^{i-1}\sqcup \mathcal{P}_{j-1,2}^{i-1}$$ 
   then each $D_{i-1}^{j}$ is a disjoint union of 2 sets of good control
   $$D_{i-1,j}=A\sqcup B,$$
   $A\subset \mathcal{P}_{j-1,1}^{i-1}$ coming from Lemma \ref{32final}. with (by \eqref{constructionrectime} $|t_{j-1}|\le N^{(\frac{3}{2}-\frac{\eta}{100})e_i-1}$)
   $$\mu_x^+(A)\ge (1-\frac{N^{\frac{3}{2}e_{i-1}-1}}{\sqrt{H_{i-1}(N)}})\mu_x^+(\bigcup P_{j-1,1}^{i-1})$$
   for sufficiently big $N$
   and $B=D_{i}^{\mathcal{P}_{j-1,2}^{i-1}}\subset \mathcal{P}_{j-1,2}^{i-1}$ with
   $$\mu_x^+(B)\ge (1-\sum_{m=0}^{k-i}\frac{H_{i-1}(N)}{H_{k-m}(N)}\frac{N^{\frac{3}{2}e_{k-m}-1}}{\sqrt{H_{k-m}(N)}})\mu_x^+(\bigcup P_{j-1,2}^{i-1})$$
   by induction hypothesis. Therefore
   \begin{equation*}
       \begin{split}
           \mu_x^+(D^P_{i-1})\ge (1-h_{i-1}(N)(\sum_{m=0}^{k-i}\frac{H_{i-1}(N)}{H_{k-m}(N)}\frac{N^{\frac{3}{2}e_{k-m}-1}}{\sqrt{H_{k-m}(N)}}+\frac{N^{\frac{3}{2}e_{i-1}-1}}{\sqrt{H_{i-1}(N)}}))\mu_x^+(\bigcup P)\\
           =(1-\sum_{m=0}^{k-i+1}\frac{H_{i-2}(N)}{H_{k-m}(N)}\frac{N^{\frac{3}{2}e_{k-m}-1}}{\sqrt{H_{k-m}(N)}})\mu_x^+(\bigcup P),
       \end{split}
   \end{equation*}
   which is the same as what we wanted.
   
\end{proof}
    We would actually like all $t_j(y)$ in the definitions of $\psi_{i,j}$ to be defined by $d_{N}$ for a set of $y\in \Sigma^+(x)$ of arbitrarily large measure. We will prove that the trajectories defined by $D_0, D_1$ and $\phi_N^{r,\bar r}(y)$ have this property. We have to use the fact that $z,\bar z\in B(N)$ here. The sets $D_0$ and $D_1$ depend on $N,r,\bar r$ and we will write $D_0(N)$ and $D_1(N)$ for $D_0$ and $D_1$ when we want to emphasize that. 
    
    Notice that given $N,r,\bar r$ the recursive definition of $\phi_N^{r,\bar r}$ determines for each $y\in\Sigma^+(x)$ a sequence of $k$ refining partitions $S_1,\dots, S_{k}$ of the interval $\mathbb{N}\cap[N^2\log^{210}N+1,L(N)]$ along with $k$ sequences $\{z^i_j\}_{j=0}^{k_i}$ of points on the orbit of $z$. $S_1$ is just a partition into subsequent intervals $[l_1^j(N)+1,l_1^{j+1}(N)]$, $j=0,\dots R_1(N)-1$ of the same length ($H_1(N)$), which we list as $(I^1_1,I^1_2,\dots,I^1_{R_1(N)})$. The first sequence of points is given by $z^1_j=K_{S_{l_1^j(N)}\varphi(y)}(z)$. Now we let $S_{i+1}$ be the refinement obtained from $S_i=(I^i_1,I^i_2,\dots,I^i_{a_i})$ by further partitioning those intervals $I^i_m=[l_i^j(N)+1,l_i^{j+1}(N)]$, $m>1$ (of length $H_i(N)$) for which the preceding interval $I^i_{m-1}$ is the interval $[l_i^{j-1}(N)+1,l_i^{j}(N)]$ (of the same length $H_i(N)$)
    and $z^i_{m-1}\notin W_i(N)$. All such $I^i_m$ get replaced by $h_{i+1}(N)$ subsequent intervals of the form $[l_{i+1}^e(N)+1,l_{i+1}^{e+1}(N)]$ for the appropriate integers $e$. The sequence $\{z^{i+1}_j\}_{j=1}^{a_{i+1}}$ is then chosen to record the position of $z$ at the times given by the ends of the intervals from $S_{i+1}$, that is if $S_{i+1}=\{I^{i+1}_1,\dots,I^{i+1}_{a_{i+1}}\}$ and $I^{i+1}_j=[a,b]$, then $z^{i+1}_j=K_{S_{b}\varphi(y)}(z)$. Similarly we define $\bar z^{i+1}_j=K_{S_{b}\varphi(\phi_N^{r,\bar r}(y))}(\bar z)$.
    Finally we let $S(y,z,N):=S_k$. Recall that during our construction of $\phi_N^{r,\bar r}$ we defined
     $\bar z_0:=K_{S_{N^2\log^{210}N}\varphi(\psi_0(y))}(\bar z),\ 
    z_0:=K_{S_{N^2\log^{210}N}\varphi(y)}(z)$.
    We now also let $z_j:=z^k_j$ and $\bar z_j:=\bar z^k_j$ for $j=1,\dots, a_k$.
    
    For $y\in D_0$, by the definition of $D_0$ we obtain from Lemma \ref{32final}. that
    for $z,\bar z\notin Z_0(N)$ and sufficiently big $N$
\begin{equation}\label{goodcontrolatfirstend}
        S_{l(N)}\varphi(\phi_N^{r,\bar r}(y))- S_{l(N)}\varphi(y)-d_N(\bar z, z)\le\log^{10}N.
    \end{equation}
    For $y\in D_1$ and the partition $S(y,z,N)=\{I_1,\dots,I_m\}$ where $I_i=[a_i+1,b_i]$, $b_i-a_i=H_j(N)$, if $d_N(\bar z_{i-1}, z_{i-1})\le N^{(\frac{3}{2}-\frac{\eta}{100})e_j-1}$, then by \eqref{constructionrecmap}, \eqref{constructionrectime} at levels $<k$ and by \eqref{constructionindj} at level $k$ we have
    \begin{equation}\label{goodcontrolatends}
        S_{b_i-a_i}\varphi(\sigma^{a_i}(\phi_N^{r,\bar r}(y)))- S_{b_i-a_i}\varphi(\sigma^{a_i}(y))-d_N(\bar z_{i-1}, z_{i-1})\le\log^{10}N
    \end{equation}
    and 
     \begin{equation*}
        S_{j-a_i}\varphi(\sigma^{a_i}(\phi_N^{r,\bar r}(y)))- S_{j-a_i}\varphi(\sigma^{a_i}y)-d_N(\bar z_{i-1}, z_{i-1})\le\delta\quad\text{and}\quad\phi_N^{r,\bar r}(y)_j=y_j
    \end{equation*}
    for $1-\epsilon^2 $ of $j$ in $[a_i,b_i]$.\\
    For $(y,z)\in B(N)$ we obtain an important property implied by the exceptional definition of $\psi_{i,j}$ in \eqref{constructionrecmap} when $j=0$, which is what makes our construction with different levels of closeness to the singularity useful. Namely, we know that if $I_m=[l_i^j(N)+1,l_i^{j+1}(N)]$, for some $I_m\in S$, then
    \begin{equation}\label{westayoutsideofbadsetswhenatlowerlevel}
        z_{m-1}\notin Z_{i-1}(N).
    \end{equation} Indeed, for $i=1$ this is 4. from Lemma \ref{51good}. If $i>1$, then $I_m\subset I^{i-1}_n $ for some $ I^{i-1}_n=[l_{i-1}^h(N)+1,l_{i-1}^{h+1}(N)]$ from $S^{i-1}$, such that $K_{S_{l_{i-1}^h(N)}\varphi(y)}(z)=z^{i-1}_{n-1}\notin W_{i-1}(N)$. Property 1. in Lemma \ref{51good}. then implies that $S_{l_i^j(N)-l_{i-1}^h(N)}\varphi(y)\le N^{e_{i-1}}$ and therefore
    $$z_{m-1}=K_{S_{l_i^j(N)-l_{i-1}^h(N)}\varphi(y)}(z^{i-1}_{n-1})\notin Z_{i-1}(N)$$
    
\begin{lemma}\label{52mapworks}
 $\exists N_0\in\mathbb{N}\ \forall N\ge N_0,\ \forall r,\ \bar r\in\Gamma_1(N),$ such that $r=\Sigma^+(x)\times\{z\},\bar r=\Sigma^+(\bar x)\times\{\bar z\}$ $\forall y\in D_{1}(N)\cap D_0(N)$ such that $(y,z)\in B(N)$ and for all $I_{n+1}=[l_{i}^h(N)+1,l_{i}^{h+1}(N)]$ from $S(y,z,N)$, for some $i,h,n$ and $\bar z_n, z_n$ defined above (appearing in the definition of $\phi_N^{r,\bar r}(y)$), we have $|d_N(\bar z_n,z_n)|\le N^{(\frac{3}{2}-\frac{\eta}{100})e_{i}-1}$
\end{lemma}
\begin{proof}
    For $n=0$ by construction $I_{n+1}=[l_1^1(N)+1,l_1^2(N)]$, so by \eqref{goodcontrolatfirstend} and Lemma \ref{4appro}. applied to $s_1= S_{l(N)}\varphi(\phi_N^{r,\bar r}(y))$, $s_2=S_{l(N)}\varphi(y)$ and $z,\bar z$, which by Lemma \ref{51good} satisfy the assumptions we get 
    $$d_N(\bar z_0, z_0)=d_N(K_{S_{l(N)}\varphi(\phi_N^{r,\bar r}(y))}(\bar z),K_{S_{l(N)}\varphi(y)}(z))\le\log^{10}N+N^{\frac{1}{2}-\frac{\eta}{3}}\le N^{(\frac{3}{2}-\frac{\eta}{100})e_1-1}$$
    for sufficiently big $N$.
    Now assuming it for some $n>0$, we want to proceed by induction. Suppose that
    $I_{n+1}=[l_{i}^h(N)+1,l_{i}^{h+1}(N)]$ and $I_{n+2}=[l_{m}^j(N)+1,l_{m}^{j+1}(N)]$, where by construction $l_{i}^{h+1}(N)=l_{m}^j(N)$. Then also by construction we know that $m\le i+1$.
    By the inductive assumption if $I_{n+1}=[l_{i}^h(N)+1,l_{i}^{h+1}(N)]$, then $|d_N(\bar z_n,z_n)|\le N^{(\frac{3}{2}-\frac{\eta}{100})e_i-1}$, thus by \eqref{goodcontrolatends} (the definition of $\phi_N^{r,\bar r}(y)$)
    $$ S_{H_i(N)}\varphi(\sigma^{l_i^{h}(N)}(\phi_N^{r,\bar r}(y)))- S_{H_i(N)}\varphi(\sigma^{l_i^{h}(N)}(y))-d_N(\bar z_{n}, z_{n})\le\log^{10}N.$$
    By \eqref{westayoutsideofbadsetswhenatlowerlevel} $z_{n}\notin Z_{i-1}(N)$. If $m>0$, we can use the above combined with Lemma \ref{4precel}, as well as Lemma \ref{51good}, 1.
    Then Lemma \ref{4precel}. applied to  $s_2= S_{H_i(N)}\varphi(\sigma^{l_i^{h}(N)}(\phi_N^{r,\bar r}(y)))$, $s_1=S_{H_i(N)}\varphi(\sigma^{l_i^{h}(N)}(y))$ and $\bar z_{n}$ and $z_{n}$, which by \eqref{westayoutsideofbadsetswhenatlowerlevel} is outside of $Z_{i-1}(N)$, gives us
      $$d_N(\bar z_{n+1}, z_{n+1})=d_N(K_{S_{H_i(N)}\varphi(\sigma^{l_i^{h}(N)}(\phi_N^{r,\bar r}(y)))}(\bar z_n),K_{S_{H_i(N)}\varphi(\sigma^{l_i^{h}(N)}(y))}(z_n))\le\log^{10}N+N^{(\frac{3}{2}-0.9\eta)e_{i-1}-1}\le N^{(\frac{3}{2}-\frac{\eta}{100})e_m-1},$$
      where the last inequality follows from $m\le i+1$ so ($e_m\ge e_{i-1}+\frac{2\eta}{10}$).
      If $m=1$ then Lemma \ref{4appro} and Lemma \ref{51good}, 1,4, imply
      $$d_N(\bar z_{n+1}, z_{n+1})=d_N(K_{S_{H_i(N)}\varphi(\sigma^{l_i^{h}(N)}(\phi_N^{r,\bar r}(y)))}(\bar z_n),K_{S_{H_i(N)}\varphi(\sigma^{l_i^{h}(N)}(y))}(z_n))\le\log^{10}N+N^{\frac{1}{2}-\frac{\eta}{3}}\le N^{(\frac{3}{2}-\frac{\eta}{100})e_1-1}$$
      This ends the induction.
\end{proof} 

\begin{lemma}\label{52precel}
    $\forall \epsilon',\delta>0\ \exists N_0\in\mathbb{N}\ \forall N\ge N_0,\ \forall r,\ \bar r\in\Gamma_1(N)$ such that $r=\Sigma^+(x)\times\{z\},\bar r=\Sigma^+(\bar x)\times\{\bar z\}$ $\forall y\in D_{1}(N)\cap D_0(N)$ such that $(y,z)\in B(N)$, we have $$|d_N(K_{S_j\varphi(\phi_N^{r,\bar r}(y))}(\bar z),K_{S_j\varphi(y)}(z))|\le \delta\quad\text{and}\quad y_j=(\phi_N^{r,\bar r}(y))_j \quad  \text{for $(1-\epsilon')$ proportion of}\  j\in[1, L(N)]$$
\end{lemma}
\begin{proof}
    We will first prove by induction on $i$ that for any $\epsilon>0$ and sufficiently big $N$, in the partition $S_i(y,z,N)$ the intervals at level $i$, that is of the form $I=[l_i^j(N)+1,l_i^j]$, take up at least $1-\epsilon$ proportion of $[N^2\log^{210}N+1,L(N)]$. For $i=1$ they take up the whole $[N^2\log^{210}N+1,L(N)]$. Assume we have this property for $i<k,\ \epsilon/3$ and sufficiently big $N$.
    Now by construction the intervals at level $i$ in $S_i(y,z,N)$ come in blocks of $h_i(N)$ adjacent intervals of the form $[l_i^j+1,l_i^{j+1}(N)]$, therefore for sufficiently big $N$ the union of the first intervals in each of the blocks takes up at most $1/h_i(N)\le\epsilon/3$ proportion of $[N^2\log^{210}N+1,L(N)]$. By Lemma \ref{51good}, 2. we also know that for sufficiently big $N$  
     $$\#\{1\le j\le R_i(N):\ z^i_j\notin W_i(N)\}>(1-\epsilon/3){R_i(N)}$$
     Therefore, by construction, in the refining partition $S_{i+1}(y,z,N)$, the intervals at level $i+1$ replace at least $1-2\epsilon/3$ proportion of intervals at level $i$ in $S_{i}(y,z,N)$. Since these comprise at least $1-\epsilon/3$ proportion of $[N^2\log^{210}N+1,L(N)]$, the induction is completed. Notice also that 
     $$\frac{L(N)-N^2\log^{210}N}{L(N)}\longrightarrow1$$
     Combining the above with \eqref{westayoutsideofbadsetswhenatlowerlevel} and choosing big enough $N$ we see that
     for at least $1-\frac{\epsilon'}{2}$ of $j\in [1,L(N)]$ we have $j\in I_n$, where $I_n\in S(y,z,N)$
     is of the form $[l_k^h+1,l_k^{h+1}]$ for some $h$ and $z_{n-1}\notin Z_{k-1}(N)$.
     By Lemma \ref{52mapworks} we can use \eqref{goodcontrolatends} for such $j$
     to get  
     $$\phi_N^{r,\bar r}(y)_j=y_j$$
     and
     $$|S_{j-l_k^{h}(N)}\varphi(\sigma^{l_k^{h}(N)}(\phi_N^{r,\bar r}(y)))-S_{j-l_k^{h}(N)}\varphi(\sigma^{l_k^{h}(N)}(y))-d_N(\bar z_{n-1},z_{n-1})|\le \frac{\delta}{2}$$
        for $1-\epsilon'$ proportion of $j\in [1,L(N)]$
     and then Lemma \ref{4precel} applied to $s_1=S_{j-l_k^{h}(N)}\varphi(\sigma^{l_k^{h}(N)}(y))$ and $s_2=S_{j-l_k^{h}(N)}\varphi(\sigma^{l_k^{h}(N)}(\phi_N^{r,\bar r}(y)))$ (which satisfy $|s_1|,|s_2|\le N^{e_k}\log^{3}N$ by Lemma \ref{51good}, 1.), and $\bar z_{n-1},z_{n-1}$ to obtain that for such $j$ we have
     $$d_N(K_{S_j\varphi(\phi_N^{r,\bar r}(y))}(\bar z),K_{S_j\varphi(y)}(z))=d_N(K_{S_{j-l_k^{h}(N)}\varphi(\sigma^{l_k^{h}(N)}(\phi_N^{r,\bar r}(y)))}(\bar z_{n-1}),K_{S_{j-l_k^{h}(N)}\varphi(\sigma^{l_k^{h}(N)}(y))}(z_{n-1}))\le\frac{\delta}{2}+N^{(\frac{3}{2}-0.9\eta)e_{k}-1}\le\delta$$

\end{proof}
 By \eqref{firstgoodcontrol}, Lemma \ref{52good} and the definition of $\Gamma_1(N)$ we see that
 $\forall \epsilon>0\ \exists N_0\in\mathbb{N}\ \forall N\ge N_0,\ \forall r,\bar r\in\Gamma_1(N)$ such that $r=\Sigma^+(x)\times\{z\}$, we have 
 $$\mu_x^+(\{y\in D_{1}(N)\cap D_0(N)\quad \text{such that}\quad (y,z)\in B(N)\})\ge 1-\epsilon$$
Making the $\epsilon'$ in Lemma \ref{52precel} sufficiently small and using 3. and 5. from Lemma \ref{51good}, \eqref{distances} implies 
$$|d(K_{S_j\varphi(\phi_N^{r,\bar r}(y))}(\bar z),K_{S_j\varphi(y)}(z))|\le 2\delta(\epsilon)$$ for $1-\epsilon$ proportion of $j\in[1, L(N)]$, for sufficiently large $N$ (independent of $r$ and $\bar r$). Thus, $U\subset\Sigma^+(x)$ equal to the set above
$$U:=\{y\in D_{1}(N)\cap D_0(N)\quad \text{such that}\quad (y,z)\in B(N)\}$$
satisfies \eqref{cond1} and \eqref{cond2} and therefore the partition $\mathcal{P}\times\mathcal{Q}$ is VWB.


\newpage

\begin{thebibliography}{12}
\bibitem{1} R. Bowen. \textit{Markov partitions for Axiom A diffeomorphisms}, Amer. J. Math.
 92, (1970), 725–747.
\bibitem{2} V. Climenhaga \textit{Gibbs measures have local product structure}, Preprint, arXiv:2310.17495v2.

\bibitem{5} D. Dolgopyat, C. Dong, A. Kanigowski, and P. Nandori. \textit{Mixing properties of generalized} $T,T^{-1}$ \textit{transformations}. Israel J. Math., 247(1):21–73, 2022.
\bibitem{6}D. Dolgopyat, A. Kanigowski, and F. Rodriguez Hertz. \textit{Exponential
mixing implies Bernoulli}. Ann. of Math. (2), 199(3):1225–1292, 2024.
\bibitem{7} C.Dong and A. Kanigowski \textit{Bernoulli property for certain skew products over hyperbolic systems}. Trans. Amer. Math. Soc. 375 (2022), 1607-1628.
\bibitem{kal} Steven Arthur Kalikow. \textit{$T,\ T^{-1}$ transformation is not loosely Bernoulli}. Ann. of Math.
(2), 115(2):393–409, 1982.
\bibitem{8}  A. Kanigowski \textit{Bernoulli property for homogeneous systems}, arXiv:1812.03209.
\bibitem{9}  A. Kanigowski, F. Rodriguez Hertz and K. Vinhage. \textit{On the non-equivalence of the Bernoulli and K properties in dimension four}. Journal of Modern Dynamics, 2018, 13: 221-250.
\bibitem{kat} A. Katok. \textit{Smooth non-Bernoulli K-automorphisms}. Invent. Math., 61(4):291–299,
1980.
\bibitem{11} Y. Katznelson. \textit{Ergodic automorphisms of $\mathbb{T}^n$ are Bernoulli shifts}, Israel J. Math. 10 (1971) 186-195.

\bibitem{12} A. V. Kochergin, \textit{Mixing in special flows over a shifting of segments
and in smooth flows on surfaces}, Mat. Sb., 96 138 (1975): 471-502.

\bibitem{16} D. S. Ornstein, B. Weiss. \textit{Geodesic flows are Bernoullian}, Israel J. Math. 14 (1973) 184–198.

\bibitem{17} W. Parry and M. Pollicott. \textit{Zeta functions and the periodic orbit structure
of hyperbolic dynamics}.
Astérisque, tome 187-188, 1990.
\bibitem{18} V. A. Rohlin, \textit{On the fundamental ideas of measure theory}. Mat. Sb. 25
(1949): 107-150; Amer. Math. Soc. Translations, ser. 1, vol. 10 (1962):
1-54.

\bibitem{rud} D. Rudolph. \textit{Asymptotically Brownian skew products give non-loosely Bernoulli
K-automorphisms}. Invent. Math., 91(1):105–128, 1988.
\bibitem{rud2} D. Rudolph, Classifying the isometric extensions of Bernoulli shifts, J.
d’Analyse Math. 34 (1978), 36–60.
\bibitem{20}  P. C. Shields, \textit{Weak and very weak Bernoulli partitions}, Monatshefte für Mathematik (1977) 84: 133.
\bibitem{21}  P. C. Shields, \textit{The theory of Bernoulli shifts}, Web Edition 1.01 Reprint of 1973 University of Chicago Press edition
\bibitem{22} Y. G. Sinai. \textit{Gibbs measures in ergodic theory}. Russ. Math. Surveys 27 (1972),
 21–70.
\end{thebibliography}
\end{document}